\documentclass[11pt]{amsart}   	

\usepackage[margin=1in]{geometry}

\usepackage{amsmath,amssymb, amsfonts,amsthm,stackrel}
\usepackage[hide links]{hyperref}
\usepackage{mathtools}
\usepackage{tikz-cd}
\usetikzlibrary{decorations.markings}
\usepackage{makecell}
\setcellgapes{4pt}
\usepackage{multicol}

\usepackage[shortlabels]{enumitem}

\newtheorem{theorem}{Theorem}[section]
\newtheorem{lemma}[theorem]{Lemma}
\newtheorem{corollary}[theorem]{Corollary}
\newtheorem{proposition}[theorem]{Proposition}
\newtheorem{definition-proposition}[theorem]{Definition-Proposition}

\theoremstyle{definition}
\newtheorem{definition}[theorem]{Definition}
\newtheorem{example}[theorem]{Example}

\newtheorem{remark}[theorem]{Remark}

\DeclareMathOperator{\add}{\mathsf{add}}
\DeclareMathOperator{\Filt}{\mathsf{Filt}}
\DeclareMathOperator{\Gen}{\mathsf{Gen}}

\DeclareMathOperator{\Cogen}{\mathsf{Cogen}}

\DeclareMathOperator{\FiltGen}{\mathsf{T}}

\DeclareMathOperator{\FiltCogen}{\mathsf{F}}
\DeclareMathOperator{\WL}{\mathsf{W_L}}
\DeclareMathOperator{\WR}{\mathsf{W_R}}
\newcommand{\J}{\mathsf J}
\newcommand{\Jinv}{\J^\mathsf{d}}
\newcommand{\T}{\mathcal T}
\newcommand{\F}{\mathcal F}
\newcommand{\W}{\mathcal W}

\DeclareMathOperator{\stt}{\mathsf{s}\tau\text{-}\mathsf{tilt}}
\DeclareMathOperator{\mods}{\mathsf{mod}}
\DeclareMathOperator{\tex}{\mathsf{c}\text{-}\tau\text{-}\mathsf{ex}}

\DeclareMathOperator{\Hom}{\mathrm{Hom}}
\DeclareMathOperator{\Ext}{\mathrm{Ext}}

\DeclareMathOperator{\End}{\mathrm{End}}
\DeclareMathOperator{\ind}{\mathrm{ind}}

\renewcommand{\P}{\mathcal{P}}
\DeclareMathOperator{\Ps}{\mathcal{P}_\mathrm{s}}
\DeclareMathOperator{\Pns}{\mathcal{P}_\mathrm{ns}}
\newcommand{\Pmod}{P}
\newcommand{\Pmods}{P_\mathrm{s}}
\newcommand{\Pmodns}{P_\mathrm{ns}}
\DeclareMathOperator{\I}{\mathcal{I}}
\DeclareMathOperator{\Is}{\mathcal{I}_\mathrm{s}}
\DeclareMathOperator{\Ins}{\mathcal{I}_\mathrm{ns}}
\newcommand{\Imod}{I}
\newcommand{\Imods}{I_\mathrm{s}}
\newcommand{\Imodns}{I_\mathrm{ns}}

\newcommand{\Ebm}{\mathsf E}
\newcommand{\Fbm}{\mathsf{E}^{-1}}
\newcommand{\tfo}{\omega}

\DeclareMathOperator{\Hasse}{\mathrm{Hasse}}

\newcommand{\derD}{\mathcal{D}^b}

\newcommand{\SC}{\mathcal S}
\newcommand{\C}{\mathcal C}

\DeclareMathOperator{\rk}{\mathsf{rk}}

\usepackage{soul}

\title{Mutation of $\tau$-exceptional pairs and sequences}
\author{Aslak B. Buan}
\address{Department of Mathematical Sciences, Norwegian University of Science and Technology (NTNU), 7491 Trondheim, NORWAY}
\email{aslak.buan@ntnu.no}
\author{Eric J. Hanson}
\address{Department of Mathematics, North Carolina State University, Raleigh, NC 27695, USA}
\email{ehansonmath@gmail.com}
\author{Bethany R. Marsh}
\address{Visiting Researcher at the School of Mathematics, University of Leeds, Leeds LS2 9JT, U.K.}
\email{B.R.Marsh@leeds.ac.uk}

\thanks{The authors would like to thank Francesca Fedele and 
Emine Y{\i}ld{\i}r{\i}m for making us aware of an error in a proof of Lemma \ref{lem:twogens} in an earlier version.
In the spring semester of 2023, all authors were Fellows at the Centre for Advanced Study of the Norwegian Academy of Science and Letters, where a major part of the work for this paper was carried out.
ABB was supported by 
 grant number FRINAT 301375 from the Norwegian Research Council. 
 ABB would like to thank BRM and the School of Mathematics at the University of Leeds for their warm hospitality. EJH was supported by the Canada Research Chairs Program CRC-2021-00120 and by NSERC Discovery Grants RGPIN/04465-2019 and RGPIN-2022-03960. A portion of this work was carried out while EJH was a postdoc at l'Université de Sherbrooke (Sherbrooke, QC, Canada).
 This work was supported by the Engineering and Physical Sciences Research Council [grant number EP/W007509/1].}


\begin{document}
	
	\maketitle
\begin{abstract}
We introduce a notion of mutation for $\tau$-exceptional sequences of modules over arbitrary finite-dimensional algebras. For hereditary algebras, we show that this coincides with the classical mutation of exceptional sequences. For rank two algebras, we show that mutation of $\tau$-exceptional sequences is transitive if and only if mutation of support $\tau$-tilting modules in the sense of Adachi-Iyama-Reiten is transitive.
\end{abstract}
	
	\tableofcontents
 
	\section*{Introduction}

 Exceptional sequences in the context of derived categories were introduced by Gorodentsev, Rudakov and others in the Moscow school of algebraic geometry in the late eighties; see~\cite{goru, rud}. They introduced the notion of {\em mutation}, an operation on such sequences. Bondal~\cite{bon} considered mutation of exceptional sequences in more general triangulated categories, and proved that these operations provide a braid group action.
 
Soon after this, Crawley-Boevey~\cite{cb} and Ringel~\cite{rin} considered exceptional sequences and their mutation in the setting of quiver representations, or, equivalently, hereditary module categories of finite-dimensional algebras. The
{\em perpendicular categories} studied by Geigle-Lenzing~\cite{gl} and Schofield~\cite{sch} are central for the theory in that setting. 

There are natural ways to generalise the concept to module categories of arbitrary global dimension; see e.g.\ \cite{kra}.
However, in general, with those definitions, the existence of complete exceptional sequences in the module category is rather rare.
In particular, the nice properties of the above-mentioned perpendicular categories do not generalise.

Inspired by connections to cluster algebras and cluster combinatorics~\cite{bmrrt,fz}, Adachi-Iyama-Reiten~\cite{air} introduced $\tau$-tilting and $\tau$-rigid modules for an arbitrary finite-dimensional algebra $\Lambda$ over a field $K$.
Recall that a finite-dimensional (left) $\Lambda$-module $M$ is called \emph{rigid} if $\Ext^1(M,M) = 0$ and is called \emph{exceptional} if it is rigid and $\End(M)$ is a division algebra.
Adachi-Iyama-Reiten define a module $M$ to be \emph{$\tau$-rigid} if $\Hom(M, \tau M)=0$,
where $\tau$ denotes the
	AR-translate for the category $\mods \Lambda$ of finite-dimensional (left) $\Lambda$-modules. 
 Then any $\tau$-rigid module is rigid, but in general
the converse holds only for hereditary algebras. In fact, the indecomposable $\tau$-rigid modules, indecomposable rigid modules, and exceptional modules all coincide for hereditary algebras.

 Now, Jasso~\cite{jasso} introduced \emph{$\tau$-perpendicular categories} as a natural generalisation of the earlier mentioned perpendicular categories. Let ${}^\perp{M}$ denote the full subcategory generated by modules $N$ with $\Hom(N,M) = 0$, and
 define $M^\perp$ dually. The \emph{$\tau$-perpendicular category} of a $\tau$-rigid module $M$ is the subcategory
 $$\J(M) : = M^\perp \cap {}^\perp{\tau M},$$
 and we say that a subcategory $\W$ is \emph{$\tau$-perpendicular} if it can be realised as such. Jasso's main theorem states that every $\tau$-perpendicular category $\W$ is equivalent to the module category of a finite-dimensional algebra, and hence has an AR-translate, which we denote by $\tau_{\W}$. It is thus natural to consider objects in $\mods \Lambda$ that are \emph{$\tau_{\W}$-rigid} in some $\tau$-perpendicular category $\W$. 
 One can then iterate the construction   
	  of $\tau$-perpendicular categories: if $\W$ is a $\tau$-perpendicular category
	  and $M$ is a $\Lambda$-module in $\W$ which is $\tau_{\W}$-rigid, we let 
	  $$\J_{\W}(M) := \W \cap M^{\perp} \cap {}^\perp (\tau_{\W}M).$$
 It was proved in~\cite[Thm. 4.12]{dirrt} that $\tau$-perpendicular categories are \emph{wide}, or equivalently exact abelian, subcategories of $\mods \Lambda$. It follows from this that objects which are $\tau_{\W}$-rigid for a
 $\tau$-perpendicular subcategory $\W$ of $\mods \Lambda$ are rigid in $\mods \Lambda$, 
 but they are in general {\em not} necessarily $\tau$-rigid.

Using these ideas, the notion of \emph{$\tau$-exceptional sequences} was introduced in~\cite{bm} using the following recursive definition. Let $n:= \rk(\Lambda)$ be the number of nonisomorphic indecomposable projective $\Lambda$-modules. A sequence $\mathcal{M} = (M_s, M_{s+1}, \dots, M_n)$ of indecomposable $\Lambda$-modules with $s\in\mathbb{Z}$ is called a 
\emph{$\tau$-exceptional sequence} if $M_n$ is $\tau$-rigid and  $(M_s, M_{s+1}, \dots, M_{n-1})$
is a $\tau_{\J(M_n)}$-exceptional sequence (implicit in this notation is that $M_s, \dots, M_{n-1}$ all lie in $\J(M_n)$). Note that, if $\Lambda$ is hereditary, then this notion is equivalent to the classical notion
of exceptional sequences in $\mods \Lambda$. Moreover, over any algebra, any $\tau$-exceptional sequence will satisfy $s \geq 1$ (justifying the indexing of the terms from $s$ to $n$). If $s = 1$, we say that the $\tau$-exceptional sequence is \emph{complete}. We refer to $\tau$-exceptional sequences of length $2$ (i.e., with $s = n-1$) as \emph{$\tau$-exceptional pairs}.

 Beyond being a natural generalisation of exceptional sequences, $\tau$-exceptional sequences appear in many important contexts. For example, complete $\tau$-exceptional sequences are in bijection with maximal chains in the poset of $\tau$-perpendicular categories~\cite[Theorem~6.16]{bh}, and this bijection has been used to understand the corresponding poset topology in special cases~\cite{bh2,h}. An extension of this bijection can also be ``categorified'' using the \emph{$\tau$-cluster morphism category} associated to $\Lambda$. This category was introduced in~\cite{igto} (hereditary case),~\cite{bm2} ($\tau$-tilting finite case), and~\cite{bh} (general case). Its objects are the $\tau$-perpendicular subcategories and its morphisms can be described using the more general \emph{signed} $\tau$-exceptional sequences (see Section~\ref{sec:sequences} for the definition). In many cases, it has been shown that the classifying space of the $\tau$-cluster morphism category is a $K(\pi,1)$ for $\pi$ the \emph{picture group}~\cite{itw} of $\Lambda$; see~\cite{bh2,hi,hi_2}. There has also been recent interest in understanding and generalising picture groups and $\tau$-cluster morphism categories using silting theory~\cite{borve} and (generalisations of) $g$-vector fans~\cite{kaipel,sttw}; $\tau$-exceptional sequences have also been shown to be related to stratifying systems~\cite{mt}.

The main aim of this paper is to introduce a notion of mutation on $\tau$-exceptional sequences for arbitrary finite-dimensional algebras. In order to do this, we first introduce a notion of mutation for $\tau$-exceptional pairs.

Similar to the triangulated case and the quiver algebra case, there will be mutually inverse notions of left and right mutation. However, unlike in the triangulated and quiver algebra cases, there sometimes exist $\tau$-exceptional pairs which can only be mutated in one direction. 
We therefore introduce notions of left/right mutable $\tau$-exceptional pairs; see Section \ref{sec:pair_mutation} for precise definitions.

    Then we define left and right mutation operations as maps 
    \[\{\text{left mutable $\tau$-exceptional pairs}\} \qquad \stackrel[\psi]{\varphi}{\rightleftarrows} \qquad \{\text{right mutable $\tau$-exceptional pairs}\}\]
    
and prove that:
 
	\begin{theorem}[Theorem~\ref{thm:mutation_pairs}]\label{thm:mutation_pairs_intro}
	The maps $\varphi$ and $\psi$ are mutually inverse bijections.
	\end{theorem}

The definitions of $\varphi$ and $\psi$ use the notions of \emph{support $\tau$-rigid objects} (see Section~\ref{sec:support}) and \emph{signed $\tau$-exceptional sequences} (see Definition~\ref{def:signed_sequences}). A symmetric group action on the set of signed $\tau$-exceptional sequences was introduced in \cite{bmpreprint} (see Theorem~\ref{thm:signed_mutation}). This action does not preserve (unsigned) $\tau$-exceptional sequences, but still plays a role in the definitions of $\varphi$ and $\psi$.
The definition of the action exploits a bijection $\Ebm_U$ between the set of ``complements'' of a support $\tau$-rigid object $U$ and the set of support $\tau_{\J(U)}$-rigid objects in a subcategory $\C(\J(U))$ of the bounded derived category $\derD(\mods\Lambda)$; see Theorem~\ref{thm:bmcorr}. Theorems~\ref{thm:E_map} and~\ref{thm:F_map_1} provide a new description of $\Ebm_U$ and an explicit description of its inverse $\Fbm_U$.

  In Section \ref{sec:examples} we provide an example of an algebra of rank three where non-mutable $\tau$-exceptional pairs exist. This example highlights that, while they certainly do exist, non-mutable pairs only occur in very specific contexts. Indeed, we prove the following.
  
  \begin{theorem}[Theorem~\ref{prop:one_mutation}]\label{thm:always_mutable_intro}
	Every $\tau$-exceptional pair is left or right mutable (or both).
\end{theorem}

\begin{theorem}[Theorem~\ref{prop:mutation_complete}]\label{thm:mutation_complete_intro}
     If $\Lambda$ is $\tau$-tilting finite, hereditary or has rank two, then all $\tau$-exceptional pairs are both left and right mutable.
\end{theorem}

Consider now a (not necessarily complete) $\tau$-exceptional sequence $\mathcal{M} = (M_s, \dots, M_n)$,
and consider the subcategories defined recursively as follows:
$$\J(M_s, M_{s+1}, \dots, M_n) := \J_{\J(M_{s+1}, \dots, M_n)}(M_s).$$
We set $\J(M_s,\ldots,M_n) = \mods\Lambda$ if $s = n+1$.
We say that the sequence $\mathcal{M}$ is \emph{left (resp. right) $i$-mutable} (for $ s \leq i \leq n-1$) if
the pair $(M_i, M_{i+1})$ is left (resp. right) mutable in $\J(M_{i+2}, \dots, M_n)$. We say that the algebra $\Lambda$ is \emph{mutation complete} if every $\tau$-exceptional sequence $(M_s,\ldots,M_n)$ is both right and left $i$-mutable for all $s \leq i < n$. The following is then a consequence of Theorem~\ref{thm:mutation_complete_intro}.

 \begin{corollary}[Corollary~\ref{cor:mutation_complete}]\label{cor:mutation_complete_intro}
  If $\Lambda$ is $\tau$-tilting finite, hereditary or has rank two (i.e. $\rk(\Lambda) =2$), then it is mutation complete.
   \end{corollary}

 For hereditary algebras, the $\tau$-exceptional sequences are precisely the exceptional sequences, and hence they are mutable in the classical sense \cite{cb,rin} (see \cite[Sect.\ 1]{brt} for explicit formulas).
 We show that the new notion of mutation recovers the classical one in this case.
		
\begin{theorem}[Theorem~\ref{thm:hereditary_case}]\label{thm:hereditary_case_intro}
		If $\Lambda$ is hereditary, then left and right mutation of
 $\tau$-exceptional sequences coincide with classical left and right mutation of exceptional sequences.
		\end{theorem}

It is also natural to ask whether mutation of $\tau$-exceptional sequences is a transitive operation (as it is for hereditary algebras~\cite[Thm.]{cb},~\cite[Sect.\ 7 Cor.]{rin}). For algebras of rank two, we give a complete answer to this question in Section~\ref{sec:rank_2_exchange}. This comes from comparing the mutation theory of $\tau$-exceptional sequences to that of \emph{support $\tau$-tilting modules}. 
More precisely, we recall the definition of the \emph{oriented exchange graph} $\Hasse(\stt\Lambda)$ in Section~\ref{sec:rank_2_exchange}. For algebras of rank two, the connectivity of the oriented exchange graph fully characterises the transitivity of the mutation operation on $\tau$-exceptional sequences:

\begin{theorem}[Corollary~\ref{cor:connected}]\label{cor:connected_intro}
Suppose $\Lambda$ has rank two. Then the number of orbits of complete $\tau$-exceptional sequences (under the notions of left/right mutation) coincides with the number of connected components of $\Hasse(\stt\Lambda)$. In particular, mutation of $\tau$-exceptional sequences is transitive if and only if $\Hasse(\stt\Lambda)$ is connected.
\end{theorem}

In Section~\ref{sec:examples}, we recall from~\cite{terland} an example of an algebra of rank two for which $\Hasse(\stt\Lambda)$ is not connected. Thus this is also an example where the mutation of $\tau$-exceptional sequences is not transitive. We also explain, in Section~\ref{sec:brick_label}, how Theorem~\ref{cor:connected_intro} is related to Asai's ``brick labelling'' of $\Hasse(\stt\Lambda)$~\cite{asai} and the ``brick-$\tau$-rigid correspondence'' of Demonet-Iyama-Jasso~\cite{dij}.

Finally, we note that, in the hereditary case, the classical mutation operators on complete exceptional sequences are known to satisfy the braid relations~\cite{cb,rin}. In other words, the braid group acts transitively on the set of complete exceptional sequences in the hereditary case. Thus one may ask whether the mutation theory developed in this paper likewise defines a (possibly non-transitive) braid group action on the set of complete $\tau$-exceptional sequences of a non-hereditary algebra. In Section~\ref{sec:examples}, we show that this is generally not the case; that is, we present a non-hereditary ($\tau$-tilting finite) algebra for which the mutation operators we introduce do not satisfy the braid relations. 

The work in this paper, including some of the examples in Section~\ref{sec:examples}, was greatly aided by the applet~\cite{geuenich} which allows for the computation of AR-quivers and support $\tau$-tilting quivers as well as other features of finite-dimensional string algebras.
	
\section{\texorpdfstring{$\tau$}{Tau}-rigid objects and torsion pairs}\label{sec:torsion}

 Let $\Lambda$ be a finite-dimensional algebra over an arbitrary field $K$. We denote by $\mods\Lambda$ the category of finite-dimensional left $\Lambda$-modules, and the term ``module'' will always refer to an object of $\mods\Lambda$. All subcategories considered are full and closed under isomorphisms, finite direct sums, and direct summands, and all modules and algebras can be assumed to be basic. The \emph{rank} of a module $M \in \mods\Lambda$, denoted $\rk(M)$, is the number of isomorphism classes of indecomposable direct summands of $M$. The rank of the algebra $\Lambda$ is its rank as a module, and we denote $n := \rk(\Lambda)$.

 For a subcategory $\SC$ of $\mods \Lambda$, we let $\Gen \SC$ (resp. $\Cogen \SC$) denote the smallest subcategory containing $\SC$ and being closed under factor objects 
(resp. subobjects). We let $\Filt \SC$  denote the smallest subcategory whose objects have finite filtrations with subfactors
in $\SC$. As in the introduction, we let $\SC^\perp$ (resp. ${}^\perp \SC$) be the full subcategory whose objects are the modules $X$ satisfying $\Hom(-,X)|_{\SC} = 0$ (resp. $\Hom(X,-)|_{\SC} = 0$).
For a module $M$, we will write $\Gen M$, rather than $\Gen (\add M)$ etc., where $\add M$ is the full additive subcategory generated by $M$. We denote by $\ind \SC$ the set of (isomorphism classes of) indecomposable objects in $\SC$.
 
	\subsection{Torsion pairs} 
	
 Recall that a pair $(\T,\F)$ of full subcategories of $\mods \Lambda$ is called a 
 {\em torsion pair} if we have both $\T = {}^\perp \F$ and $\T^{\perp} = \F$. In this case, we say that $\T$ is a {\em torsion class} and that $\F$ is a {\em torsion-free class}.
 Given a torsion pair $(\T,\F)$ and an arbitrary module $M$, there is a canonical exact sequence, 
 $$0 \to tM \to M \to fM \to 0,$$
 unique up to isomorphism, with the property that $tM$ is in $\T$ and $fM$ is in $\F$.
 
	A subcategory $\SC$ gives rise
	to two torsion pairs in a natural way. Note that while \cite[Lem.~3.1]{ms} considers only the case where $\SC$ is a wide subcategory, it is well-known (see e.g.~\cite[Sect.\ 1.1]{asai}) that the same argument applies to an arbitrary subcategory $\SC$. 

\begin{definition-proposition}
	\label{defprop:torsionpairs}
	Let $\SC$ be a subcategory of $\mods \Lambda$. 
 \begin{itemize}
     \item[(a)]  \cite[VI.1.2]{ASS06} Then $({}^\perp(\SC^\perp), \SC^\perp)$ and $({}^\perp \SC,({}^\perp \SC)^\perp)$ are torsion pairs.
     \item[(b)] \cite[Lem.~3.1]{ms} The subcategory ${}^\perp(\SC^\perp) = \mathsf{Filt}(\Gen \SC) =: \FiltGen(\SC)$ is the smallest torsion class containing~$\SC$.
      \item[(c)] \cite[Lem.~3.1]{ms} The subcategory $({}^\perp \SC)^\perp = \mathsf{Filt}(\Cogen \SC) =: \FiltCogen(\SC)$ is the smallest torsion-free class containing~$\SC$.
 \end{itemize}
\end{definition-proposition}

A torsion pair $(\T, \F)$ is called functorially finite if $\T$
 is functorially finite. We have the following.

 \begin{proposition}\cite{as, smalo84}\label{prop:asff}
  Let $(\T, \F)$ be a torsion pair. Then the following are equivalent.
  \begin{enumerate}
      \item[(a)] $\T$ is functorially finite.
      \item[(b)] $\F$ is functorially finite.
      \item[(c)] $\T = \Gen M$ for some module $M$.
    \item[(d)] $\F = \Cogen N$ for some module $N$.
  \end{enumerate}
 \end{proposition}

We will freely use these equivalences throughout the paper.

\subsection{Projective and injective objects}\label{sec:proj_inj}
	
Let $\SC$ be a subcategory of $\mods\Lambda$. We denote by $\P(\SC)$ (respectively, $\Ps(\SC)$) the subcategory of $\SC$ consisting of the objects in $\SC$ which are Ext-projective (respectively, split Ext-projective) in $\SC$.
We denote by $\Pns(\SC)$ the subcategory of $\SC$ consisting of the objects in  $\P(\SC)$ none of whose indecomposable direct summands are split Ext-projective in $\SC$. Each of these subcategories is an additive subcategory of $\SC$. As a special case, we denote by $\P(\Lambda) := \P(\mods\Lambda)$ the category of projective $\Lambda$-modules. Note that in this case we have $\P(\Lambda) = \Ps(\Lambda)$ and $\Pns(\Lambda) = 0$.

Similarly, we denote by $\I(\SC)$ (respectively, $\I_{s}(\SC)$) the subcategory of $\SC$ consisting of the objects in $\SC$ which are Ext-injective in $\SC$ (respectively, split Ext-injective in $\SC$).
We denote by $\I_{ns}(\SC)$ the subcategory of $\SC$ consisting of the objects in $\I(\SC)$ 
none of whose indecomposable direct summands are split Ext-injective in $\SC$. Each of these subcategories is an additive subcategory of $\SC$.

We will mostly be interested in the
subcategories $\P(\T)$,  $\Ps(\T), \Pns(\T), \I(\F), \Is(\F)$, and $\Ins(\F)$ for $(\T,\F)$ a functorially finite torsion pair. By Proposition~\ref{prop:asff}, we have that 
$\P(\T) = \add M$ for some module $M$ in this case.  
So when $(\T, \F)$ is a functorially finite torsion pair, then $\P(\T)$, $\Ps(\T)$ and $\Pns(\T)$ have additive generators, and we denote these by
$\Pmod(\T)$, $\Pmods(\T)$ and $\Pmodns(\T)$, respectively.
 Similarly, we denote by $\Imod(\F)$, $\Imods(\F)$ and $\Imodns(\F)$ the additive generators of $\I(\F)$, $\Is(\F)$ and $\Ins(\F)$, respectively. We give a precise relationship between the additive generators for $\T$ and $\F$ in Corollary~\ref{cor:proj_inj}. Note also that, without the functorial finiteness assumption, we have the following:
 
 \begin{proposition} \label{prop:tauprojectiveinjective}
     \textnormal{(See~\cite[VI.1.11]{ASS06})}
     Let $(\T,\F)$ be a torsion pair. Then a module $M \in \T$ (resp. $M \in \F$) satisfies $M \in \P(\T)$ (resp. $M \in \I(\F)$) if and only if $\tau M \in \F$ (resp. $\tau^{-1}M \in \T$).
 \end{proposition}


\subsection{Functorially finite torsion pairs and $\tau$-rigid modules}	\label{sec:ff}

The definition of $\tau$-rigidity  was
recalled in the introduction. We also recall, from~\cite{air}, that a $\tau$-rigid module $M$ is called \emph{support $\tau$-tilting} if there exists a projective module $P$ (possibly $P = 0$, in which case the word ``support'' may be removed) such that $\Hom(P,M) = 0$ and $\rk(P \amalg M) = \rk(\Lambda)$. We will also use the dual of each of these notions. From~\cite[Sect. 2]{air}, a module $N$ is \emph{$\tau^{-1}$-rigid} if $\Hom(\tau^{-1}N,N) = 0$. If in addition there is an injective module $I$ such that 
 $\Hom(N,I) = 0$ and $\rk (N \amalg I) = \rk(\Lambda)$, then $N$ is called a \emph{support $\tau^{-1}$-tilting module}. As with $\tau$-tilting, the word ``support'' may be removed in case $I = 0$.

The following result can be useful for determining whether a module is $\tau$-rigid.

\begin{proposition}\label{prop:ASExt} \cite[Props.~5.6 and~5.8]{as}
    Let $M, N \in \mods\Lambda$.
    \begin{enumerate}
        \item[(a)] Then $\Hom(N,\tau M) = 0$ if and only if $\Ext^1(M,-)|_{\Gen N} = 0$.
        \item[(b)] Then $\Hom(\tau^{-1}N,M) = 0$ if and only if $\Ext^1(-,N)|_{\Cogen M} = 0$.
    \end{enumerate}
\end{proposition}

We will often consider $\tau$-rigid modules which satisfy the following additional property.

\begin{definition}\label{def:gen_min}
A module $M$ is called {\em gen-minimal} if, for any proper direct summand $M'$ of $M$,
we have $\Gen M' \subsetneq \Gen M$. 
{\em Cogen-minimal} modules are defined dually.
\end{definition}  
  
The following gives precise connections between functorially finite torsion pairs and $\tau$-rigid 
modules, and is essentially contained in~\cite{air}.

		\begin{theorem}
			\cite[Sect.\  2.3]{air} 
			\label{thm:Bongartz}
			Let $M$ be a $\tau$-rigid module. Then
			\begin{itemize}
			\item[(a)] The subcategories ${}^\perp \tau M$ and
			$\Gen M$ are functorially finite torsion classes. 
   
   \item[(b)] The map $M \mapsto \Gen M$ gives a bijection between support $\tau$-tilting modules and functorially finite torsion classes in $\mods \Lambda$, with inverse $\T \mapsto \Pmod(\T)$. The map $M \mapsto \Gen M$ also gives a bijection between gen-minimal 
   $\tau$-rigid modules and functorially finite torsion classes in $\mods \Lambda$, with inverse $\T \mapsto \Pmods(\T)$.
				\item[(c)]
				The module $\Pmod({}^\perp \tau M)$ is $\tau$-tilting and contains $M$ as a direct summand. Writing $\Pmod({}^\perp \tau M)=M\amalg B_M$, we have $${}^\perp (\tau M)={}^\perp \tau (M\amalg B_M)=\Gen (M\amalg B_M),$$
				and the equality ${}^\perp \tau M={}^\perp \tau (M\amalg B_M)$ characterises $B_M$ amongst the modules whose direct sum with $M$ is a $\tau$-tilting module.
				\item[(d)] The module $\Pmod(\Gen M)$ is support $\tau$-tilting and we can write $\Pmod(\Gen M)=M\amalg C_M$ for a module $C_M$. 
			\end{itemize}
		\end{theorem}

	The module $B_M$ is known as the \emph{Bongartz complement} of $M$. The module  
	$C_M$ is known as the \emph{co-Bongartz} complement of $M$.

	\begin{lemma}  \label{lem:Bongartzsplit}
		Let $M$ be a $\tau$-rigid module, $B_M$ its Bongartz complement, and $C_M$ its co-Bongartz complement. 
		\begin{itemize}
			\item[(a)] \cite[Lem.\ 4.13]{bm}
			Then $B_M$ is a direct summand of $\Pmods({}^\perp \tau M)$.
			\item[(b)] \cite[Proof of Prop.\ 6.15]{bh} If $M$ is indecomposable non-projective, then $B_M = \Pmods({}^\perp \tau M)$.
   \item[(c)] The modules $B_M$ and $C_M$ have no direct summands in common.
		\end{itemize}
	\end{lemma}
 \begin{proof}
     Note that if $X$ is a direct summand of $C_M$, then, since $X$ lies in $\Gen M\subseteq {}^{\perp}(\tau M)$, it cannot be split projective in ${}^{\perp}(\tau M)$; part (c) then follows from part (a). 
 \end{proof}
  We have dual versions of the above.
 
	\begin{theorem}
		\label{thm:dualBongartz}
		Let $N$ be a $\tau^{-1}$-rigid module. Then
		\begin{itemize}
			\item[(a)]
   The subcategories $(\tau^{-1} N)^{\perp}$ and
			$\Cogen N$ are functorially finite torsion-free classes. 
   \item[(b)] The map $N \mapsto \Cogen N$ gives a bijection between support $\tau^{-1}$-tilting modules and functorially finite torsion-free classes in $\mods \Lambda$, with inverse $\F \mapsto \Imod(\F)$. The map $N \mapsto \Cogen N$ also gives a bijection between cogen-minimal $\tau^{-1}$-rigid objects and functorially finite torsion-free classes in $\mods \Lambda$, with inverse $\F \mapsto \Imods(\F)$.

   \item[(c)]
			The module $\Imod((\tau^{-1}N)^{\perp})$ is $\tau^{-1}$-tilting and contains $N$ as a direct summand. Writing $\Imod((\tau^{-1} N)^{\perp})=N\amalg B'_N$, we have $$(\tau^{-1} N)^{\perp}=(\tau^{-1} (N\amalg B'_N))^{\perp}=\Cogen (N\amalg B'_N),$$
			and the equality $(\tau^{-1} N)^{\perp}=(\tau^{-1} (N\amalg B'_N))^{\perp}$ characterises $B'_N$ amongst the modules whose direct sum with $N$ is a $\tau^{-1}$-tilting module.
			\item[(d)] The module $\Imod(\Cogen N)$ is support $\tau^{-1}$-tilting and we can write $\Imod(\Cogen N)=N\amalg C'_N$ for a module $C'_N$. 
		\end{itemize}
	\end{theorem}
	
	We refer to the module $B'_N$ as the \emph{dual Bongartz complement} of $N$ and $C'_N$ as the \emph{dual co-Bongartz complement} of $N$.
	
	\begin{lemma} \cite[Lem.\ 4.13]{bm}
		\label{lem:dualBongartzSplit}
Let $N$ be a $\tau^{-1}$-rigid module, and $B'_N$ its dual Bongartz complement.
  \begin{itemize}
			\item[(a)] Then $B'_N$ is a direct summand of $\Imods((\tau^{-1} N)^{\perp})$. 
			
			\item[(b)] If $N$ is indecomposable non-injective, then $B'_N = \Imods((\tau^{-1} N)^{\perp})$.
		\end{itemize}
		 	\end{lemma}

The following are direct consequences of Theorems~\ref{thm:Bongartz} and \ref{thm:dualBongartz}.

\begin{corollary}\label{cor:genmin}
   \begin{itemize}
       \item[(a)] Let $M$ be a $\tau$-rigid module. Then $\Pmods(\Gen M)$ is a direct summand of $M$. 
 If in addition $M$ is gen-minimal, then $M = \Pmods(\Gen M)$.
 \item[(b)] Let $N$ be a $\tau^{-1}$-rigid module. Then $\Imods(\Cogen N)$ is a direct summand of $N$. If in addition $N$ is cogen-minimal, then $N = \Imods(\Cogen N)$.
   \end{itemize} 
\end{corollary}

The following two results will be useful.
 
	\begin{lemma}
		\label{lem:eachinperp}
		Assume $M$ and $N$ are indecomposable $\tau$-rigid modules, not both projective.
		If $N\in\P({}^{\perp} \tau M)$ and $M\in \P({}^{\perp} \tau N)$, then $M= N$.
	\end{lemma}

\begin{proof}
Suppose, for a contradiction, that
$N\in\P({}^{\perp} \tau M)$ and $M\in \P({}^{\perp} \tau N)$, with $N\not=M$.

Since $N\in \P({}^{\perp}\tau M)$, and $N\not= M$, we can write the Bongartz complement of $M$ as $B_M=N\amalg \overline{B}_M$ for some module $\overline{B}_M$.
Then, using the properties of the Bongartz complement, we have:
\begin{equation*}
        {}^{\perp} \tau(M\amalg N)
        \subseteq {}^{\perp} \tau M
        ={}^{\perp}\tau(M\amalg B_M)
        ={}^{\perp}\tau(M\amalg N\amalg \overline{B_M})
        \subseteq {}^{\perp}\tau (M\amalg N),
\end{equation*}
so we have equality everywhere, and, in particular, ${}^{\perp} \tau M={}^{\perp} \tau(M\amalg N)$. 
A symmetric argument gives ${}^{\perp} \tau N={}^{\perp} \tau(M\amalg N)$, so ${}^{\perp} \tau M={}^{\perp} \tau N$.
Since $M\in \P({}^{\perp} \tau N)$ and $M\not= N$, we have that $M$ is a direct summand of $B_N$. 
Hence $M\in \P_s({}^{\perp} \tau N)$ by Lemma \ref{lem:Bongartzsplit}(a). But then $M$ is a direct summand of $P_s({}^{\perp}\tau N) = P_s({}^{\perp}\tau M)$. If $M$ is non-projective, $P_s({}^{\perp}\tau M)=B_M$ by Lemma~\ref{lem:Bongartzsplit}(b), giving a contradiction. A similar argument gives a contradiction in the case where $N$ is non-projective.
\end{proof}

The following is well-known, and an explicit proof can be found in 
\cite[Lemma 3.9]{cfy}.	
	\begin{lemma} \label{lem:twogens}
		Let $M$ and $N$ be indecomposable modules, and suppose that $M\in \Gen N$ and $N\in \Gen M$. Then $M= N$.
	\end{lemma}


\subsection{Support $\tau$-rigid objects}\label{sec:support}

Following~\cite{air}, a \emph{$\tau$-rigid pair} is a pair of modules $(M,P)$ such that 
$M$ is $\tau$-rigid, the module $P$ is projective, and $\Hom(P,M) = 0$. As in \cite{bm}, we will consider the corresponding object $U = M\amalg P[1]$ in the full subcategory $\C(\Lambda) := \mods\Lambda \amalg \mods \Lambda[1]$ of the bounded derived category $\derD(\Lambda) := \derD(\mods\Lambda)$. With this perspective, we call $U$ a \emph{support $\tau$-rigid object in $\C(\Lambda)$.}
 The object $U$ is said to be \emph{support $\tau$-tilting} if in addition $\rk(U) = \rk(M) + \rk(P) = \rk(\Lambda)$. By convention, whenever we say that $U = M \amalg P[1]$ is a support $\tau$-rigid object (in $\C(\Lambda)$), we are referring to a decomposition with $M\in \mods \Lambda$ and $P \in \P(\Lambda)$.
Note that if $\Lambda$ is a hereditary algebra, then the support ($\tau$-)tilting
objects naturally correspond to the (basic) cluster tilting objects of the cluster category \cite{bmrrt} associated to $\Lambda$.
By~\cite[Prop. 2.3]{air}, the association $M\amalg P[1] \mapsto M$ is a bijection between the set of support $\tau$-tilting objects and the set of support $\tau$-tilting modules. The inverse sends $M$ to the object $M \amalg P[1]$ for $P$ the maximal projective module satisfying $\Hom(P,M) = 0$.

  Dually, recall from
		\cite[Sects. 2, 3]{bh} that an object $U= M \amalg I[-1]$ in $\C(\Lambda)[-1]$ is called \emph{support $\tau^{-1}$-rigid} if
		$M$ is a $\tau^{-1}$-rigid module and $I$ is an injective module with $\Hom(M,I) =0$. As with support $\tau$-rigid objects, whenever we say that $U = M \amalg I[-1]$ is a support $\tau^{-1}$-rigid object (in $\C(\Lambda)[-1]$), we are referring to a decomposition with $M \in \mods\Lambda$ and $I \in \I(\Lambda)$. We denote by $\nu$ the Nakayama functor on $\mods \Lambda$. The results of~\cite[Sect.~2.2]{air} imply the following relationship between support $\tau$-rigid objects and support $\tau^{-1}$-rigid objects.

  \begin{lemma}
		\label{lem:twokinds2}
		\begin{itemize}
			\item [(a)] Let $M \amalg P[1]$ be a support $\tau$-rigid object in $\C(\Lambda)$, and let $Q$ be the maximal projective direct summand of $M$. Then $\tau M \amalg \nu P \amalg \nu Q[-1]$ is a support $\tau^{-1}$-rigid object.
			\item[(b)] Let $N \amalg I[-1]$ be a support $\tau^{-1}$-rigid object. Let $E$ be the maximal injective direct summand of $N$. Then $\tau^{-1}N \amalg \nu^{-1} I \amalg \nu^{-1} E[1]$ is a support $\tau$-rigid object.
		\end{itemize}
  \end{lemma}

The following summarises some useful characterisations of functorially finite torsion classes.

\begin{lemma}\label{lem:perp_injectives_2}
    Let $(\T, \F)$ be a functorially finite torsion pair in $\mods\Lambda$. Let $P$ be the maximal projective $\Lambda$-module with the property that $\Hom(P, \T) = 0$. Then $\T$ coincides with all of the following.
    \begin{multicols}{3}
    \begin{enumerate}
        \item[(a)] $\Gen \Pmods(\T)$
        \item[(b)] $\Gen \Pmod(\T)$
        \item[(c)] ${}^\perp{\tau \Pmod(\T)} \cap P^\perp$
        \item[(d)] ${}^\perp{\tau \Pmodns(\T)} \cap P^\perp$ 
        \item[(e)] ${}^\perp {\Imod(\F)}$
        \item[(f)] ${}^\perp{\Imods(\F)}$
    \end{enumerate}
    \end{multicols}
\end{lemma}

\begin{proof}
   The fact that $\T = \Gen \Pmods(\T) = \Gen \Pmod(\T)$ is contained in Theorem~\ref{thm:Bongartz}(b). Similarly, we have that $\F = \Cogen\Imods(\F) = \Cogen\Imod(\F)$ by Theorem~\ref{thm:dualBongartz}(b). Since ${}^\perp(-) = {}^\perp \Cogen(-)$, it follows that $\T = {}^\perp \F = {}^\perp \Imods(\F) = {}^\perp \Imod(\F)$. Finally, note that $\Pmod(\T) \amalg P[1]$ is a support $\tau$-tilting object by Theorem~\ref{thm:Bongartz}(d) and the discussion preceding Lemma~\ref{lem:twokinds2}. The fact that $\Gen \Pmod(\T) = {}^\perp{\tau \Pmod(\T)} \cap P^\perp = {}^\perp{\tau \Pmodns(\T)} \cap P^\perp$ is thus contained in~\cite[Prop.~2.7]{bh}.
\end{proof}

\section{Wide subcategories}\label{sec:wide}

 \subsection{$\tau$-perpendicular wide subcategories}

 Recall that a subcategory $\W$ is called \emph{wide} if it is closed under extensions, kernels, and cokernels. Alternatively, $\W$ is wide if and only if it is an exact embedded abelian subcategory.

Using the notation introduced at the end of the previous section, it will be convenient to consider the following more general construction of $\tau$-perpendicular categories. 

\begin{definition-proposition}\cite[Cor. 3.25]{bst}  \cite[Thm.~1.4]{jasso}  \cite[Thm.~4.12]{dirrt}\label{prop:J_def}
Let  $U = M \amalg P[1]$ be a support $\tau$-rigid object in $\C(\Lambda)$. 
Then the $\tau$-perpendicular category of $U$ is 
$$\J(U) := \J(M) \cap \J(P) = M^\perp \cap {}^\perp \tau M \cap P^\perp.$$ This is a wide subcategory of $\mods \Lambda$ and is equivalent to the module category of a finite-dimensional algebra. Moreover, 
$$\rk(\J(U)) := \rk(\Pmod(\J(U))) = \rk(\Lambda) - \rk(M) - \rk(P).$$
\end{definition-proposition}

\begin{remark}\label{rem:tau_perp}
    Note that every $\tau$-perpendicular category $\W$ can be realised as $\W = \J(M)$ for some $\tau$-rigid module $M$ (with no shifted projective summand). See~\cite[Remark~4.3]{bh2} for an explicit construction. As a consequence, results from \cite{jasso} (which considers only $\tau$-perpendicular categories of the form $\J(M)$ for $M$ a module) can be freely applied to $\tau$-perpendicular subcategories of the form $\J(M\amalg P[1])$.
\end{remark}

In the notation of Definition-Proposition~\ref{prop:J_def} above, we refer to $\rk(\J(U))$ as the \emph{rank} of $\J(U)$ and to $\rk(\Lambda) - \rk(\J(U)) = \rk(U)$ as the \emph{corank} of $\J(U)$.

We also have the following dual construction.

\begin{definition}\label{def:Jinv_def}
Let  $V = N \amalg I[-1]$ be a support $\tau^{-1}$-rigid object in $\C(\Lambda)$. Then the \emph{$\tau^{-1}$-perpendicular category} of $V$ is
$$\Jinv(V) := {}^\perp N \cap (\tau^{-1} N)^\perp \cap {}^\perp I.$$
\end{definition}

The following is a straightforward consequence of the definitions (noting Lemma~\ref{lem:twokinds2}). See also~\cite[Theorem~1.1]{bh}.

\begin{lemma}\label{lem:Jtauminus}
    	\begin{itemize}
			\item [(a)] Let $M \amalg P[1]$ be a support $\tau$-rigid object (with $M, P \in \mods \Lambda)$. Let $Q$ be the maximal projective direct summand of $M$. Then $\J(M\amalg P[1]) = \Jinv(\tau M \amalg \nu P \amalg \nu Q[-1])$.
			\item[(b)] Let $N \amalg I[-1]$ be a support $\tau^{-1}$-rigid object (with $N, I \in \mods \Lambda)$. Let $E$ be the maximal injective direct summand of $N$. Then $\Jinv(N \amalg I[-1]) = \J(\tau^{-1}N \amalg \nu^{-1} I \amalg \nu^{-1} E[1])$.
		\end{itemize}
\end{lemma}

In~\cite{jasso}, Jasso compared the functorially finite torsion classes in $\mods \Lambda$ and in $\J(M)$ for a $\tau$-rigid module
$M$, and in~\cite{bm} it was shown that his technique in particular gives rise to a useful 
bijection in terms of indecomposable $\tau$-rigid modules.

\begin{theorem}\cite[Prop.~4.5]{bm} \textnormal{(see also \cite[Thm.~3.14]{jasso})}\label{thm:jasso}
Let $M$ be a $\tau$-rigid module, and let $f_M$ denote the torsion-free functor relative
to the torsion pair $(\Gen M, M^{\perp})$.
Then $X \mapsto f_M X $ induces a bijection \[f_M: \{X \in \ind \mods \Lambda \mid
 X \amalg M \text{ $\tau$-rigid and $X \not \in \Gen M$}\} \rightarrow \{Y \in \ind \J(M) \mid Y \text{ $\tau_{\J(M)}$-rigid}\}.\]
\end{theorem}

An extended version of Theorem \ref{thm:jasso} was given in~\cite{bm}. In order to recall the precise statement, we first establish some additional notation and terminology. For 
a $\tau$-perpendicular subcategory $\W$, we denote $\C(\W):= \W \amalg \W[1] \subseteq \derD(\Lambda)$. We emphasise that $[1]$ denotes the shift functor in $\derD(\Lambda)$.

\begin{definition}\label{def:support_tau_rigid_reduction}
    Let $\W \subseteq \mods \Lambda$ be a $\tau$-perpendicular category. An object $V = N \amalg Q[1] \in \C(\W)$ is then said to be \emph{support $\tau_{\W}$-rigid} if $N$ is a $\tau_{\W}$-rigid module, $Q \in \P(\W)$, and $\Hom(Q,N) = 0$.
\end{definition}

\begin{remark}\label{rem:CW}
    As recalled in the introduction, \cite[Thm.~1.4]{jasso} says that every $\tau$-perpendicular subcategory $\W$ of $\mods \Lambda$ admits an exact equivalence of categories $F_\W: \W \rightarrow \mods\Lambda_\W$ for some finite-dimensional algebra $\Lambda_\W$. For clarity, denote by $\langle1\rangle$ the shift functor in $\derD(\Lambda_\W)$. As explained in \cite[Sect.\ 1]{bm}, the functor $F_\W$ induces an equivalence of categories $\C(\W) \rightarrow \mods\Lambda_\W \amalg \mods\Lambda_\W\langle1\rangle$ given on objects by $N\amalg Q[1] \mapsto F_\W N \amalg F_\W Q\langle 1\rangle$. Moreover, this functor gives a bijection between the set of support $\tau_{\W}$-rigid objects in $\C(\W)$ and the set of support $\tau$-rigid objects in $\mods\Lambda_\W \amalg \mods\Lambda_\W\langle1\rangle$. As a result, the general theory of support $\tau$-rigid objects in $\C(\Lambda)$ can also be applied to support $\tau_\W$-rigid objects in $\C(\W)$.
\end{remark}

\begin{theorem}
\label{thm:bmcorr}
Let $U = M \amalg P[1]$ be a support $\tau$-rigid object in $\C(\Lambda)$. Then
\begin{itemize}
\item[(a)]
\cite[Sect.\ 3]{bm2}~\cite[Prop.~5.6]{bm} 
\begin{itemize}
\item[(i)] 
 There is a bijection 
 $$\{V \in \ind \C(\Lambda) \mid V \amalg U \text{ (basic) support $\tau$-rigid}\}$$ 
 $$\downarrow \Ebm_U$$ 
 $$\{W \in \ind \C(\J(U)) \mid W 
 \text{ support $\tau_{\J(U)}$-rigid}\}.$$
\item[(ii)] For $V\in \ind \C(\Lambda)$ with $V\amalg U$ support $\tau$-rigid, we have $E_U(V)\in \P(\J(U))[1]$ if and only if $V\in \Gen M$ or $V\in \P(\Lambda)[1]$.
\item[(iii)] If $U\in \P(\Lambda)[1]$ and $V\in \ind(\C(\Lambda))$ with $V\amalg U$ support $\tau$-rigid, we have $\Ebm_U(V)=V$. Equivalently, if $W$ is a $\tau_{J(P[1])}$-rigid module then $\Fbm_{P[1]}(W)=W$.
\end{itemize}
\item[(b)]
The bijection in (a) extends the bijection from Theorem~\ref{thm:jasso} in the sense that $\Ebm_U(V) = \Ebm_M(V) = f_MV$ for $V \in (\ind\mods\Lambda) \setminus \Gen M$ with $V \amalg U$ support $\tau$-rigid. 
\end{itemize}
\end{theorem}
\begin{proof}
    For (b), we use the notation $\Ebm^{\J(P[1])}$ to mean the version of the bijection $\Ebm$ associated to the $\tau$-perpendicular category $\J(P[1])$. Then
    we have
    $$\Ebm_{M\amalg P[1]}(V) = \Ebm_{\Ebm_{P[1]}(M)}^{\J(P[1])}\circ \Ebm_{P[1]}(V) = \Ebm_M^{\J(P[1])}(V) = f_M(V)$$ by~\cite[Theorem~6.12]{bh}, the $P = 0$ case~\cite[Prop.\ 5.5]{bm} and (a)(iii).
\end{proof}

We will give a new description of $\Ebm$ and an explicit description of its inverse $\Fbm$ in Theorems~\ref{thm:E_map} and~\ref{thm:F_map_1}.
Note also that the maps $\Ebm$ and $\Fbm$ are denoted by $\mathcal{E}$ and $\mathcal{F}$, respectively, in~\cite{bm,bm2}.
 
We will use the following relationship between the operators $\J$ and $\Ebm$.

 \begin{theorem}\label{thm:J_E}\cite[Thm.~6.4]{bh}
    \begin{enumerate}
         \item[(a)] Let $U \amalg V$ be a support $\tau$-rigid object in $\C(\Lambda)$ with $U$ indecomposable. Then $\J(U\amalg V) = \J(\Ebm_V(U),V)$.
         \item[(b)] Let $V$ be a support $\tau$-rigid object in $\C(\Lambda)$ and let $U'$ be an indecomposable support $\tau_{\J(V)}$-rigid object in $\C(\J(V))$. Then $\J(U',V) = \J(\Fbm_V(U')\amalg V)$.
     \end{enumerate}
 \end{theorem}

 \subsection{Left and right finite wide subcategories}

Precise results concerning the relationship between torsion pairs and wide subcategories were given in~\cite{ms}.
We recall first their main results, adopting the notation and definitions of~\cite{asai}.

A wide subcategory $\W$ is called {\em left finite} (resp. {\em right finite}) if the torsion class $\FiltGen(\W)$ (resp. the torsion-free class $\FiltCogen(\W)$) is functorially finite. Note that both left and right finite wide subcategories are also themselves functorially finite.

In~\cite{ms}, the authors also consider a map $\WL$ (resp. $\WR$) from the set of torsion classes (resp. torsion-free classes) to the set of wide subcategories. We will not need the explicit definitions of $\WL$ and $\WR$, but instead give an alternative formulation in Theorem \ref{thm:WLtorsionclass}.

\begin{theorem}\cite[Props. 3.3, 3.10; Thm. 3.10]{ms}\label{thm:ms}
With notation as above, the following hold.
\begin{itemize}
\item[(a)] The map $\W \mapsto \FiltGen(\W)$ (resp. $\W \mapsto \FiltCogen(\W)$) is an injection from the set of wide subcategories of $\mods \Lambda$ to the set of
torsion classes (resp. torsion-free classes).
\item[(b)] The map $\WL$ (resp. $\WR$) gives an injection from the set of functorially finite torsion classes (resp. torsion-free classes) to the set of functorially finite wide subcategories. 
\item[(c)] The maps of (a) and (b) restrict to inverse bijections between the functorially finite torsion classes (resp. torsion-free classes) and the left finite (resp. right finite) wide subcategories. Furthermore, for any wide subcategory $\W$, we have $\WL(\FiltGen(\W))=\W = \WR(\FiltCogen(\W))$.
\end{itemize}    
\end{theorem}

Descriptions of the above maps in terms of $\tau$-perpendicular wide subcategories were given in~\cite{bh}.
	
	\begin{theorem}
		\label{thm:WLtorsionclass}\cite[Lem.~4.3]{bh}
		Let $(\T,\F)$ be a functorially finite torsion pair. Let $P$ (resp. $I$) be the maximal projective (resp. injective) module which satisfies $\Hom(P,\Pmod(\T))=0$ (resp. $\Hom(\Imod(\F),I) = 0$), so that $\Pmod(\T)\amalg P[1]$ (resp. $\Imod(\F) \amalg I[-1]$) is the support $\tau$-tilting object (resp. support $\tau^{-1}$-tilting object) corresponding to $\Pmod(\T)$ (resp. $\Imod(\F)$).
		Then
		\begin{itemize}
			\item[(a)] $\WL(\T)=\J(\Pmodns(\T)\amalg P[1]) = \Jinv(\Imods(\F))$.
			\item[(b)] $\WR(\F)=\J(\Pmods(\T)) = \Jinv(\Imodns(\F)\amalg I[-1])$.
		\end{itemize}
	\end{theorem}

It was pointed out by Asai~\cite{asai} that Serre subcategories
of left (or right) finite wide subcategories are not necessarily
left (or right) finite. In particular, $\tau$-perpendicular wide subcategories are not necessarily left (or right) finite. Moreover, it was proved in~\cite[Theorem~1.1]{bh} that a wide subcategory is 
 $\tau$-perpendicular if and only if it is a Serre subcategory of a left (or right) finite wide subcategory. Note that part of~\cite[Theorem~1.1]{bh} says that this condition is left/right symmetric, i.e., that every Serre subcategory of a left finite wide subcategory is also a Serre subcategory of a right finite wide subcategory and vice versa.

Theorem~\ref{thm:WLtorsionclass} allows us to deduce the following.

\begin{corollary}\label{cor:proj_inj}
    Let $(\T, \F)$ be a functorially finite torsion pair. Let $P$ be the maximal projective module with the property that $\Hom(P,\T) = 0$. Then the following hold.
    \begin{enumerate}
        \item[(a)] $\Imods(\F) = \tau(\Pmodns(\T)) \amalg \nu P$.
        \item[(b)] $\Imodns(\F) = \tau(\Pmods(\T))$.
    \end{enumerate}
\end{corollary}

\begin{proof}
We have that $\Pmodns(\T) \amalg P[1]$ is a support $\tau$-rigid object by Theorem~\ref{thm:Bongartz}(b). Thus $\tau(\Pmodns(\T)) \amalg \nu P$ is a $\tau^{-1}$-rigid module by Lemma~\ref{lem:twokinds2}.
Definition-Proposition~\ref{defprop:torsionpairs} and Theorem~\ref{thm:dualBongartz}(a) then imply that $\FiltCogen(\tau \Pmodns(\T) \amalg \nu P)
        = \Cogen(\tau \Pmodns(\T) \amalg \nu P)$. Using this and Lemma~\ref{lem:perp_injectives_2}(d),  we then compute
\begin{align*}
    \F &= \T^\perp\\
        &= \left({}^\perp\tau\Pmodns(\T) \cap P^\perp\right)^\perp\\
        &= \left({}^\perp\tau\Pmodns(\T) \cap 
        {^{\perp}(\nu P)}\right)^\perp \\
        &= \left({}^\perp(\tau \Pmodns(\T) \amalg \nu P)\right)^\perp\\
        &= \FiltCogen(\tau \Pmodns(\T) \amalg \nu P)\\
        &= \Cogen(\tau \Pmodns(\T) \amalg \nu P).
\end{align*}
 Thus $\Imods(\F)$ is a direct summand of $\tau \Pmodns(\T) \amalg \nu P$ by Corollary~\ref{cor:genmin}(b). Now note that $\J(\Pmodns(\T) \amalg P[1]) = \WL(\T) = \Jinv(\Imods(\F))$ by Theorem~\ref{thm:WLtorsionclass}(a). 
 Thus, using Definition-Proposition \ref{prop:J_def} and Lemma \ref{lem:Jtauminus}, we conclude that $\tau \Pmodns(\T) \amalg \nu P$ and $\Imods(\F)$ have the same number of indecomposable direct summands. We conclude that $\Imods(\F) = \tau(\Pmodns(\T)) \amalg \nu P$, proving item (a). Given item (a), item (b) then follows from Proposition~\ref{prop:tauprojectiveinjective}.
\end{proof}

	\subsection{Projectives in $\tau$-perpendicular subcategories}

Let $M$ be a $\tau$-rigid module. Recall from Definition-Proposition \ref{defprop:torsionpairs} that the wide subcategory $\J(M)$ generates two minimal torsion pairs:
$$(\FiltGen(\J(M)), \J(M)^\perp)\qquad \text{   and   }\qquad ({}^\perp{\J(M)}, \FiltCogen(\J(M))).$$

\begin{proposition}\label{prop:gen_minimal}\
		Let $(\T,\F)$ be a functorially finite torsion pair. Then $\T = {}^\perp \J(\Pmods(\T))$ and $\F = \Jinv(\Imods(\F))^\perp$.
	\end{proposition}

\begin{proof}
We prove $\T = {}^\perp \J(\Pmods(\T))$. The argument that $\F = \Jinv(\Imods(\FiltCogen))^\perp$ is dual.
By Lemma~\ref{lem:perp_injectives_2}, we have $\T=\Gen \Pmods(\T)$, and hence $\F = \T^\perp = \Pmods(\T)^{\perp}$. 
By Theorem~\ref{thm:WLtorsionclass}(b), $$\WR(\Pmods(\T)^\perp)=\WR(\F)=\J(P_s(\T)).$$
Since $({}^\perp{\J(P_s(\T))}, \FiltCogen(\J(P_s(\T))))$ is a torsion pair (see Definition-Proposition~\ref{defprop:torsionpairs}), we 
have that ${}^\perp\J(\Pmods(\T))={}^{\perp} \FiltCogen(\J(\Pmods(\T)))$.  
Using Theorem~\ref{thm:ms}, we have that $\Pmods(\T)^{\perp} = \FiltCogen(\WR(\Pmods(\T)^{\perp}))$, and hence, using Theorem~\ref{thm:WLtorsionclass}(b), we have
\begin{align*}
\T&= {}^{\perp}{\F}\\
    &= {}^{\perp}(\Pmods(\T)^{\perp}) \\  
&={}^{\perp} \FiltCogen(\WR(\Pmods(\T)^{\perp}) )\\
&={}^\perp \FiltCogen(\J(\Pmods(\T))) \\
&={}^\perp \J(\Pmods(\T)).\qedhere
\end{align*}
\end{proof}

	We now state characterisations of gen-minimal $\tau$-rigid
 modules. These will be crucial for our main results. We also include the dual version of each statement.
	
	\begin{theorem}\label{thm:gen_minimal}
 \begin{itemize}
\item[(a)]		Let $M$ be $\tau$-rigid. Then $M$ is gen-minimal if and only if $\Gen M = {}^\perp{\J(M)}$.
\item[(b)]		Let $N$ be a $\tau^{-1}$-rigid $\Lambda$-module. Then $N$ is cogen-minimal if and only if $\Cogen N=\Jinv(N)^{\perp}$.
  \end{itemize}
\end{theorem}

\begin{proof}
  We prove only (a) since (b) is its dual. Assume $M$ is gen-minimal and note that the torsion class $\Gen M$ is functorially finite, so we have $\Pmods(\Gen M )= M$ by Corollary~\ref{cor:genmin}(a). Hence, by Proposition~\ref{prop:gen_minimal}, we have $\Gen M={}^\perp \J(\Pmods(\Gen M))={}^\perp \J(M)$.
   
    Conversely, assume that $\Gen M= {}^\perp\J(M)$.
 Then, by Proposition~\ref{prop:gen_minimal}, we have that $\Gen M={}^\perp \J(\Pmods(\Gen M))$.
    Combining this, we have ${}^\perp \J(M)={}^\perp \J(\Pmods(\Gen M))$, so by Definition-Proposition~\ref{defprop:torsionpairs} we also have
    $\FiltCogen(\J(M)) = \FiltCogen(\J(\Pmods(\Gen M)))$.
     Thus
     $$\WR(\FiltCogen(\J(M)))=\WR(\FiltCogen(\J(\Pmods(\Gen M)))),$$
and therefore $$\J(M)=\J(\Pmods(\Gen M)),$$
by Theorem~\ref{thm:ms}.
By Definition-Proposition~\ref{prop:J_def}, the modules $M$ and $\Pmods(\Gen M)$ have the same number of indecomposable direct summands.
Since $\Pmods(\Gen M)$ is a direct summand of $M$ (by Corollary~\ref{cor:genmin}(a)), we have that $M=\Pmods(\Gen M)$ is gen-minimal as required.
\end{proof}

As an immediate consequence of Theorem~\ref{thm:gen_minimal}, we obtain the following.

\begin{corollary}\label{cor:gen_minimal}
  \begin{itemize}
\item[(a)] Suppose $M$ is a $\tau$-rigid gen-minimal module. Then the torsion pair 
$({}^\perp{\J(M)}, \FiltCogen(\J(M)))$
is functorially finite, and thus $\J(M)$ is right finite.
\item[(b)]  Suppose $N$ is a $\tau^{-1}$-rigid cogen-minimal module. Then the torsion pair $(\FiltGen(\Jinv(N)), \Jinv(N)^\perp)$ is functorially finite, and thus $\Jinv(N)$ is left finite. 
   \end{itemize}       
 \end{corollary}

 Together with Corollary~\ref{cor:proj_inj}, the above results allow us to deduce the following descriptions of the bijections $\Ebm$ and $\Fbm$ recalled in Theorem~\ref{thm:bmcorr}.

 \begin{theorem}\label{thm:E_map}
    Let $U = M \amalg P[1]$ be a support $\tau$-rigid object in $\C(\Lambda)$, and let $V \in \ind(\mods\Lambda \amalg\P(\Lambda)[1])$ be such that $U \amalg V$ is also a support $\tau$-rigid object in $\C(\Lambda)$. Then
    $$\Ebm_{U}(V) = \begin{cases} \Pmods\left(\J(U) \cap {}^\perp{\J(U \amalg V)}\right) & V \in \mods\Lambda \textnormal{ and } V \notin \Gen M\\\Pmods\left(\J(U) \cap {}^\perp{\J(U \amalg V)}\right)[1] & \textnormal{$V\in \Gen M$ or $V\in \P(\Lambda)[1]$}.\end{cases}$$
\end{theorem}

\begin{proof}
We recall the following:
    \begin{itemize}
        \item Since $\rk(\J(U\amalg V)) = \rk(\J(U)) - 1$ (by Definition-Proposition~\ref{prop:J_def}), there exists a unique module $N$ which is $\tau_{\J(U)}$-rigid and satisfies $\J(N,U) = \J(U \amalg V)$, by~\cite[Thm.\ 6.4, Prop.\ 6.15(2)]{bh} (see also~\cite[Thm.\ 1.4, Prop.\ 10.7]{bm2} for the $\tau$-tilting finite case).
        \item If $V \in \mods\Lambda$ and $V \notin \Gen M$, then $\Ebm_U(V)$ is an indecomposable module which is $\tau_{\J(U)}$-rigid (by Theorem~\ref{thm:bmcorr})
        and satisfies $\J(\Ebm_U(V),U) = \J(U \amalg V)$ (by Theorem~\ref{thm:J_E}). Thus $\Ebm_U(V) = N$ in this case.
        \item If either $V \in \Gen M$ or $V \in \mods\Lambda[1]$, then $\Ebm_U(V)$ is the shift of an indecomposable module and satisfies $\J(\Ebm_U(V),U) = \J(\Ebm_U(V)[-1], U) = \J(U\amalg V)$ (by Theorem~\ref{thm:J_E}). Thus $\Ebm_U(V) = N[1]$ in this case.
    \end{itemize}

    It remains to show that the module $N$ from the first bullet point is $P_s\left(\J(U) \cap {}^\perp{\left[\J(U \amalg V)\right]}\right)$. To see this, note that $N$ is indecomposable, and thus gen-minimal. Thus the desired equality follows from applying Theorem~\ref{thm:gen_minimal} and Corollary~\ref{cor:genmin}(a) in the subcategory $\J(U)$ and recalling that $\J(U\amalg V) = \J_{\J(U)}(N)$ by assumption.
\end{proof}

As before, let $f_M$ denote the torsion-free functor associated to the torsion pair $(\Gen M,M^{\perp})$,
 where $M$ is $\tau$-rigid. 
 To characterise the map $\Fbm$, it will be useful to have the following.

\begin{lemma}\label{lem:E_cases}\label{lem:fUbijection}
    Let $U = M \amalg P[1]$ be a support $\tau$-rigid object in $\C(\Lambda)$.
    \begin{itemize}
        \item[(a)] The functor $f_M$ induces a bijection
            $$\ind \P(P^\perp \cap {}^\perp \tau M)\setminus \ind \add M \rightarrow \P(\J(U)).$$
        \item[(b)] If $X \in \ind \P(P^\perp \cap {}^\perp \tau M)\setminus \ind \add M$, then $\Ebm_U(X) = f_M(X)$.
    \end{itemize}
\end{lemma}

\begin{proof}

   (a) If $P = 0$, then this is \cite[Lem.~4.9]{bm}. Thus suppose $P \neq 0$. Because $M \amalg P[1]$ is support $\tau$-rigid and $\J(P[1]) = P^\perp$ is closed under taking quotients, Proposition~\ref{prop:ASExt} immediately implies that $M$ is $\tau_{\J(P[1])}$-rigid and that $P^\perp \cap {}^\perp \tau M = P^\perp \cap {}^\perp \tau_{\J(P[1])}M$. We then compute
    \begin{align*}
        \J(U) &= M^\perp \cap {}^\perp \tau M \cap P^\perp\\
            &= M^\perp \cap {}^\perp \tau_{\J(P[1])}M \cap \J(P[1])\\
            &= \J_{\J(P[1])}(M).
    \end{align*}
    Item (a) thus follows from applying \cite[Lem.~4.9]{bm} in the subcategory $\J(P[1])$.

    (b) Now let $X \in \ind \P(P^\perp \cap {}^\perp \tau M)\setminus \ind \add M = \ind \P(P^\perp \cap {}^\perp \tau_{\J(P[1])} M)\setminus
    \ind \add M$. Recall that the (relative) Bongartz complement $B_M$ of $M$ in the subcategory $P^\perp$ is defined so that $P(P^\perp \cap {}^\perp\tau_{\J(P[1])} M) = M\amalg B_M$ (see Theorem~\ref{thm:Bongartz}(c)). So, by definition, $X$ is a direct summand of $B_M$. Lemma~\ref{lem:Bongartzsplit} then says that $X \in \Ps(P^{\perp} \cap {}^\perp \tau_{\J(P[1])} M)$, and therefore that $X \amalg M$ is $\tau_{\J(P[1])}$-rigid with $X \notin \Gen M$. Again using Proposition~\ref{prop:ASExt} and the fact that $\J(P[1]) = P^\perp$ is closed under taking quotients, this means $X \amalg M$ is also $\tau$-rigid. Since $X \notin \Gen M$, item (b) then follows from Theorem~\ref{thm:bmcorr}.
\end{proof}

In the result that follows, we use the notation $N/\add M$ to denote the quotient of a module $N$ by its maximal direct summand which lies in $\add M$. Moreover, for an indecomposable support $\tau$-rigid object $U$ or an indecomposable support $\tau^{-1}$-rigid object $V$, we denote
$$\overline{\tau} U = \begin{cases} \tau U & U \in \mods\Lambda \setminus \P(\Lambda) \\ (\nu U)[-1] & U \in \P(\Lambda) \\ \nu(U[-1]) & U\in \P(\Lambda)[1], \qquad\qquad\qquad\end{cases}
\overline{\tau}^{-1}V = \begin{cases} \tau^{-1} V & V \in \mods\Lambda \setminus \I(\Lambda)\\ (\nu^{-1}V)[1] & V \in \I(\Lambda) \\ \nu^{-1}(V[1]) & V\in \I(\Lambda)[-1],\end{cases}$$
extended additively. We likewise denote by $\overline{\tau}_{\J(U)}$ and $\overline{\tau}^{-1}_{\J(U)}$ the versions of these operators defined in the $\tau$-perpendicular subcategory $\J(U)$.
\begin{remark} \label{rem:twokinds2indec}
Note that, by Lemma~\ref{lem:twokinds2}, if $U$ is a support $\tau$-rigid object, then $\overline{\tau}U$ is a support $\tau^{-1}$-rigid object, and if $V$ is a support $\tau^{-1}$-rigid object, then $\overline{\tau}^{-1} W$ is a support $\tau$-rigid object.\end{remark}

\begin{theorem}\label{thm:F_map_1}
    Let $U = M \amalg P[1]$ be a support $\tau$-rigid object in $\C(\Lambda)$, and let $W$ be an indecomposable support $\tau_{\J(U)}$-rigid object in $\C(\J(U))$. Then
    $$\Fbm_U(W) = \begin{cases}
        \Pmods(\FiltGen(M \amalg W))/\add M & \text{if $W \in \mods\Lambda$}\\
        \overline{\tau}^{-1}(\Imods(\FiltCogen(\tau M \amalg \nu P\amalg \overline{\tau}_{\J(U)} W)/\add(\tau M\amalg \nu P))) & \text{if $W \in \P(\J(U))[1]$}
    \end{cases}$$
\end{theorem}

\begin{proof}
    Suppose first that $W \in \mods\Lambda$ and denote $N:= \Fbm_U(W)$. We recall from Theorem~\ref{thm:bmcorr} that $U \amalg N$ is a support $\tau$-rigid object which satisfies $W = f_M(N)$. Hence there is an exact sequence 
    $$0 \to X \to N \to W \to 0$$
    with $X$ in $\Gen M$. It follows that $\FiltGen(M\amalg W) =  \Gen(M\amalg N)$.
    Moreover, by Theorem~\ref{thm:Bongartz} we have that $M \amalg N \in \P(\Gen(M\amalg N))$, and Corollary~\ref{cor:genmin} says that $\Pmods(\Gen(M\amalg N)) = \Pmods(\FiltGen(M\amalg W))$ is a direct summand of $M \amalg N$. Thus we need only show that $N$ is a direct summand of $\Pmods(\FiltGen(M\amalg W))$. This follows from the fact that $N \notin \Gen M$ by construction (since otherwise we would have $W=f_M(N)=0$).

    Now suppose that $W \notin \mods\Lambda$. For readability, denote $Y := \overline{\tau}_{\J(U)}W\in \J(U)$. (Note also that $Y = \nu_{\J(U)}(W[-1])$.) By Remark~\ref{rem:twokinds2indec}, the object $\overline{\tau} U$ is support $\tau^{-1}$-rigid.
    We decompose this as $\overline{\tau}U = N \amalg I[-1]$ with $N \in \mods\Lambda$ and $I \in \I(\Lambda)$. (Explicitly, $N = \tau M \amalg \nu P$ and $I = \nu Q$ for $Q$ the maximal projective direct summand of $M$.) We denote the torsion functor associated to the torsion pair $({}^\perp N,\Cogen N)$ by $t_{N}$.
    
    By the dual of Theorem~\ref{thm:E_map} and the dual of Theorem~\ref{thm:bmcorr},
    there is an indecomposable module $Z$ such that $\overline{\tau}U \amalg Z$ is support $\tau^{-1}$-rigid and $Y 
    =t_{N}  Z$. Hence there is an exact sequence
    $$0 \rightarrow Y \rightarrow Z \rightarrow X \rightarrow 0$$
    with $X \in \Cogen N$. It follows that $\FiltCogen(N \amalg Y) = \Cogen(N \amalg Z)$. Moreover, by Theorem~\ref{thm:dualBongartz} we have that $N \amalg Z \in \I(\Cogen(N \amalg Z))$, and Corollary~\ref{cor:genmin} says that $\Imods(\Cogen(N \amalg Z)) = \Imods(\FiltCogen(N \amalg Y))$ is a direct summand of $N \amalg Z$. Now note that $Z\notin \Cogen N$ by construction (since otherwise we would have $Y = t_{N}(Z) = 0$),
    and so $Z$ is a direct summand of $\Imods(\FiltCogen(N \amalg Y))$. By Lemma~\ref{lem:twokinds2}, we have that $\overline{\tau}^{-1}Z$ is an indecomposable object with $U \amalg \overline{\tau}^{-1}Z$ support $\tau$-rigid. That is, $\overline{\tau}^{-1}Z$ is in the image of the map $\Fbm_U$.
    
    We then compute
    \begin{align*}
        \J_{\J(U)}(\Ebm_U(\overline{\tau}^{-1}Z)) &= \J(U \amalg \overline{\tau}^{-1} Z) & \text{by Theorem~\ref{thm:J_E}}\\
        &= \Jinv(Z \amalg \overline{\tau} U) & \text{by Lemma~\ref{lem:Jtauminus}}\\
        &= \Jinv_{\Jinv(\overline{\tau}U)}(Y) & \text{by the dual of Theorem~\ref{thm:bmcorr} and of Theorem~\ref{thm:J_E}}\\
        &= \J_{\J(U)}(W) & \text{by Lemma~\ref{lem:Jtauminus}.}
    \end{align*}
    It thus follows from~\cite[Prop.~6.15(2)]{bh} that either $\Ebm_U(\overline{\tau}^{-1}Z) = W$ or $\Ebm_U(\overline{\tau}^{-1}Z) = W[-1]$. Thus to prove the result it suffices to show that $\Ebm_U(\overline{\tau}^{-1}Z) \notin \mods\Lambda$. To see this, first note that $\nu^{-1}I \amalg \tau^{-1}N = \nu^{-1}\nu Q \amalg \tau^{-1}\tau M = M$. By the dual of Lemma~\ref{lem:fUbijection} and the fact that $Y = t_{N}Z \in \I(\J(U)) = \I(\Jinv(\overline{\tau}U))$, we then have that $Z \in \I({}^\perp I \cap (\tau^{-1}N)^\perp) = \I(M^\perp)$. Now if $Z \in \I(\Lambda)$, then $\overline{\tau}^{-1}Z \notin \mods\Lambda$. Otherwise $\overline{\tau}^{-1}Z = \tau^{-1}Z \in \Gen M$ by Proposition~\ref{prop:tauprojectiveinjective} applied to the torsion pair $(\Gen M,M^{\perp})$. Thus in either case, $\Ebm_U(\overline{\tau}^{-1}Z) \notin \mods\Lambda$ by the definition of $\Ebm$ (see Theorem~\ref{thm:bmcorr}).
\end{proof}
	
\subsection{Sincere $\tau$-perpendicular subcategories}

Recall that a subcategory $\SC$ of $\mods \Lambda$ is said to be {\em sincere} if $\Hom(P,-)|_{\SC} \neq 0$ for any non-zero projective module $P$. The sincere left/right finite $\tau$-perpendicular categories have particularly nice descriptions, which are crucial for our main results.

	\begin{lemma}
	\label{lem:nosummand}
Let $M$ be a $\tau$-rigid module and suppose that $\J(M)$ is sincere.
 \begin{itemize}
		\item[(a)] Suppose $\J(M)$ is right finite. Then $\Pmods({}^\perp \J(M))$ has no projective direct summands and $\Imodns(\FiltCogen(\J(M)))$ has no injective direct summands. In particular, $\tau \Pmods({}^\perp \J(M)) = \Imodns(\FiltCogen(\J(M)))$ and $\Pmods({}^\perp \J(M)) = \tau^{-1}\Imodns(\FiltCogen(\J(M)))$.
		\item[(b)] Suppose $\J(M)$ is left finite. Then $\Pmodns(\FiltGen(\J(M)))$ has no projective direct summands and $I_{s}(\J(M)^\perp)$ has no injective direct summands. In particular, $\tau \Pmodns(\FiltGen(\J(M))) = \Imods(\J(M)^\perp)$ and $\Pmodns(\FiltGen(\J(M))) = \tau^{-1} \Imods(\J(M)^\perp)$.
	\end{itemize}
\end{lemma}

\begin{proof}
  
(a)
Since $\J(M)$ is right finite, $\FiltCogen(\J(M))$ is functorially finite, and hence so is ${}^\perp\J(M)$ \sloppy by Definition-Proposition~\ref{defprop:torsionpairs}(a) and~Proposition~\ref{prop:asff}.
In particular, by applying 
Lemma~\ref{lem:perp_injectives_2}(a) to the torsion class ${}^{\perp}\J(M)=\Gen \Pmods({}^{\perp}(\J(M)))$ 
we obtain
$$
\FiltCogen(\J(M))=({}^{\perp}\J(M))^{\perp}=(\Gen P_s({}^{\perp} \J(M)))^{\perp}=P_s({}^{\perp}\J(M))^{\perp}.
$$

Since $\J(M)$ is sincere, $\FiltCogen(\J(M))$ is also sincere. Hence ${}^\perp\J(M) = {}^\perp\FiltCogen(\J(M))$ contains no projectives, and so in particular $\Pmods({}^\perp\J(M))$ has no projective direct summands. Moreover, since $\Imodns(
   \FiltCogen(\J(M)))$ does not contain any split injectives in $\FiltCogen(\J(M))$ as a direct summand, it cannot contain any modules injective in $\mods \Lambda$ as a direct summand. The equalities then follow from Corollary~\ref{cor:proj_inj}(b).

   (b)
Since $\J(M)$ is left finite, we have that  $\FiltGen(\J(M))$ and $\J(M)^\perp$ are both functorially finite. In particular, $\FiltGen(\J(M)) = {}^\perp \Imods(\J(M)^{\perp})$
 by Lemma~\ref{lem:perp_injectives_2}(f) applied to the torsion pair $(\FiltGen(\J(M)),\J(M)^\perp)$.
 Since $\J(M)$ is sincere, $\FiltGen(\J(M))$ is also sincere. Hence $\J(M)^\perp = \FiltGen(\J(M))^\perp$ contains no injectives, and so in particular $\Imods(\J(M)^\perp)$ has no injective direct summands. Moreover, since $\Pmodns(
   \FiltGen(\J(M)))$ does not contain any split projectives in $\FiltGen(\J(M))$ as a direct summand, it cannot contain any modules projective in $\mods \Lambda$ as a direct summand. The equalities then follow from Corollary~\ref{cor:proj_inj}(a).\end{proof}

\begin{lemma}
	\label{lem:Jsame}
	Let $M$ be a $\tau$-rigid module and suppose that $\J(M)$ is sincere. 
	\begin{itemize}
		\item[(a)] Assume $\J(M)$ is right finite. Then $\J(\Pmods({}^\perp\J(M)))=\Jinv(\Imodns (\FiltCogen(\J(M))))=\J(M)$;
		\item[(b)] Assume $\J(M)$ is left finite. Then $\J(\Pmodns(\FiltGen(\J(M))))=\Jinv(\Imods(\J(M)^\perp)) =\J(M)$.
	\end{itemize}
\end{lemma}

\begin{proof}
    (a) Since $\J(M)$ is right finite, we have that ${}^\perp \J(M)$ is functorially finite by Definition-Proposition~\ref{defprop:torsionpairs} and Proposition~\ref{prop:asff}. Thus ${}^\perp \J(M) = {}^\perp \J(\Pmods({}^\perp \J(M)))$ by Proposition~\ref{prop:gen_minimal}. This implies that 
    $$ \FiltCogen(\J(M)) = ({}^\perp \J(M))^\perp  = ({}^\perp \J(\Pmods({}^\perp \J(M))))^{\perp} = \FiltCogen(\J(\Pmods({}^\perp \J(M)))).$$
    Thus $\J(M) = \J(\Pmods({}^\perp \J(M)))$ by Theorem~\ref{thm:ms}. Furthermore, using the assumption that $\J(M)$ is sincere, it follows from Theorem~\ref{thm:WLtorsionclass} that $\J(\Pmods({}^\perp \J(M))) = \Jinv(\Imodns(\FiltCogen(\J(M))))$.

    (b)
    Since $\J(M)$ is left finite, we have that $\J(M)^\perp$ is functorially finite by Definition-Proposition~\ref{defprop:torsionpairs} and Proposition~\ref{prop:asff}. Thus $\J(M)^\perp = \Jinv(\Imods(\J(M)^\perp))^\perp$ by Proposition~\ref{prop:gen_minimal}.
    Hence, $${}^\perp(\J(M)^\perp) = \FiltGen(\J(M)) = \FiltGen(\Jinv(\Imods(\J(M)^\perp))).$$ Thus $\J(M) = \Jinv(\Imods(\J(M)^\perp)) = J(\Pmodns(\FiltGen(\J(M))))$ by Theorem~\ref{thm:ms}, Theorem~\ref{thm:WLtorsionclass} and the assumption that $\J(M)$ is sincere.
\end{proof}
	
\section{\texorpdfstring{$\tau$}{Tau}-exceptional pairs and sequences}\label{sec:sequences}

\subsection{Definitions and background}

We recalled the definition of $\tau$-exceptional pairs and sequences in the introduction.
A more general concept was introduced in~\cite{bm}, motivated by a similar definition in the hereditary case~\cite{igto}. Recall the notation $\C(\W)$ and the discussion of support $\tau_{\W}$-rigid objects from Definition~\ref{def:support_tau_rigid_reduction} and Remark~\ref{rem:CW}.

\begin{definition}\label{def:signed_sequences}
A sequence $\mathcal{U} = (U_s,U_{s+1},\ldots,U_n)$  of indecomposable objects in 
$\C(\Lambda)$ is called a \emph{signed $\tau$-exceptional sequence}
if $U_n$ is support $\tau$-rigid, and 
$(U_s,U_{s+1},\ldots,U_{n-1})$ is a signed $\tau_{\J(U_n)}$-exceptional sequence.
We say that $\mathcal{U}$ is \emph{complete} if $s = 1$ and that $\mathcal{U}$ is a \emph{signed $\tau$-exceptional pair} if $s = n-1$.
\end{definition}

Using the bijections described in Section \ref{sec:wide}, it was proved in~\cite{bm} that signed
$\tau$-exceptional sequences are in bijection with ordered support $\tau$-rigid objects. Restricting this bijection to (unsigned) $\tau$-exceptional sequences requires the following definition.

\begin{definition}\cite[Def.~3.1]{mt}\label{def:tfo}
    Let $\mathcal{N} = (N_s,N_{s+1},\ldots,N_n)$ be a sequence of indecomposable modules. We say that $\mathcal{N}$ is a \emph{TF-ordered $\tau$-rigid module} if (a) $\amalg_{i = s}^n N_i$ is $\tau$-rigid, and (b) $N_i \notin \Gen(\amalg_{i < j} N_j)$ for all $s \leq i < n$. 
\end{definition}

\begin{theorem}\label{thm:tfo}\cite[Thm.~5.1]{mt} Let $\mathcal{N} = (N_s, N_{s+1}, \dots, N_n)$ be a TF-ordered $\tau$-rigid module. Let $M_n = N_n$, and let $M_i = f_{N_{i+1} \amalg \dots \amalg N_n} N_i$ for $i=s, \dots, n-1$.
Then $\mathcal{M} := (M_s, M_{s+1}, \dots, M_{n})$ is a $\tau$-exceptional sequence. Moreover, the association $\omega: \mathcal{N} \mapsto \mathcal{M}$ is a bijection between the set of TF-ordered $\tau$-rigid modules (of length $n-s+1$) and the set of $\tau$-exceptional sequences (of length $n - s + 1$). 
\end{theorem}

\begin{remark}
\label{rem:orderedpair}
Note that $f_{N_{i+1} \amalg \dots \amalg N_n} N_i = \Ebm_{N_{i+1} \amalg \dots \amalg N_n} N_i$ in the setting of Theorem~\ref{thm:tfo}.
Let $(M,N)$ be a $\tau$-exceptional pair. Then $\Fbm_N(M)\amalg N$ is a $\tau$-rigid module by Theorem~\ref{thm:bmcorr}. 
In particular, Theorem~\ref{thm:tfo} implies that 
the pair $(\Fbm_N(M),N)$ is a TF-ordered $\tau$-rigid module.
\end{remark}

The following results will also be useful.

\begin{theorem}\label{thm:unique}\cite[Thm.~8]{ht}
    Let $(M_1,\ldots,M_n)$ and $(N_1,\ldots,N_n)$ be complete $\tau$-exceptional sequences. If there exists $i \in \{1,\ldots,n\}$ such that $M_j = N_j$ for all $j \neq i$, then also $M_i = N_i$.
\end{theorem}

Recall that complete $\tau$-exceptional sequences exist for any finite-dimensional algebra (as observed in the introduction to~\cite{bmpreprint}). This can be seen, for example, using a straightforward induction argument, noting that the projective modules are always $\tau$-rigid.

\begin{corollary}\label{cor:unique}
    Let $(M,N)$ and $(M',N')$ be $\tau$-exceptional pairs with $\J(M,N) = \J(M',N')$. Then $M = M'$ if and only if $N = N'$.
\end{corollary}

\begin{proof}
    Let $(M_1,\ldots,M_{n-2})$ be a complete $\tau_{\J(M,N)}$-exceptional sequence.
    The result then follows from applying Theorem~\ref{thm:unique} to the complete $\tau$-exceptional sequences $(M_1,\ldots,M_{n-2},M,N)$ and $(M_1,\ldots,M_{n-2},M',N')$.
\end{proof}

We will also use the following result from
\cite{bmpreprint},
which was crucial
for defining a mutation operation on signed $\tau$-exceptional pairs and sequences. The operation considered there does not, however, preserve (unsigned) 
$\tau$-exceptional pairs.

\begin{theorem}\label{thm:signed_mutation}\cite[Thm.~3.4 and Prop.~3.6]{bmpreprint}
    Let $(M,N)$ be a signed $\tau$-exceptional pair. Then $(\Ebm_{\Fbm_N(M)}(N),\Fbm_N(M))$ is also a signed $\tau$-exceptional pair and $\J(M,N) = \J(\Ebm_{\Fbm_N(M)}(N),\Fbm_N(M)).$
\end{theorem}

\subsection{Sincere $\tau$-perpendicular categories}

We will need criteria for when $\tau$-perpendicular categories are 
sincere. In the corank one case, we have the following straightforward criterion.

\begin{lemma}
		\label{lem:sincerecorank1}
		Let $M\in\mods \Lambda$ be an indecomposable $\tau$-rigid module. Then $\J(M)$ is sincere in $\mods \Lambda$ if and only if $M$ is non-projective.
	\end{lemma}

 \begin{proof}
    It is clear that if $M$ is projective then $\J(M)=M^{\perp}$ is not sincere. Thus suppose that $\J(M)$ is not sincere; i.e., that $\J(M)\subseteq P^{\perp}$ for some indecomposable projective module $P$. Since $M$ and $P$ are indecomposable, $\J(M)$ and $\J(P)=P^{\perp}$ are both rank $n-1$ module categories by Definition-Proposition~\ref{prop:J_def}. 
    Moreover, by Theorem~\ref{thm:ms}, we have $\WR(\FiltCogen (\J(M)))=\J(M)$. Since there is a torsion pair $(\Gen M, M^{\perp})$, we have, by Theorem~\ref{thm:WLtorsionclass}(b) and Corollary~\ref{cor:genmin}, that $\WR(M^{\perp})=\J(\Pmods(\Gen M))=\J(M)$, since $M$ is gen-minimal. Hence, by Theorem~\ref{thm:ms}, noting that $\Gen M$ is functorially finite, we have $\FiltCogen(\J(M))=M^{\perp}$.
   In particular, $\FiltCogen(\J(M)) \subseteq P^\perp$ is functorially finite, and so $\J(M)$ is a right finite wide subcategory of $P^\perp$. Thus, by~\cite[Lem.\ 4.3]{bh}, there exists $N \in P^\perp$ which is $\tau_{\J(P)}$-rigid and satisfies $\J(M) = \J(N,P)$. Then $N = 0$, or equivalently $J(M)= P^\perp = \J(P)$, by Definition-Proposition~\ref{prop:J_def}. We conclude that $M=P$ by~\cite[Prop.\ 6.15]{bh}.
\end{proof}

We consider now the iterated $\tau$-perpendicular category $\J(M,N)$.
	
	\begin{proposition}
		\label{prop:JJCBsincere}
		Let $(M,N)$ be a $\tau$-exceptional pair in $\mods \Lambda$ satisfying $N\not\in \P(\Lambda)$ and $M\not\in \P(\J(N))$. Then $\J(M,N)$ is sincere in $\mods \Lambda$.
	\end{proposition}

 \begin{proof}
    Let $\mathcal{S}_N$ be a complete set of isomorphism-class representatives of the objects of $\mods \Lambda$ which are simple in $\J(N)$.
    Note that this is a finite set, since 
    $\J(N)$ is equivalent to the module category of a finite-dimensional algebra by Definition-Proposition~\ref{prop:J_def}. 
    Since $\J(M,N)$ is sincere in $\J(N)$ (by Lemma~\ref{lem:sincerecorank1}), every object of $\mathcal{S}_N$ appears as a $\J(N)$-composition factor of the direct sum, $T$, of the simple objects in $\J(M,N)$. Similarly, since $\J(N)$ is sincere in $\mods\Lambda$, every simple object of $\mods\Lambda$ appears as a composition factor of $\bigoplus_{S \in \mathcal{S}_N} S$. It follows that every simple object of $\mods\Lambda$ is a composition factor of $T$, and therefore that $\J(M,N)$ is sincere in $\mods\Lambda$. 
\end{proof}

We now interpret and summarise 
the key results in the following proposition, in terms of $\tau$-exceptional pairs. 

\begin{proposition}\label{prop:sincere_summary_2}
    Let $(M,N)$ be a $\tau$-exceptional pair, let $L = \Fbm_N(M )\amalg N$, and suppose that $\J(L)$($=\J(M,N)$) is sincere.
    
    If $\J(L)$ is right finite, then:
        \begin{enumerate}
            \item[(a)] $\Pmods({}^\perp\J(L)) = \tau^{-1} \Imodns(\FiltCogen(\J(L)))$.
            \item[(b)] ${}^\perp\J(L) = \Gen \Pmods({}^\perp \J(L))$.
            \item[(c)] $\FiltCogen(\J(L)) = \Pmods({}^\perp\J(L))^\perp$.
            \item[(d)] $\J(\Pmods({}^\perp \J(L))) = \J(L)$.
        \end{enumerate}
Dually, if $\J(L)$ is left finite, then:
        \begin{enumerate}
            \item[(a')] $\tau \Pmodns(\FiltGen(\J(L))) = \Imods(\J(L)^\perp)$.
            \item[(b')] $\J(L)^\perp = \Cogen \Imods(\J(L)^\perp)$.
            \item[(c')] $\FiltGen(\J(L)) = {}^\perp \Imods(\J(L)^\perp)$.
            \item[(d')] $\J(\Pmodns(\FiltGen(\J(L)))) = \J(L)$.
    \end{enumerate}
\end{proposition}

\begin{proof}
Recall that $L$ is a $\tau$-rigid module (see Remark~\ref{rem:orderedpair}).
Note that $\J(L) = \J(\Fbm_N(M)\amalg N) = \J(M,N)$ by Theorem~\ref{thm:J_E}.

    Suppose first that $\J(L)$ is right finite and consider the torsion pair $(\T,\F) := ({}^\perp \J(L), \FiltCogen(\J(L)))$ (see Definition-Proposition~\ref{defprop:torsionpairs}). Note that $\T$ and $\F$ are both functorially finite by the assumption that $\J(L)$ is right finite and Proposition~\ref{prop:asff}. Moreover, $\F$ is sincere because $\J(L)$ is sincere.
    Items (a)-(d) are then specialisations of previous results as follows.
    \begin{itemize}
        \item Item (a) is contained in Lemma~\ref{lem:nosummand}(a).
        \item Item (b) follows from Lemma~\ref{lem:perp_injectives_2}. 
        \item Item (c) follows from applying the operator $(-)^\perp$ to item (b).
        \item Item (d) follows from Lemma~\ref{lem:Jsame}(a).
    \end{itemize}
    Now suppose that $\J(L)$ is left finite and consider the torsion pair $(\T',\F') := (\FiltGen(\J(L)), \J(L)^\perp)$. Note that $\T'$ and $\F'$ are both functorially finite by the assumption that $\J(L)$ is left finite. Moreover, $\T'$ is sincere because $\J(L)$ is.
    Items (a')-(d') are then specialisations of previous results as follows.
    \begin{itemize}
        \item Item (a') is contained in Lemma~\ref{lem:nosummand}(b).
    \item Item (b') follows from the dual of Lemma~\ref{lem:perp_injectives_2}.
    \item Item (c') follows from applying the operator ${}^\perp(-)$ to item (b').
        
        \item Item (d') follows from Lemma~\ref{lem:Jsame}(b).\qedhere
    \end{itemize}
\end{proof}


\subsection{Irregular $\tau$-exceptional pairs}	

As we will see in Section~\ref{sec:pair_mutation}, the notions of mutation for $\tau$-exceptional pairs that we introduce do not admit clear case-free definitions. More precisely, we split the definitions of left and right mutation into ``regular'' cases and ``irregular'' cases, with the irregular cases occurring only under very specific assumptions. In this section, we prepare the necessary definitions and background related to the irregular cases.

\begin{definition} \label{def:regular}
\begin{enumerate}
    \item[(a)] A $\tau$-exceptional pair $(B,C)$ with $C \not \in \P({}^\perp{\tau \Fbm_C(B)})$ or $C \in  \P(\Lambda)$ is called {\em left regular}.
Otherwise, it is called \emph{left irregular}.
\item[(b)] A $\tau$-exceptional pair $(X,Y)$ with $\Fbm_Y(X) \in \P({}^\perp{\tau Y})$ or $Y \not \in \Gen \Fbm_Y(X)$ is called {\em right regular}. Otherwise it is called {\em right irregular}.
\end{enumerate}
 \end{definition}

Note that a $\tau$-exceptional pair $(B,C)$ is left irregular if and only if $C$ is a non-projective summand in the Bongartz complement of $\Fbm_C(B)$, while a $\tau$-exceptional pair $(X,Y)$ is right irregular if and only if $Y$ is a summand in the co-Bongartz complement of $\Fbm_Y(X)$ and $\Fbm_Y(X)$ is not a summand in the Bongartz complement of $Y$.

We have the following.
		
\begin{lemma}\label{lem:regular}
Any $\tau$-exceptional pair $(B,C)$  with 
$C \in \Gen \Fbm_C(B)$ is left regular.
\end{lemma}

\begin{proof}
Note that $C\in {}^{\perp} \tau \Fbm_C(B)$ (see Remark~\ref{rem:orderedpair}).
If $\Fbm_C(B)$ is not projective, then we have that
\begin{equation}
    \label{eq:Psperp}
    \P(^{\perp}\tau \Fbm_C(B)) \setminus \add \Fbm_C(B) = \Ps(^{\perp}\tau \Fbm_C(B))
\end{equation}
by Lemma \ref{lem:Bongartzsplit}(b).
Moreover, if $\Fbm_C(B)$ is projective, we have
\begin{equation}
    \label{eq:Psperp2}
    \P(^{\perp}\tau \Fbm_C(B)) = \Ps(^{\perp}\tau \Fbm_C(B)) =\mods \Lambda
\end{equation}
If $C\in \Gen \Fbm_C(B)$, then $C\not\in \Ps(^{\perp}\tau \Fbm_C(B))$, and hence $C\not\in \P(^{\perp}\tau \Fbm_C(B))$ by~\eqref{eq:Psperp} or
~\eqref{eq:Psperp2}, so the pair $(B,C)$ is left regular.
\end{proof}

We now state two technical results which will be used in Section~\ref{sec:pair_mutation} to define mutations of (left or right) irregular pairs.

\begin{proposition}
		\label{prop:JUsinceresetup1}
		Assume $(B,C)$ is a left irregular $\tau$-exceptional pair and denote $L =\Fbm_C(B) \amalg C$.
  Then the following hold:

  \begin{itemize}
      \item[(a)] $L$ is $\tau$-rigid and has no projective direct summands.
      \item[(b)] $\J(L)$($= \J(B,C)$) is sincere.
      \item[(c)] $L$ is gen-minimal.
      \item[(d)] ${}^\perp \J(L)$ is functorially finite, and thus $\J(L)$ is right finite.
	  \item[(e)] $L= \Pmods({}^\perp \J(L)) = \tau^{-1}  \Imodns(\FiltCogen(\J(L)))$.
  \end{itemize}

  Assume in addition that $\J(L)$ is left finite. Then:
   \begin{itemize}
    \item[(f)]  $\J(\Pmodns(\FiltGen(\J(L)))) = \J(L)$.
       \item[(g)]  $\Pmodns(\FiltGen(\J(L)))$ is not gen-minimal and decomposes into $\Pmods(\Gen \Pmodns(\FiltGen(\J(L)))) \amalg Y$, with both direct summands indecomposable. 
        \end{itemize}
	\end{proposition}

\begin{proof}
(a) The module $L$ is $\tau$-rigid by Theorem \ref{thm:bmcorr}, and $C$ is non-projective by the irregularity assumption. By Lemma \ref{lem:eachinperp}, since we have that $C \in \P({}^\perp{\tau \Fbm_C(B)})$, we have that 
\begin{equation}\label{eq:notinBc}
\Fbm_C(B) \not \in \P({}^\perp{\tau C})    
\end{equation}
(noting that $C\not=\Fbm_C(B)$)
and hence $\Fbm_C(B)$ is not projective.

(b) By \eqref{eq:notinBc} it also follows that $B \not \in \P(\J(C))$, using Lemma~\ref{lem:fUbijection}. 
Then it follows from  Theorem~\ref{thm:J_E} that $\J(L)= \J(B,C)$, and from Proposition~\ref{prop:JJCBsincere} that this is a sincere subcategory.

(c) Since $(\Fbm_C(B),C)$ is $TF$-ordered, we have $\Fbm_C(B)\not\in \Gen C$.  Moreover, $C \notin \Gen\Fbm_C(B)$ by Lemma~\ref{lem:regular}. Hence $L$ is gen-minimal.

(d) This follows from item (c) and Corollary~\ref{cor:gen_minimal}.

(e) Note that, by item (c) and Theorem~\ref{thm:gen_minimal}, we have
    ${}^{\perp}\J(L)=\Gen L$. Since $L$ is gen-minimal, $L=\Pmods(\Gen L)=\Pmods({}^\perp \J(L))$, using Corollary~\ref{cor:genmin}(a).
The final equality follows from Proposition~\ref{prop:sincere_summary_2}(a).

(f) follows from (b) and Proposition \ref{prop:sincere_summary_2}(d').

(g) By items (c) and (e), we have that $\Pmods({}^\perp \J(L)) = L$ is gen-minimal.
    Suppose, for a contradiction, that $\Pmodns(\FiltGen(\J(L)))$ is also gen-minimal.
    We first show that then $\Pmods({}^\perp \J(L))= \Pmodns(\FiltGen(\J(L)))$. 
We have \begin{align*}
\Gen \Pmods({}^\perp \J(L)) & ={}^{\perp}\J(\Pmods({}^\perp\J(L))) & \text{by Theorem~\ref{thm:gen_minimal}} \\
&= {}^{\perp}\J(L) & \text{by item (e)} \\
&={}^{\perp}\J(\Pmodns(\FiltGen(\J(L)))) & \text{by item (f)} \\
&=\Gen \Pmodns(\FiltGen(\J(L))) & \text{by Theorem~\ref{thm:gen_minimal}}.
\end{align*}

    Hence, since both $\Pmods({}^\perp \J(L))$ and $\Pmodns(\FiltGen(\J(L)))$ are $\tau$-rigid and gen-minimal, they are equal by Theorem~\ref{thm:Bongartz}(b). In particular, this means that by item (e) we have: 
    \begin{equation}\label{eq:Uisomorphism}
    L = \Pmodns(\FiltGen(\J(L)))    
    \end{equation}

By the definition of left irregularity, we have $C\in \P({}^{\perp}\tau \Fbm_C(B))\setminus \{\Fbm_C(B)\}$ (noting that $C\not=\Fbm_C(B)$ by Theorem~\ref{thm:bmcorr}), so $C$ is a direct summand of the Bongartz complement $B_{\Fbm_C(B)}$ of $\Fbm_C(B)$. Hence, using this and the characterisation of the Bongartz complement (see Theorem~\ref{thm:Bongartz}):
\begin{align*}
    {}^{\perp} \tau (C\amalg \Fbm_C(B)) &\subseteq {}^{\perp} \tau \Fbm_C(B) \\
    &={}^{\perp}\tau (\Fbm_C(B)\amalg B_{\Fbm_C(B)}) \\
    &\subseteq {}^{\perp}\tau (C\amalg \Fbm_C(B)).
\end{align*}
It follows that
\begin{equation}\label{eq:equality}
{}^{\perp} \tau (C\amalg \Fbm_C(B)) = {}^{\perp} \tau \Fbm_C(B)
\end{equation}
Since $C$ is a direct summand of $B_{\Fbm_C(B)}$, it is split projective in ${}^{\perp}\tau \Fbm_C(B)$ by Lemma~\ref{lem:Bongartzsplit}, i.e. it is a summand of $P_s({}^{\perp}\tau \Fbm_C(B))$. We then have (since $\J(L)$ is assumed to be left finite):
\begin{align*}
\FiltGen(\J(L))
&= {}^{\perp}\tau \Pmodns(\FiltGen(\J(L))) & \text{by item (b) and Lemma~\ref{lem:perp_injectives_2}} \\
&= {}^{\perp} \tau L &  \text{by Equation \eqref{eq:Uisomorphism}} \\
&= {}^{\perp} \tau (\Fbm_C(B)\amalg C) & \text{by the definition of $L$} \\
&= {}^{\perp} \tau \Fbm_C(B) & \text{by Equation \eqref{eq:equality}}
\end{align*}
Hence, $C$ is a direct summand of $P_s(\FiltGen(\J(L)))$.
However, using~\eqref{eq:Uisomorphism}, we have that $C$ is also a direct summand of $L= \Pmodns(\FiltGen(\J(L)))$,
giving a contradiction. Hence,  $\Pmodns(\FiltGen(\J(L)))$ is not gen-minimal, as required.

Now recall that $\J(L) = \J( \Pmodns(\FiltGen(\J(L))))$ by Proposition~\ref{prop:sincere_summary_2}(d'). Thus $ \Pmodns(\FiltGen(\J(L)))$ has exactly two indecomposable direct summands by Definition-Proposition~\ref{prop:J_def}.
Moreover, by Corollary~\ref{cor:genmin}, $\Pmods(\Gen \Pmodns(\FiltGen(\J(L))))$ is a nonzero direct summand of  $\Pmodns(\FiltGen(\J(L)))$. Since $ \Pmodns(\FiltGen(\J(L)))$ is not gen-minimal, we thus obtain the desired direct sum decomposition.
\end{proof}

  \begin{proposition}
		\label{prop:JUsinceresetup2}
		Assume $(X,Y)$ is a right irregular $\tau$-exceptional pair, with $\widetilde{L} =\Fbm_Y(X) \amalg Y$.
  Then the following hold:

  \begin{itemize}
      \item[(a)] $\widetilde{L}$ is $\tau$-rigid and has no projective direct summands. 
      \item[(b)] $\J(\widetilde{L})$($= \J(X,Y)$) is sincere.
      \item[(c)] $\tau \widetilde{L}$ is $\tau^{-1}$-rigid and cogen-minimal.
      \item[(d)] $\J(\widetilde{L})^\perp$ is functorially finite, and thus $\J(\widetilde{L})$ is left finite.
	  \item[(e)] $\tau \widetilde{L} = \tau \Pmodns(\FiltGen(\J(\widetilde{L}))) = \Imods(\J(\widetilde{L})^\perp)$.
  \end{itemize}
  Assume in addition that $\J(\widetilde{L})$ is right finite. Then:
   \begin{itemize}
    \item[(f)]  $\Jinv(\Imodns(\FiltCogen(\J(\widetilde{L})))) = \J(P_{s}({}^\perp \J(\widetilde{L}))) = \J(\widetilde{L})$. 
       \item[(g)] 
$\Imodns(\FiltCogen(\J(\widetilde{L})))$ is not cogen-minimal and can be written
       $$\Imodns(\FiltCogen(\J(\widetilde{L}))) = \Imods(\Cogen \Imodns(\FiltCogen(\J(\widetilde{L})))) \amalg \tau C,$$
       where both direct summands are indecomposable and non-injective.
       Applying $\tau^{-1}$ to this equality gives
       $$P_{s}({}^\perp \J(\widetilde{L})) = \tau^{-1} \Imods(\Cogen \Imodns(\FiltCogen(\J(\widetilde{L})))) \amalg C,$$ with both direct summands indecomposable and non-projective.
        \end{itemize}
	\end{proposition}

\begin{proof}
    (a) 
    The module $\widetilde{L}$ is $\tau$-rigid by Theorem \ref{thm:bmcorr}.
    Since $(X,Y)$ is right irregular, we have $\Fbm_Y(X)\not\in \P({}^{\perp}(\tau Y)$ and $Y\in \Gen \Fbm_Y(X)$. Hence, neither $\Fbm_Y(X)$ nor $Y$ is projective.

    (b) We have $Y \notin \P(\Lambda)$ by (a).
    Since $\Fbm_Y(X)\not\in \P({}^{\perp}(\tau Y))$, we have $X \notin \P(\J(Y))$ by Lemma~\ref{lem:fUbijection}. Then it follows from Proposition~\ref{prop:JJCBsincere} and Theorem~\ref{thm:J_E} that $\J(\widetilde{L}) = \J(X,Y)$ is sincere.

(c) The module $\tau\widetilde{L}$ is 
$\tau^{-1}$-rigid by Lemma \ref{lem:twokinds2}.
By the irregularity assumption, we have $\Fbm_Y(X) \notin \P({}^\perp \tau Y)$. By Theorem \ref{thm:Bongartz}(c), we have that 
    $\Fbm_Y(X) \in \P({}^\perp \tau \widetilde{L})$, since $\Fbm_Y(X)$ is a direct summand in $\widetilde{L}$.
    It follows that 
    \begin{equation}\label{eq:diffPerp1}
    {}^\perp \tau \widetilde{L} \neq {}^\perp \tau Y.
    \end{equation}

Next note that, since $\Fbm_Y(X)$ is indecomposable non-projective (by item (a)) and $\tau$-rigid, we have by Theorem~\ref{thm:Bongartz}(c) and Lemma~\ref{lem:Bongartzsplit}(b) that 
\begin{equation}\label{eq:split}P({}^\perp \tau \Fbm_Y(X)) = \Pmods({}^\perp \tau \Fbm_Y(X)) \amalg \Fbm_Y(X).\end{equation} 
Now, using again the irregularity assumption, we have $Y \in \Gen \Fbm_Y(X)$. Recalling from Theorem~\ref{thm:Bongartz}(c) that $\Fbm_Y(X) \in \P({}^\perp \tau \Fbm_Y(X))$, this implies that
$Y \notin \P_s({}^\perp \tau \Fbm_Y(X))$. Since $Y$ is indecomposable and $Y \neq \Fbm_Y(X)$, it then follows from~\eqref{eq:split} that $Y \notin \P({}^\perp \tau \Fbm_Y(X))$. 

Since $Y$ is a direct summand of $\widetilde{L}$, we have $Y \in \P({}^\perp \tau \widetilde{L})$ using Theorem~\ref{thm:Bongartz}(c),
and it follows that 
\begin{equation}\label{eq:diffPerp2}
{}^\perp \tau \widetilde{L} \neq {}^\perp \tau \Fbm_Y(X).
    \end{equation}
Now, let $Z \in \{\widetilde{L}, Y, \Fbm_Y(X)\}$.
By Theorem~\ref{thm:Bongartz}(a) we have that ${}^\perp \tau Z$ is a torsion class. 
We claim that the corresponding torsion-free class $(^\perp \tau Z)^{\perp}$ is equal to $\Cogen(\tau Z)$.
By Definition-Proposition~\ref{defprop:torsionpairs} we have $(^\perp \tau Z)^{\perp}= \Filt(\Cogen \tau Z)$. 
Since $Z$ is $\tau$-rigid,  
$\tau Z$ is $\tau^{-1}$-rigid by Lemma \ref{lem:twokinds2}. Hence, $\Cogen \tau Z$ is closed under extensions, and it follows that $(^\perp \tau Z)^{\perp} = \Filt(\Cogen \tau Z) = \Cogen \tau Z$.

Combining this now with \eqref{eq:diffPerp1} and \eqref{eq:diffPerp2}, we obtain that $\Cogen \tau \widetilde{L}$ must properly contain both $\Cogen \tau Y$ and $\Cogen \tau \Fbm_Y(X)$. Since $\widetilde{L} = \Fbm_Y(X) \amalg Y$, it follows that 
$\tau\widetilde{L}$ is cogen-minimal.

    (d) Note that $\J(\widetilde{L}) = \Jinv(\tau \widetilde{L})$ by (a) and Lemma~\ref{lem:Jtauminus}. The result thus follows from item (c) and Corollary~\ref{cor:gen_minimal}.

    (e) By item (c) and Theorem~\ref{thm:gen_minimal}, we have $\Jinv(\tau \widetilde{L})^\perp =  \Cogen \tau \widetilde{L}$. Since $\tau \widetilde{L}$ is cogen-minimal, we have by Theorem~\ref{thm:dualBongartz}(b) that $\tau\widetilde{L} = \Imods
    (\Cogen \tau \widetilde{L}) =\Imods(\Jinv(\tau\widetilde{L})^\perp)$. Now using that $\J(\widetilde{L})=\Jinv(\tau \widetilde{L})$ is sincere and left finite, the final isomorphism follows from Proposition~\ref{prop:sincere_summary_2}(a').

    (f) By Theorem~\ref{thm:ms}, we have that $\J(\widetilde{L}) = \WR(\FiltCogen(\J(\widetilde{L})))$. As we have assumed $\FiltCogen(\J(\widetilde{L}))$ is functorially finite, the result follows from Theorem~\ref{thm:WLtorsionclass} and item (b), noting the torsion pair $({}^{\perp}J(\widetilde{L}),\FiltCogen(\J(\widetilde{L})))$ (see Definition-Proposition~\ref{defprop:torsionpairs}).
    
    (g) Note that $\J(\widetilde{L}) = \Jinv(\tau \widetilde{L})$ by (a) and Lemma~\ref{lem:Jtauminus}.
    We thus write $\Jinv(\tau \widetilde{L})$ in place of $\J(\widetilde{L})$ throughout the proof. By items (c) and (e), we have that $\Imods(\Jinv(\tau\widetilde{L})^\perp) = \tau \widetilde{L}$ is cogen-minimal. Suppose, for a contradiction, that $\Imodns(\FiltCogen(\Jinv(\tau\widetilde{L})))$ is also cogen-minimal. We first show that then $
    \Imods(\Jinv(\tau\widetilde{L})^{\perp})= \Imodns(\FiltCogen(\Jinv(\tau\widetilde{L})))$. 

    We have
    \begin{align*}
        \Cogen \Imods(\Jinv(\tau\widetilde{L})^\perp) &= \Jinv(\Imods(\Jinv(\tau\widetilde{L})^\perp))^\perp & \text{by Theorem~\ref{thm:gen_minimal}}\\
        &= \Jinv(\tau\widetilde{L})^\perp & \text{by Lemma~\ref{lem:Jsame}}\\
        &= \Jinv(\Imodns(\FiltCogen(\Jinv(\tau \widetilde{L}))))^\perp & \text{by item (f)}\\
        &= \Cogen \Imodns(\FiltCogen(\Jinv(\tau\widetilde{L}))) & \text{by Theorem~\ref{thm:gen_minimal}}
    \end{align*}
    Hence, since both 
    $\Imods(\Jinv(\tau \widetilde{L})^{\perp})$ and $\Imodns(\FiltCogen(\Jinv(\tau\widetilde{L})))$ are $\tau^{-1}$-rigid and cogen-minimal, they are equal by Theorem~\ref{thm:dualBongartz}(b). In particular, by item (e), we have
    \begin{equation}\label{eqn:tau_U_tilde}
        \tau \widetilde{L} = \Imodns(\FiltCogen(\Jinv(\tau\widetilde{L}))).
    \end{equation}
    By the definition of right irregularity, we have $Y \in \Gen \Fbm_Y(X)$. It follows that 
    $\Gen \widetilde{L} =  \Gen \Fbm_Y(X)$ and that $\widetilde{L}^{\perp} = \Fbm_Y(X)^{\perp} $.

    Using that $\widetilde{L}$ is $\tau$-rigid and~\cite[Lemma 4.4]{bm}, we have that 
$Y \in \P(\Gen\Fbm_Y(X)) = \P(\Gen \widetilde{L})$. Since $Y$ is in $\Gen\Fbm_Y(X)$, we must have 
$Y \in \Pns(\Gen \widetilde{L})$. Using that $(\Gen \widetilde{L},\widetilde{L}^{\perp})$ is a functorially 
finite torsion pair, we obtain by Corollary~\ref{cor:proj_inj} that $\tau Y \in \Is(\widetilde{L}^{\perp})$.
    
    Moreover, since $\tau \widetilde{L}$ is cogen-minimal (by (c)), we have
    \begin{align*}
        \FiltCogen(\Jinv(\tau\widetilde{L})) &= \tau^{-1}(\Imodns(\FiltCogen(\Jinv(\tau\widetilde{L}))))^\perp & \text{by (b) and the dual of Lemma~\ref{lem:perp_injectives_2}}\\
        &= \tilde{L}^\perp & \text{by Equation~(\ref{eqn:tau_U_tilde}) and item (a).}
    \end{align*} 
    Hence, $\tau Y$ is a direct summand of $\Imods(\FiltCogen(\Jinv(\tau \widetilde{L})))$. However, using (\ref{eqn:tau_U_tilde}), we have that $\tau Y$ is also a direct summand of $\tau\widetilde{L} = \Imodns(\FiltCogen(\Jinv(\tau \widetilde{L})))$, giving a contradiction. Hence, $\Imodns(\FiltCogen(\Jinv(\tau\widetilde{L})))$ is not cogen-minimal, as required.

Recall that $\Jinv(\tau\widetilde{L}) =\J(\widetilde{L})$.
By item (f), $\Imodns(\FiltCogen(\J(\widetilde{L})))$ is a direct sum of two indecomposable modules. Since $\Imodns(\FiltCogen(\J(\widetilde{L})))$ is non-zero, $I_s(\Cogen \Imodns(\FiltCogen(\J(\widetilde{L}))))$ is also non-zero.
By Corollary~\ref{cor:genmin} we can write 
\begin{equation}\label{eq:decompose}
    \Imodns(\FiltCogen(\J(\widetilde{L}))) = \Imods(\Cogen \Imodns (\FiltCogen(\J(\widetilde{L})))) \amalg C'.
\end{equation}

Since $\Imodns(\FiltCogen(\J(\widetilde{L})))$ is not cogen-minimal (established above), $C'$ is non-zero, so both summands on the right-hand side are indecomposable.
Since an injective module cannot be non-split injective in a torsion-free class, we see that
$\Imodns(\FiltCogen(\J(\widetilde{L})))$ has no injective direct summands, enabling us to write $C'=\tau C$ where $C$ is indecomposable, giving the first statement.

From Corollary~\ref{cor:proj_inj}, we have 
\begin{equation}\label{eq:projtoinj}
\Imodns(\FiltCogen(\J(\widetilde{L}))) =  \tau( \Pmods({}^\perp \J(\widetilde{L}))).
\end{equation}
So,
applying $\tau^{-1}$ to \eqref{eq:decompose} and using \eqref{eq:projtoinj}, we obtain the decomposition
\begin{equation}
\label{eq:decompose2}
    P_{s}({}^\perp \J(\widetilde{L})) = \tau^{-1} \Imods(\Cogen \Imodns(\FiltCogen(\J(\widetilde{L})))) \amalg C.
    \end{equation}
Since both terms on the right-hand side of~\eqref{eq:decompose} are indecomposable and non-injective, we see that both terms on the right-hand side of~\eqref{eq:decompose2} are indecomposable and non-projective, as desired.
\end{proof}

We conclude this section with two results which show that, when they exist, left and right irregular pairs uniquely determine their (iterated) $\tau$-perpendicular categories. We expand further upon these uniqueness results in Corollary~\ref{cor:irregular_unique}.

\begin{proposition}\label{prop:left_irregular_unique}
    Let $(B,C)$ and $(\overline{B},\overline{C})$ be left irregular $\tau$-exceptional pairs such that $\J(B,C) = \J(\overline{B},\overline{C})$. Then $(B,C) = (\overline{B},\overline{C})$.
\end{proposition}

\begin{proof}
    Suppose for a contradiction that $(B,C) \neq (\overline{B},\overline{C})$. Since $\J(B,C) = \J(\overline{B},\overline{C})$, Proposition~\ref{prop:JUsinceresetup1}(e) says that
    $$\Fbm_C(B) \amalg C = \Pmods({}^\perp \J(B,C)) = \Pmods({}^\perp \J(\overline{B},\overline{C}))=\Fbm_{\overline{C}}(\overline{B}) \amalg \overline{C}.$$
    Theorem~\ref{thm:tfo} and the assumption that $(B,C) \neq (\overline{B},\overline{C})$ thus imply that $\overline{C} = \Fbm_C(B)$ and $\Fbm_{\overline{C}}(\overline{B}) = C$. Since $(B,C)$ is left irregular, we have that 
    $C \in \P({}^\perp \tau \Fbm_C(B)) \setminus \P(\Lambda)$.
    Likewise, since $(\overline{B},\overline{C}) = (\Ebm_{{\Fbm_C}(B)}(C),\Fbm_C(B))$ is left irregular, we have that $\Fbm_C(B) \in \P({}^\perp \tau C) \setminus \P(\Lambda)$. Since $C\not\in \P(\Lambda)$, we have $C = \Fbm_C(B)$ by Lemma~\ref{lem:eachinperp}, a contradiction to the definition of $\Ebm$ (see Theorem~\ref{thm:bmcorr}).
\end{proof}

\begin{proposition}\label{prop:right_irregular_unique}
    Let $(X,Y)$ and $(\overline{X},\overline{Y})$ be right irregular $\tau$-exceptional pairs such that $\J(X,Y) = \J(\overline{X},\overline{Y})$. Then $(X,Y) = (\overline{X},\overline{Y})$.
\end{proposition}

\begin{proof}
    Suppose for a contradiction that $(X,Y) \neq (\overline{X},\overline{Y})$. Since $\J(X,Y) = \J(\overline{X},\overline{Y})$, Proposition~\ref{prop:JUsinceresetup2}(e) says that 
    $$\Fbm_Y(X) \amalg Y = \tau^{-1} \Imods(\J(X,Y)^\perp) = \Fbm_{\overline{Y}}(\overline{X})\amalg \overline{Y}.$$
    Theorem~\ref{thm:tfo} and the assumption that $(X,Y) \neq (\overline{X},\overline{Y})$ thus imply that $\overline{X} = \Fbm_Y(X)$ and $\Fbm_{\overline{Y}}(\overline{X}) = Y$. Now since $(X,Y)$ is right irregular, we have that $Y \in \Gen \Fbm_Y(X)$. Likewise, since $(\overline{X},\overline{Y}) = (\Ebm_{\Fbm_Y(X)}(Y),\Fbm_Y(X))$ is right irregular, we have that $\Fbm_Y(X) \in \Gen Y$. Lemma~\ref{lem:twogens} then implies that $Y = \Fbm_Y(X)$, a contradiction.
\end{proof}

\section{Mutation of \texorpdfstring{$\tau$}{tau}-exceptional pairs}\label{sec:pair_mutation}

In this section we define left and right mutation of $\tau$-exceptional pairs. We show that these are inverse operations, proving Theorem \ref{thm:mutation_pairs_intro}. For algebras admitting torsion classes which are {\em not} functorially finite, we will see that, 
under certain circumstances, there can be pairs which are only left (or only right) mutable. Such pairs are always left (or right) irregular, and appear exactly when the corresponding $\tau$-perpendicular categories (which are always functorially finite themselves) are not left finite (resp. right finite), that is, when the torsion class (resp. torsion-free class) that they generate is {\em not} functorially finite. We emphasise that, as proved in Proposition~\ref{prop:one_mutation}, immutability is always ``one-sided'', that is, there do not exist $\tau$-exceptional pairs which are neither left nor right mutable.

\begin{definition}

A $\tau$-exceptional pair $(B,C)$ is called {\em left immutable} 
if it is left irregular and the $\tau$-perpendicular category $\J(B,C)$ is not left finite. A $\tau$-exceptional pair $(B,C)$ is called {\em left mutable} if it is not left immutable.
\end{definition}

\begin{definition}

A $\tau$-exceptional pair $(X,Y)$ is called {\em right immutable}
if it is right irregular and the $\tau$-perpendicular category $\J(X,Y)$ is not right finite. A $\tau$-exceptional pair $(X,Y)$ is called {\em right mutable} if it is not right immutable.
\end{definition}

Note that, by the above definitions, any left regular $\tau$-exceptional pair $(B,C)$ is left mutable, and any right 
regular $\tau$-exceptional pair $(X,Y)$ is right mutable.

To state the mutation maps, we establish some notation. For $U$ indecomposable in $\mods \Lambda \amalg \mods \Lambda[1]$, we denote
$$|U| = \begin{cases} U, & \text{if $U \in \mods\Lambda$}; \\ U[-1], & \text{if $U \in \mods\Lambda[1]$.}\end{cases}$$
and extend this additively to all objects in $\mods\Lambda \amalg \mods\Lambda[1]$.

For $(M,N)$ a $\tau$-exceptional pair, we denote
\begin{align*}
    N_+ &:= \begin{cases} N[1], & \text{if $N$ is projective;}\\ N, & \text{otherwise;} \end{cases}\\
    M^+ &:= \begin{cases} M[1], & \text{if $M$ is projective in $\J(N)$;}\\ M, & \text{otherwise;} \end{cases}\\
    M_{N\uparrow} &:= \Fbm_{N_+}(M).\\
    M^+_{N\uparrow} &:= \Fbm_{N}(M^+).
\end{align*}
Note that $|N_+| = N$ and $|M^+| = M$.

\begin{definition-proposition}\label{defprop:forward}
For a left regular $\tau$-exceptional pair $(B,C)$ let
$$\varphi(B,C) = (\lvert \Ebm_{B_{C\uparrow}}(C_+) \rvert, B_{C\uparrow}).$$
Then $\varphi(B,C)$ is a right regular $\tau$-exceptional pair and $\J(\varphi(B,C)) = \J(B,C)$.
\end{definition-proposition}

\begin{proof}

Since $\J(Q) = \J(Q[1])$ for any projective module $Q$, we have that $(B,C_+)$ is a signed $\tau$-exceptional pair. Thus $(\Ebm_{B_{C\uparrow}}(C_+),B_{C\uparrow})$ is also a signed $\tau$-exceptional pair by Theorem~\ref{thm:signed_mutation}. Moreover, we have $B_{C\uparrow} \in \mods\Lambda$ by Theorem~\ref{thm:F_map_1} and Theorem~\ref{thm:bmcorr}. It follows that $(\lvert \Ebm_{B_{C\uparrow}}(C_+) \rvert, B_{C\uparrow})$ is a $\tau$-exceptional pair.  We then compute
    \begin{align*}
        \J(\lvert \Ebm_{B_{C\uparrow}}(C_+) \rvert, B_{C\uparrow}) &= \J(\Ebm_{B_{C\uparrow}}(C_+), B_{C\uparrow}) &\text{by the definition of $\J$}\\
        &=\J\left(\Ebm_{\Fbm_{B_{C\uparrow}}\circ\Ebm_{B_{C\uparrow}}(C_+)}(B_{C\uparrow}),\Fbm_{B_{C\uparrow}}\circ\Ebm_{B_{C\uparrow}}(C_+)\right) & \text{by Theorem~\ref{thm:signed_mutation}}\\
            &= \J(\Ebm_{C_+}(\Fbm_{C_+}(B)),C_+) & \text{by the definition of $B_{C\uparrow}$}\\
            &= \J(B,C_+) &\\
            &= \J(B,C) &\text{by the definition of $\J$.}
    \end{align*} 
    It remains to show that $(\lvert \Ebm_{B_{C\uparrow}}(C_+) \rvert, B_{C\uparrow})$ is right regular. Suppose first that $\Ebm_{B_{C\uparrow}}(C_+) = X[1] \in \P(\J(B_{C\uparrow}))[1]$, so that $\varphi(B,C) = (X, B_{C\uparrow})$. Then it follows from 
    Lemma~\ref{lem:fUbijection}(a) that 
    $\Fbm_{B_{C\uparrow}}(X) \in \P(^{\perp}\tau B_{C\uparrow})$,
    and hence $(X, B_{C\uparrow})$ is right regular by definition. Now suppose that $\Ebm_{B_{C\uparrow}}(C_+) \notin \P(\J(B_{C\uparrow}))[1]$. By the definition of $\Ebm$ (see Theorem~\ref{thm:bmcorr}), this means that $C_+ \notin (\Gen B_{C\uparrow}) \amalg \P(\Lambda)[1]$,
    and so $C_+ = C$ and $\lvert \Ebm_{B_{C\uparrow}}(C_+)\rvert = \Ebm_{B_{C\uparrow}}(C)$.
    We then compute
    $$
            \Fbm_{B_{C\uparrow}}(\lvert \Ebm_{B_{C\uparrow}}(C_+) \rvert) = \Fbm_{B_{C\uparrow}}(\Ebm_{B_{C\uparrow}}(C)) = C.
    $$
    Note that $\Fbm_C(B) \notin \Gen C$, because 
    $B = \Ebm_C(\Fbm_C(B))\in \mods \Lambda.$
It follows that 
    $$B_{C\uparrow} = \Fbm_C(B) \notin \Gen C = \Gen \Fbm_{B_{C\uparrow}}(\lvert \Ebm_{B_{C\uparrow}}(C_+) \rvert).$$
    We conclude that $(\lvert \Ebm_{B_{C\uparrow}}(C_+) \rvert, B_{C\uparrow})$ is right regular.
\end{proof}

\begin{definition-proposition}\label{defprop:forward2}
Let $(B, C)$ be a left mutable left irregular $\tau$-exceptional pair and let $L = \Fbm_C(B) \amalg C$. 
Let $X' = \Pmods(\Gen \Pmodns(\FiltGen(\J(L))))$, and let $Y = \Pmodns(\FiltGen(\J(L)))/X'$. Then $\varphi(B,C) := (\Ebm_Y(X'),Y)$ is a right irregular right mutable $\tau$-exceptional pair which satisfies $\J(\varphi(B,C)) = \J(B,C)=\J(L)$.
\end{definition-proposition}	

\begin{proof}
Note that $\J(B,C) = \J(L)$ by Theorem~\ref{thm:J_E}. Since $(B,C)$ is left mutable and left irregular, we have that $\J(B,C)$ is left finite. It then follows from Proposition~\ref{prop:JUsinceresetup1}(g) that 
\begin{equation}
    X'\amalg Y=\Pmodns(\FiltGen(\J(L)))
\label{eq:XYsum}
\end{equation}
which is a $\tau$-rigid module by Theorem~\ref{thm:Bongartz}. Note that $(X',Y)$ is TF-ordered since $X'$ is split projective in $\Gen \Pmodns(\FiltGen(\J(L)))$.
Thus $(\Ebm_Y(X'),Y)$ is a $\tau$-exceptional pair by Theorem~\ref{thm:tfo}. We then compute
    \begin{align*}
        \J(\Ebm_Y(X'),Y) &= \J(X' \amalg Y) & \text{by Theorem~\ref{thm:J_E}}\\
            &=  \J(\Pmodns(\FiltGen(\J(B,C)))) & \text{by the definition of $X' \amalg Y$}\\
            &= \J(B,C) \ (=\J(L))& \text{by Proposition~\ref{prop:JUsinceresetup1}(f).}
    \end{align*}
    It remains to show that $(\Ebm_Y(X'),Y)$ is right irregular and right mutable. We first show that it is right irregular. We have $Y \in \Gen(X') = \Gen(\Fbm_Y\circ \Ebm_Y(X'))$ by Proposition~\ref{prop:JUsinceresetup1}(g). Now suppose for a contradiction
    that $X' \in \P({}^\perp \tau Y)$; i.e., that $X'$ 
    is a direct summand of the Bongartz complement $B_Y$ of $Y$. So $X'\in \Ps({}^{\perp}\tau Y)$, by Lemma~\ref{lem:Bongartzsplit}.
Hence, using Theorem~\ref{thm:Bongartz},
\begin{align*}
   {}^{\perp} \tau (X'\amalg Y) &\subseteq {}^{\perp} \tau Y \\
    &={}^{\perp} \tau (Y\amalg B_Y) \\
    &\subseteq {}^{\perp} \tau (X'\amalg Y).
\end{align*}
We must have equalities all the way through, so ${}^\perp \tau (X'\amalg Y)={}^{\perp} \tau Y$. Moreover, we have that $\J(L)$ (hence also $\FiltGen(\J(L))$) is sincere by Proposition~\ref{prop:JUsinceresetup1}(b), and that $\FiltGen(\J(L))$ is functorially finite by the assumption that $(B,C)$ is left mutable. We then use Lemma~\ref{lem:perp_injectives_2} to compute
\begin{equation}\label{eq:FGJU}\FiltGen(\J(L)) = {}^\perp \tau \Pmodns(\FiltGen(\J(L))) = {}^\perp \tau(X'\amalg Y) = {}^\perp \tau Y.\end{equation}
By the above, $X' \in \Ps({}^{\perp} \tau Y)$, so $X'$ is a direct summand of $\Pmods(\FiltGen(\J(L)))$ by~\eqref{eq:FGJU}. 
But, by~\eqref{eq:XYsum}, the module $X'$ is a direct summand of 
$\Pmodns(\FiltGen(\J(L)))$, so this is a contradiction. We conclude that $X' \notin \P({}^\perp \tau Y)$, and thus that $(\Ebm_Y(X'), Y)$ is right irregular.

Note that the torsion class ${}^\perp \J(L)$ is functorially finite by Proposition~\ref{prop:JUsinceresetup1}(d). Equivalently, the torsion-free class $\FiltCogen(\J(L))$ is functorially finite (by Proposition~\ref{prop:asff});  that is, the wide subcategory $\J(L) = \J(\Ebm_Y(X'),Y)$ is right finite. Thus $(\Ebm_Y(X'),Y)$ is right mutable.
\end{proof}

\begin{definition-proposition}\label{defprop:backward}
For a right regular $\tau$-exceptional pair $(X,Y)$ let
$$\psi(X,Y) = (\Ebm_{X^+_{Y\uparrow}}(Y), \lvert X^+_{Y\uparrow}\rvert).$$ Then $\psi(X,Y)$ is a left regular $\tau$-exceptional pair and $\J(\psi(X,Y)) = \J(X,Y)$.
\end{definition-proposition}

\begin{proof}
    Since $\J_{\J(Y)}(Q) = \J_{\J(Y)}(Q[1])$ for any $Q \in \P(\J(Y))$, we have that $(X^+, Y)$ is a signed $\tau$-exceptional pair with $\J(X^+,Y) = \J(X,Y)$. Thus
    $$\left(\Ebm_{X^+_{Y\uparrow}}(Y), X^+_{Y\uparrow}\right) = \left(\Ebm_{\Fbm_Y(X^+)}(Y), \Fbm_Y(X^+)\right)$$
    is also a signed $\tau$-exceptional pair with $\J\left(\Ebm_{X^+_{Y\uparrow}}(Y), X^+_{Y\uparrow}\right) = \J(X,Y)$ by Theorem~\ref{thm:signed_mutation}. Since $\J(Q) = \J(Q[1])$ for all $Q \in \P(\Lambda)$, it follows that $\psi(X,Y)$ is a signed $\tau$-exceptional pair with $\J(\psi(X,Y)) = \J(X,Y)$. To see that this is an (unsigned) $\tau$-exceptional pair, we need to show that $\Ebm_{X^+_{Y\uparrow}}(Y)  \in \mods\Lambda$. We must consider two cases.

    Suppose first that $X^+ \in \mods\Lambda$. Equivalently, this means $X^+ = X \notin \P(\J(Y))$. Then $X_{Y\uparrow}^+ = \Fbm_Y(X) \notin \P({}^\perp \tau Y)$ by Lemma~\ref{lem:fUbijection}. Combining this with the fact that $(X,Y)$ is right regular, we have by definition that $Y \notin \Gen \Fbm_Y(X)$, and thus that $$\Ebm_{X^+_{Y\uparrow}}(Y) = \Ebm_{\Fbm_Y(X)}(Y) = f_{\Fbm_Y(X)}(Y) \in \mods\Lambda$$ by Theorem~\ref{thm:bmcorr}.

    Now suppose that $X^+ \notin \mods\Lambda$. Then, by the definition of the map 
    $\Ebm$ (see Theorem~\ref{thm:bmcorr}), we have that either 
$X_{Y\uparrow}^+ = \Fbm_Y(X^+) \in \P(\Lambda)[1]$ or $X_{Y\uparrow}^+ = \Fbm_Y(X^+) \in \Gen Y$.
    Now if $X_{Y\uparrow}^+ \in \P(\Lambda$)[1], then $\Ebm_{X^+_{Y\uparrow}}(Y) \in \mods\Lambda$ by the definition of the map $\Ebm$ (see Theorem~\ref{thm:bmcorr}).
    In the case $X_{Y\uparrow}^+ \in \Gen Y$, Lemma~\ref{lem:twogens} implies that $Y \notin \Gen X_{Y\uparrow}^+$, and thus $\Ebm_{X_{Y\uparrow}^+}(Y) = f_{X_{Y\uparrow}^+}(Y) \in \mods\Lambda$ by the definition of $\Ebm$ (see Theorem~\ref{thm:bmcorr}).
    
    We conclude that $\Ebm_{X_Y^+}(Y) \in \mods\Lambda$, and thus that $\psi(X,Y)$ is a $\tau$-exceptional pair which satisfies $\J(\psi(X,Y)) = \J(X,Y)$. 
    
    It remains to show that $\psi(X,Y) = (\Ebm_{X^+_{Y\uparrow}}(Y), \lvert X^+_{Y\uparrow}\rvert)$ is left regular. If $\lvert X^+_{Y\uparrow}\rvert \in \P(\Lambda)$, then there is nothing to show.
    Thus suppose that $\lvert X^+_{Y\uparrow}\rvert \not \in \P(\Lambda)$. Then in particular $X^+_{Y\uparrow} \in \mods\Lambda$, and so $X^+ \notin \P(\J(Y))[1]$ 
by the definition of the map $\Ebm$ (see Theorem~\ref{thm:bmcorr}).
    By the definition of $X^+$, this means $X^+ = X \notin \P(\J(Y))$. We conclude that $X_{Y\uparrow}^+ = \Fbm_Y(X) \notin \P({}^\perp \tau Y)$, by Lemma~\ref{lem:fUbijection},
    and thus that $\psi(X,Y)$ is left regular. 
\end{proof}

\begin{definition-proposition}\label{defprop:backward2}
Let $(X,Y)$ be a right mutable right irregular $\tau$-exceptional pair and let $\widetilde{L} = \Fbm_Y(X)\amalg Y$. Let $B' = \tau^{-1}\Imods(\Cogen \tau \Pmods({}^\perp \J(\widetilde{L})))$, and let $C = \Pmods({}^\perp \J(\widetilde{L}))/B'$. Then $\psi(X,Y) := (\Ebm_C(B'),C)$ is a left irregular left mutable $\tau$-exceptional pair which satisfies $\J(\psi(X,Y)) = \J(X,Y)$($=\J(\widetilde{L})$).
\end{definition-proposition}

\begin{proof}
Note that $\J(X,Y) = \J(\widetilde{L})$ by Theorem~\ref{thm:J_E}. Since $(X,Y)$ is right mutable and right irregular, we have that $\J(\widetilde{L})$ is right finite. We thus have a functorially finite torsion pair
    $({}^\perp \J(\widetilde{L}),\FiltCogen(\J(\widetilde{L})))
    $.
Now recall from Proposition~\ref{prop:JUsinceresetup2}(g) and~\eqref{eq:projtoinj} that $$\Pmods({}^\perp \J(\widetilde{L})) = B' \amalg C = \tau^{-1}\Imods(\Cogen \Imodns(\FiltCogen(\J(\widetilde{L}))))\amalg C,$$
that $\Imodns(\FiltCogen(\J(\widetilde{L})))$ is not cogen-minimal and satisfies
\begin{equation}
    \label{eqn:irregular_right3b}\Imodns(\FiltCogen(\J(\widetilde{L}))) = \tau \Pmods({}^\perp \J(\widetilde{L}))
    = \tau B' \amalg \tau C
    = \Imods(\Cogen \Imodns(\FiltCogen(\J(\widetilde{L}))))\amalg \tau C,
\end{equation}
and that all of $B', \tau B', C$, and $\tau C$ are indecomposable. 

Thus $(B',C)$ is a TF-ordered $\tau$-rigid module since 
    $B' \amalg C = \Pmods({}^\perp \J(\widetilde{L}))$ is split projective in ${}^{\perp} \J(\widetilde{L})$. By Theorem~\ref{thm:tfo}, we conclude that $(\Ebm_C(B'),C)$ is a $\tau$-exceptional sequence. We then compute
    \begin{align*}
        \J(\Ebm_C(B'),C) &= \J(B' \amalg C) & \text{by Theorem~\ref{thm:J_E}}\\
            &= \J(\Pmods({}^\perp \J(\widetilde{L}))) & \text{by the definition of }B'\amalg C\\
            &= \J(X,Y) & \text{by Proposition~\ref{prop:JUsinceresetup2}(f).}
    \end{align*}
    It remains to show that $(\Ebm_C(B'),C)$ is left irregular and left mutable. We first show that it is left irregular. We have that $C \notin \P(\Lambda)$ since $\J(B'\amalg C) = \J(\widetilde{L})$ is sincere
    (by Proposition~\ref{prop:JUsinceresetup2}(b)), so it suffices to show that $C \in \P({}^\perp \tau B')$.
    
    Since $\tau B'=\Imods(\Cogen\Imodns(\FiltCogen(\J(\widetilde{L}))))$ (by~\eqref{eqn:irregular_right3b}), we have by Theorem~\ref{thm:dualBongartz}(b) that
    \begin{equation}
    \label{eq:tauB'}
    \Cogen \Imodns(\FiltCogen(\J(\widetilde{L})))=\Cogen \tau B'.
    \end{equation}
    Since $\Imodns(\FiltCogen(\J(\widetilde{L})))= \tau B' \amalg \tau C$ is $\tau^{-1}$-rigid by Theorem~\ref{thm:dualBongartz}, we have $\Ext^1(-,\tau C)|_{\Cogen \tau B'} = 0$ by Proposition~\ref{prop:ASExt}. Thus $\tau C\in \I(\Cogen \tau B')=\I(\Cogen \Imodns(\FiltCogen(\J(\widetilde{L}))))$, \sloppy and hence
    $\tau C\in \I_{\text{ns}}(\Cogen\Imodns(\FiltCogen(\J(\widetilde{L}))))$ since $\tau C\not\in \I_{\text{s}}(\Cogen \Imodns(\FiltCogen(\J(\widetilde{L}))))$.
    Hence, by~\eqref{eq:tauB'}, we have $\tau C\in \I_{\text{ns}}(\Cogen \tau B')$. Considering the torsion pair $({}^\perp \tau B', \Cogen \tau B')$, it
    then follows from Corollary~\ref{cor:proj_inj} that $\tau C \in \tau\P({}^\perp \tau B')$. So, since $C$ is not projective, $C\in \P({}^\perp \tau B')$.
    We conclude that $(\Ebm_C(B'),C)$ is left irregular.

    Finally, we show that $(\Ebm_C(B'),C)$ is left mutable. Note that the torsion-free class $\J(\widetilde{L})^\perp$ is functorially finite by Proposition~\ref{prop:JUsinceresetup2}(d). Equivalently, the torsion class $\FiltGen(\J(\widetilde{L}))$ is functorially finite; that is, the wide subcategory $\J(\widetilde{L}) = \J(\Ebm_C(B'),C)$ is left finite. Thus $(\Ebm_C(B'),C)$ is left mutable.
\end{proof}

We are now prepared to prove our first main theorem.

\begin{theorem}[Theorem~\ref{thm:mutation_pairs_intro}]\label{thm:mutation_pairs}

The maps $\varphi$ and $\psi$ give inverse bijections
    \[\{\text{left mutable $\tau$-exceptional pairs}\} \qquad \stackrel[\psi]{\varphi}{\rightleftarrows} \qquad \{\text{right mutable $\tau$-exceptional pairs}\}\] 
\end{theorem}

\begin{proof}

We first show that $\psi\circ\varphi$ is the identity. Let $(B,C)$ be a left mutable $\tau$-exceptional pair.

We first consider the case where $(B,C)$ is left irregular. Denote $(\overline{B},\overline{C}) := \psi(\varphi(B,C))$. By Definition-Propositions~\ref{defprop:forward2} and~\ref{defprop:backward2}, $(\overline{B},\overline{C})$ is left irregular and satisfies $\J(\overline{B},\overline{C}) = \J(B,C)$. Proposition~\ref{prop:left_irregular_unique} thus implies that $(B,C) = (\overline{B},\overline{C})$.

It remains to consider the case where $(B,C)$ is left regular. By Definition-Propositions~\ref{defprop:forward} and~\ref{defprop:backward}, we compute
\begin{align*}
    \psi(\varphi(B,C)) &= \psi\left(\left\lvert \Ebm_{B_{C\uparrow}}(C_+) \right\rvert, B_{C\uparrow}\right)\\
      &= \left(\Ebm_{\left\lvert \Ebm_{B_{C\uparrow}}(C_+) \right\rvert_{(B_{C\uparrow})\uparrow}^+}(B_{C\uparrow}),\left\lvert\left\lvert \Ebm_{B_{C\uparrow}}(C_+) \right\rvert_{(B_{C\uparrow})\uparrow}^+\right\rvert\right).
\end{align*}

Definition-Propositions~\ref{defprop:forward} and~\ref{defprop:backward} also imply that $\J(B,C) = \J(\psi(\varphi(B,C)))$. Using Corollary~\ref{cor:unique}, it therefore suffices to show that
\begin{equation}\label{eqn:identity_regular}\left\lvert\left\lvert \Ebm_{B_{C\uparrow}}(C_+) \right\rvert_{(B_{C\uparrow})\uparrow}^+\right\rvert = C.\end{equation}
Since $(B,C)$ is left regular by assumption, we have $C \not \in \P(^{\perp}\tau\Fbm_C(B))$ or $C \in \P(\Lambda)$. Assume first that $C\not \in \P(\Lambda)$. Then we have
$$C_+ = C \notin \P(^{\perp}\tau\Fbm_C(B)) = \P({}^\perp \tau \Fbm_{C_+}(B))= \P({}^\perp \tau B_{C\uparrow}).$$
Lemma~\ref{lem:fUbijection} then implies that $$f_{B_{C\uparrow}}(C_+) =\Ebm_{B_{C\uparrow}}(C_+) \notin \P(\J(B_{C\uparrow})),$$ and so 
\begin{equation}\label{eq:not-proj}
\left\lvert \Ebm_{B_{C\uparrow}}(C_+) \right\rvert^+ = \Ebm_{B_{C\uparrow}}(C_+).
\end{equation}
If $C\in \P(\Lambda)$ then $C_+ \in \P(\Lambda)[1]$, and hence 
$\Ebm_{B_{C\uparrow}}(C_+) \in \P(\J(B_{C\uparrow}))[1]$ by the definition of 
$\Ebm$ (see Theorem~\ref{thm:bmcorr}).
This implies that Equation \eqref{eq:not-proj} also holds in this case.

We now verify Equation~\eqref{eqn:identity_regular}:

\begin{align*}
    \left\lvert\left\lvert \Ebm_{B_{C\uparrow}}(C_+) \right\rvert^+_{(B_{C\uparrow})\uparrow}\right\rvert &= \left\lvert \Fbm_{B_{C\uparrow}}\left(\left\lvert\Ebm_{B_{C\uparrow}}(C_+)\right\rvert^+\right)\right\rvert & \text{by the definition of $(-)^+_{(?)\uparrow}$}\\
        &= \left\lvert\Fbm_{B_{C\uparrow}}\circ \Ebm_{B_{C\uparrow}}(C_+))\right\rvert & \text{by Equation~\eqref{eq:not-proj}}\\
        &= \lvert C_+\rvert\\
        &= C.
\end{align*}
We conclude that $\psi\circ\varphi(B,C) = (B,C)$ if $(B,C)$ is left regular.

We next show that $\varphi\circ \psi$ is the identity. Let $(X,Y)$ be a right mutable $\tau$-exceptional pair.

We first consider the case where $(X,Y)$ is right irregular. Denote $(\overline{X},\overline{Y}) := \varphi(\psi(X,Y))$. By Definition-Propositions~\ref{defprop:forward2} and~\ref{defprop:backward2}, $(\overline{X},\overline{Y})$ is left irregular and satisfies $\J(\overline{X},\overline{Y}) = \J(X,Y)$. Proposition~\ref{prop:right_irregular_unique} thus implies that $(X,Y) = (\overline{X},\overline{Y})$.

It remains to consider the case where $(X,Y)$ is right regular. By Definition-Propositions~\ref{defprop:forward} and~\ref{defprop:backward}, we compute
\begin{align*}
    \varphi(\psi(X,Y)) &= \varphi\left(\Ebm_{X^+_{Y\uparrow}}(Y),|X_{Y\uparrow}^+|\right)\\
    &= \left(\left|\Ebm_{\left(\Ebm_{X^+_{Y\uparrow}}(Y)\right)_{|X_{Y\uparrow}^+|\uparrow}}\left(|X_{Y\uparrow}^+|_+\right)\right|,\left(\Ebm_{X^+_{Y\uparrow}}(Y)\right)_{|X_{Y\uparrow}^+|\uparrow}\right).
\end{align*}
Definition-Propositions~\ref{defprop:forward} and~\ref{defprop:backward} also imply that $\J(X,Y) = \J(\varphi(\psi(X,Y)))$. Using Corollary~\ref{cor:unique}, it therefore suffices to show that
\begin{equation}\label{eqn:right_comp}
    \left(\Ebm_{X^+_{Y\uparrow}}(Y)\right)_{|X_{Y\uparrow}^+|\uparrow} = Y.
\end{equation}

 Note that, by definition, $X^+ \notin \P(\J(Y))$. By Lemma~\ref{lem:fUbijection}, this implies that $X_{Y\uparrow}^+ = \Fbm_Y(X^+) \notin \P({}^\perp \tau Y)$.
Thus $X_{Y\uparrow}^+ \notin \P(\Lambda)$, and so $|X_{Y\uparrow}^+|_+ = X_{Y\uparrow}^+$. We then verify Equation~\eqref{eqn:right_comp}:

\begin{align*}
    \left(\Ebm_{X^+_{Y\uparrow}}(Y)\right)_{|X_{Y\uparrow}^+|\uparrow} &= \Fbm_{|X_{Y\uparrow}^+|_+}\left(\Ebm_{X^+_{Y\uparrow}}(Y)\right) & \text{by the definition of $(-)_{(?)\uparrow}$}\\
    &= \Fbm_{X_{Y\uparrow}^+}\left(\Ebm_{X^+_{Y\uparrow}}(Y)\right) & \text{since $|X_{Y\uparrow}^+|_+ = X_{Y\uparrow}^+$} \\
    &= Y.
\end{align*}
We conclude that $\varphi\circ\psi(X,Y) = (X,Y)$ also in the case when $(X,Y)$ is right regular.
\end{proof}

Using Theorem~\ref{thm:mutation_pairs}, we can expand upon the uniqueness results of Propositions~\ref{prop:left_irregular_unique} and~\ref{prop:right_irregular_unique} as follows.

\begin{corollary}\label{cor:irregular_unique}
    Let $\W$ be a $\tau$-perpendicular category of corank two. Then the following are equivalent.
    \begin{enumerate}
        \item[(a)] There exists a left irregular left mutable $\tau$-exceptional pair $(B,C)$ with $\J(B,C) = \W$.
        \item[(b)] There exists a right irregular right mutable $\tau$-exceptional pair $(X,Y)$ with $\J(X,Y) = \W$.
        \item[(c)] There exists a unique left irregular left mutable $\tau$-exceptional pair $(B,C)$ with $\J(B,C) = \W$ and a unique right mutable $\tau$-exceptional pair $(X,Y)$ with $\J(X,Y) = \W$. These satisfy $(X,Y) = \varphi(B,C)$ and $(B,C) = \psi(X,Y)$.
    \end{enumerate}
\end{corollary}

\begin{proof}
    The implications $(c\implies a)$ and $(c \implies b)$ are trivial. Moreover, the equivalence between (a) and (b) follows immediately from Definition-Propositions~\ref{defprop:forward2} and~\ref{defprop:backward2}. Thus it suffices to show that (a) and (b) together imply (c).

    Suppose (a) and (b) hold. Then the $\tau$-exceptional pairs $(B,C)$ and $(X,Y)$ realising conditions (a) and (b) are unique by Propositions~\ref{prop:left_irregular_unique} and~\ref{prop:right_irregular_unique}. The fact that these are interchanged by $\varphi$ and $\psi$ then follows from Definition-Propositions~\ref{defprop:forward2} and~\ref{defprop:backward2}.
\end{proof}

We conclude this section by proving our second and third main theorems.

\begin{theorem}[Theorem~\ref{thm:always_mutable_intro}]\label{prop:one_mutation}
    Every $\tau$-exceptional pair is left mutable or right mutable (or both).
\end{theorem}

\begin{proof}
    Let $(B,C)$ be a $\tau$-exceptional pair, and suppose that $(B,C)$ is not left mutable. By definition, this means $(B,C)$ is left irregular, and so the torsion class ${}^\perp \J(B,C)$ is functorially finite by Proposition~\ref{prop:JUsinceresetup1}(d). Equivalently (by Proposition~\ref{prop:asff} and Definition-Proposition~\ref{defprop:torsionpairs}), we have that $\J(B,C)$ is right finite. By definition, this means that $(B,C)$ is right mutable.
\end{proof}

\begin{remark}\label{rem:immutable}
    More generally, the argument used to prove Theorem~\ref{prop:one_mutation} shows that if $(B,C)$ is left immutable (resp. right immutable), then there cannot exist a right immutable (resp. left immutable) $\tau$-exceptional pair $(X,Y)$ with $\J(B,C) = \J(X,Y)$.

    In particular, note that the formula for $\varphi$ used in the left regular case (Definition-Proposition~\ref{defprop:forward}) could also be applied to a left irregular $\tau$-exceptional pair $(B,C)$. The only use of left regularity in the proof was to show that the resulting $\tau$-exceptional pair is right regular, so the output would be a $\tau$-exceptional pair $(X,Y)$ with $\J(X,Y) = \J(B,C)$. The above paragraph then says that $(X,Y)$ is right mutable, so we could not hope to establish an extension of Theorem~\ref{thm:mutation_pairs} using this extension of the map $\varphi$. Applying the formula for $\psi$ used in the right regular case to a right irregular pair leads to a similar consequence. 
\end{remark}

\begin{theorem}
[Theorem~\ref{thm:mutation_complete_intro}]\label{prop:mutation_complete}
    Suppose that $\Lambda$ is hereditary, $\tau$-tilting finite, or has rank two. Then every $\tau$-exceptional pair is both left and right mutable.
\end{theorem}

\begin{proof}
    It suffices to show that every $\tau$-perpendicular category is both left and right finite. We consider each of the three cases separately.
    
    Suppose that $\Lambda$ is hereditary. It follows from~\cite[Cor.~2.17]{IngTho} that the functorially finite wide subcategories, $\tau$-perpendicular categories, left finite wide subcategories, and right finite wide subcategories all coincide (see also~\cite[Rk.~4.10]{bh}). In particular, every $\tau$-perpendicular subcategory is both left and right finite.

    Now suppose $\Lambda$ is $\tau$-tilting finite. Then every torsion class (and so also every torsion-free class) is functorially finite by~\cite[Thm.~1.2]{dij}. Theorem~\ref{thm:ms} and its dual thus imply that every wide subcategory is both left and right finite. Since $\tau$-perpendicular categories are wide (Definition-Proposition~\ref{prop:J_def}), this implies the result. (In fact, Theorem~\ref{thm:WLtorsionclass} implies that the wide subcategories, left finite wide subcategories, right finite wide subcategories, and $\tau$-perpendicular categories all coincide in the $\tau$-tilting finite case.)
    
    Finally, suppose that $\Lambda$ has rank two.
    Since
    $0$ and $\mods\Lambda$ are Serre subcategories, they are both left
    and right finite as wide subcategories. Since $\rk(\Lambda) = 2$, every other $\tau$-perpendicular category has corank one by Definition-Proposition~\ref{prop:J_def}, and is thus both left and right finite by~\cite[Cor.~4.8]{bh}.
\end{proof}

\section{Mutation of \texorpdfstring{$\tau$}{tau}-exceptional sequences}\label{sec:sequence_mutation}

We now extend the mutation operators on $\tau$-exceptional pairs to a family of mutation operators on $\tau$-exceptional sequences. For a $\tau$-perpendicular category $\W$, we will denote by $\varphi^\W$ and $\psi^\W$ the operators $\varphi$ and $\psi$ for the category $\W$. When $\W = \mods\Lambda$, we will still sometimes write $\varphi$ and $\psi$ in place of $\varphi^{\mods\Lambda}$ and $\psi^{\mods\Lambda}$.

We begin by expanding the definitions of mutability and regularity.

\begin{definition}\label{def:i_mutable}
Let $\mathcal{M} = (M_s,M_{s+1},\ldots,M_n)$ be a $\tau$-exceptional sequence and $i \in \{s,\ldots,n-1\}$.
    \begin{enumerate}
        \item[(a)] We say that $\mathcal{M}$ is \emph{left $i$-mutable} (resp. \emph{right $i$-mutable}) if  $(M_i,M_{i+1})$ is a left mutable (resp. right mutable) $\tau_{\J(M_{i+2},\ldots,M_n)}$-exceptional pair.
        \item[(b)] We say that $\mathcal{M}$ is \emph{left $i$-regular} (resp. \emph{right $i$-regular}) if $(M_i,M_{i+1})$ is a left regular (resp. right regular) $\tau_{\J(M_{i+2},\ldots,M_n)}$-exceptional pair.
        \item[(c)] We say that $\mathcal{M}$ is \emph{left $i$-irregular} (resp. \emph{right $i$-irregular}) if $(M_i,M_{i+1})$ is a left irregular (resp. right irregular) $\tau_{\J(M_{i+2},\ldots,M_n)}$-exceptional pair.
    \end{enumerate}
\end{definition}

\begin{definition}\label{def:i_mutation}
Let $\mathcal{M} = (M_s,M_{s+1},\ldots,M_n)$ be a $\tau$-exceptional sequence and $i \in \{s,\ldots,n-1\}$.
\begin{enumerate}
    \item[(a)] If $\mathcal{M}$ is left $i$-mutable, then the \emph{$i$-th left mutation} of $\mathcal{M}$ is the $\tau$-exceptional sequence
    $$\varphi_i(\mathcal{M}) = (M_s,\ldots,M_{i-1},M'_i,M'_{i+1},M_{i+2},\ldots,M_n)$$
    where $(M'_i,M'_{i+1}) = \varphi^{\J(M_{i+2},\ldots,M_n)}(M_i,M_{i+1})$.
    \item[(b)] If $\mathcal{M}$ is right $i$-mutable, then the \emph{$i$-th right mutation} of $\mathcal{M}$ is the $\tau$-exceptional sequence
    $$\psi_i(\mathcal{M}) = (M_s,\ldots,M_{i-1},M''_i,M''_{i+1},M_{i+2},\ldots,M_n)$$
    where $(M''_i,M''_{i+1}) = \psi^{\J(M_{i+2},\ldots,M_n)}(M_i,M_{i+1})$.
\end{enumerate}
\end{definition}

Note that the outputs of $\varphi_i$ and $\psi_i$ are indeed $\tau$-exceptional sequences by the results of Section~\ref{sec:pair_mutation}. Moreover, Theorem~\ref{thm:mutation_pairs} and the Definition-Propositions in Section~\ref{sec:pair_mutation} immediately imply the following.

\begin{corollary}
    For all $s \leq i < n$, the maps $\varphi_i$ and $\psi_i$ give bijections
    \[\left\{\begin{array}{c} \text{left $i$-mutable $\tau$-exceptional}\\\text{ sequences of length $n-s+1$}\end{array}\right\} \qquad \stackrel[\psi_i]{\varphi_i}{\rightleftarrows} \qquad \left\{\begin{array}{c} \text{right $i$-mutable $\tau$-exceptional}\\\text{ sequences of length $n-s+1$}\end{array}\right\}.\]
    Moreover, these bijections satisfy the following.
    \begin{enumerate}
        \item[(a)] Let $\mathcal{M}$ be left $i$-mutable. Then $\J(\mathcal{M}) = \J(\varphi_i(\mathcal{M}))$.
        \item[(b)] Let $\mathcal{M}$ be left $i$-mutable. Then $\mathcal{M}$ is left $i$-irregular if and only if $\varphi_i(\mathcal{M})$ is right $i$-irregular.
        \item[(c)] Let $\mathcal{M} = (M_s,M_{s+1},\ldots,M_n)$ be a $\tau$-exceptional sequence and let $s \leq i < n$. Then $\mathcal{M}$ is left $i$-mutable or right $i$-mutable (or both).
    \end{enumerate}
\end{corollary}

Recall from the introduction that the algebra $\Lambda$ is \emph{mutation complete} if all of its $\tau$-exceptional sequences are both left and right $i$-mutable (for all $s \leq i < n$). As a consequence of Theorem~\ref{prop:mutation_complete}, we obtain the following.

\begin{corollary}[Corollary~\ref{cor:mutation_complete_intro}]\label{cor:mutation_complete}
    If $\Lambda$ is hereditary, $\tau$-tilting finite, or has rank two, then it is mutation complete.
\end{corollary}

\begin{proof}
    If $\Lambda$ has rank two, then the $\tau$-exceptional pairs are precisely the complete $\tau$-exceptional sequences. Moreover, it is well-known that if $\Lambda$ is hereditary (resp. $\tau$-tilting finite) then every $\tau$-perpendicular category of $\mods\Lambda$ is also hereditary (resp. $\tau$-tilting finite).
    This was proved in~\cite[Prop.~1.1]{gl} in the hereditary case. In the $\tau$-tilting finite case it is a direct consequence of~\cite[Thm.~1.1]{jasso}.
      Theorem~\ref{prop:mutation_complete} then implies the result in all three cases.
\end{proof}


\section{The hereditary case}\label{sec:hereditary}

Suppose for this section that $\Lambda$ is hereditary. Recall that, in this case, the exceptional modules, indecomposable rigid modules, and indecomposable ($\tau$-)rigid modules all coincide. As discussed in the introduction, Crawley-Boevey~\cite{cb} and Ringel~\cite{rin} have shown that the braid group on $n$ strands acts transitively on the set of complete exceptional sequences. More precisely, for $\mathcal{M} = (M_1,\ldots,M_n)$ a complete exceptional sequence and $1 \leq i \leq n-1$, they prove that there is a unique exceptional module $M'_{i}$ such that $\mathcal{M}' = (M_1,\ldots,M_{i-1},M'_i,M_i,M_{i+2},\ldots,M_n)$ is a complete exceptional sequence. (See~\cite[Sect.\ 1]{brt} for an explicit description of the module $M'_i$.) The generator $\sigma_i$ of the braid group then sends $\mathcal{M}$ to $\mathcal{M}'$.

Note that $\tau$-exceptional pairs will in the hereditary case be referred to as {\em exceptional pairs}.
Recalling that $\tau$-exceptional sequences and (classical) exceptional sequences coincide for hereditary algebras, the purpose of this section is to prove the following. 

\begin{theorem}[Theorem~\ref{thm:hereditary_case_intro}]\label{thm:hereditary_case}
    Let $\Lambda$ be a hereditary algebra and let $\mathcal{M}$ be a complete ($\tau$-)exceptional sequence. Then $\varphi_i(\mathcal{M}) = \sigma_i (\mathcal{M})$ for all $1 \leq i \leq n-1$. That is, the mutation of ($\tau$-)exceptional sequences that we define here coincides with the classical braid group action~\cite{cb,rin} on exceptional sequences.
\end{theorem}

We deduce Theorem~\ref{thm:hereditary_case} from the following proposition. Note that we assume throughout this section that we are working over a fixed hereditary algebra $\Lambda$, and so by Corollary~\ref{cor:mutation_complete} left and right mutation are defined for any ($\tau$-)exceptional pair or sequence.

\begin{proposition}\label{prop:hereditary_move_left}
    Let $(B,C)$ be an exceptional pair and write $\varphi(B,C) = (C',B')$. Then $B' = B$.
\end{proposition}

\begin{proof}[Proof that Proposition~\ref{prop:hereditary_move_left} implies Theorem~\ref{thm:hereditary_case}]
    Let $\mathcal{M} = (M_1,\ldots,M_n)$ be a complete exceptional sequence and let $i \in \{1,\ldots,n-1\}$. Since $\J(M_{i+2},\ldots,M_n)$ is also hereditary, Proposition~\ref{prop:hereditary_move_left} implies that both $\sigma_i(\mathcal{M})$ and $\varphi_i(\mathcal{M})$ are of the form 
    $(M_1,\ldots,M_{i-1},?,M_i, M_{i+2},\ldots,M_n)$.
    Now note that there is a unique such sequence by the results of Crawley-Boevey \cite[Lem.~8]{cb} and Ringel \cite[Prop.~3]{rin}. 
\end{proof}

In the remainder of this section, we prove Proposition~\ref{prop:hereditary_move_left} in several steps.

\begin{lemma}\label{lem:brick_gen}
    Suppose $(B,C)$ is an exceptional pair with $C \in \Gen(\Fbm_C(B))$. Then $\Fbm_C(B) = B$.
\end{lemma}

\begin{proof}
    It suffices to show that $\Hom(C, \Fbm_C(B)) = 0$, by the definition of $\Ebm$ (see Theorem~\ref{thm:bmcorr}). Let $f: C \rightarrow \Fbm_C(B)$. By assumption, there is a surjection $q: (\Fbm_C(B))^k \twoheadrightarrow C$. Let $N$ be a submodule of $(\Fbm_C(B))^k$ which is isomorphic to $\Fbm_C(B)$. Since $\Fbm_C(B)$ is  indecomposable and 
    ($\tau$-)rigid, and
    $\Lambda$ is hereditary, we have that $\End(N)$ is a division algebra. It follows that $f \circ q|_N$ is either 0 or an isomorphism. Since $f$ is not surjective (otherwise we would have $C \in \Gen \Fbm_C(B)$ and $\Fbm_C(B) \in \Gen C$, a contradiction to Lemma~\ref{lem:twogens}), this composition is zero. It follows that $f = 0$, as desired.
\end{proof}

The following is well-known. See e.g.\ \cite[Proposition 2.15]{hi_2} (items 1 and 2) and~\cite[Lemma~1.1]{brt} (item 3) for proofs.

\begin{lemma}\label{lem:hereditary_inj_surj}
    Let $(B,C)$ be an exceptional pair.
    \begin{enumerate}
        \item[(a)] Let $g: B \rightarrow C^k$ be a minimal left $\add(C)$-approximation. Then $g$ is either a monomorphism or an epimorphism.
        \item[(b)] Let $g: B^{k} \rightarrow C$ be a minimal right $\add(B)$-approximation. Then $g$ is either a monomorphism or an epimorphism.
        \item[(c)] Either $\Hom(B,C)=0$ or $\Ext^1(B,C)=0$ (or both). 
   \end{enumerate} 
\end{lemma}

\begin{proposition}\label{prop:hereditary_trivial}
    Let $(B,C)$ be an exceptional pair. If $\Ext^1(B,C) = 0$, then $\Fbm_C(B) = B$.
\end{proposition}

\begin{proof}
    By assumption, $B\amalg C$ is ($\tau$-)rigid. 
Since $\Hom(C,B) = 0$, we then have $ B = f_C(B) = \Ebm_C(B)$, and the claim follows.
\end{proof}

\begin{proposition}\label{prop:hereditary_case_III}
    Let $(B,C)$ be an exceptional pair. Suppose that $\Hom(B,C) \neq 0$, that $C \notin \P(\Lambda)$, and that $C \notin \Gen B$. Then $B \in \P(\J(C))$ and $\Ebm_{\Fbm_C(B)}(C) = \Ebm_B(C) \notin \P(\J(B))$. 
\end{proposition}

\begin{proof} 
    Since $\Hom(B,C) \neq 0$ by assumption, Lemma~\ref{lem:hereditary_inj_surj}(c) and Proposition~\ref{prop:hereditary_trivial} imply that $B = \Fbm_C(B)$. Moreover, since $C \notin \Gen B$ and $\Hom(B,C) \neq 0$ by assumption, we have $B \in \Cogen C$ by 
    Lemma~\ref{lem:hereditary_inj_surj}(a). Then for $M \in \J(C)$, the fact that $\Ext^1(C,M) = 0$ implies that $\Ext^1(B,M) = 0$ since $\Lambda$ is hereditary. We conclude that $B \in \P(\J(C))$, and hence (using Lemma~\ref{lem:fUbijection}) that $B = \Fbm_C(B) \in \P({}^\perp{\tau C})$. Since $C$ is not projective, it follows from Lemma~\ref{lem:eachinperp} that $C \notin \P({}^\perp{\tau B})$, and hence, again by Lemma~\ref{lem:fUbijection}, that 
    $\Ebm_B(C) \notin \P(\J(B))$.
\end{proof}

\begin{proposition}\label{prop:hereditary_case_IV}
    Let $(B,C)$ be an exceptional pair, and suppose that $\Ext^1(B,C) \neq 0$ and $C \notin \P(\Lambda)$. Let $\eta = (C \hookrightarrow E \twoheadrightarrow B^k)$ and $\eta' = (C^\ell \hookrightarrow E' \twoheadrightarrow B)$ be short exact sequences such that $\eta: B^k \rightarrow C[1]$ is a minimal right $\add(B)$-approximation and $\eta': B \rightarrow C^\ell[1]$ is a minimal left $\add(C[1])$-approximation.  Then the following hold.
    \begin{enumerate}
        \item[(a)] $\Fbm_C(B) = E'$.
        \item[(b)] $\Hom(E',C) = 0$ and $L := C \amalg E'$ is a gen-minimal rigid module.
        \item[(c)] $\widetilde{L}:= B \amalg E$ satisfies $\widetilde{L}\in \Gen L$ and $\tau\widetilde{L} \in \J(L)^\perp$.
        \item[(d)] $\Hom(B,E) = 0$ and $\tau \widetilde{L}$ is a cogen-minimal rigid module.
        \item[(e)] $\Ebm_{E'}(C) = C \in \P(\J(E'))$.
        \item[(f)] $\widetilde{L} = \Pmodns(\FiltGen(\J(L)))$.
    \end{enumerate}
We observe that $k$ and $\ell$ are positive and the sequences $\eta$ and $\eta'$ are not split, since $\Ext^1(B,C)\not=0$.
\end{proposition}

\begin{proof}
(a)
It suffices to prove that $E'$ is 
indecomposable and that $C\amalg E'$ is a ($\tau$-)rigid module. Indeed, we automatically have that $C \neq E'$,
so these conditions and the fact that $\Hom(C,B) = 0$ imply that $B = f_C(E') = \Ebm_C(E')$, and thus that $E' = \Fbm_C(B)$ as desired.

We use an argument similar to those appearing in \cite[Lem.~2.1 and Prop.~2.5]{hi_2} and the references therein. Note that $\eta'$ can be defined explicitly as a universal extension; that is, we have $\dim\Ext^1(B,C) = \ell$. Now note from Lemma~\ref{lem:hereditary_inj_surj} that $\Hom(B,C) = 0$ since $\Ext^1(B,C) \neq 0$. Applying $\Hom(-,C)$ to $\eta'$ thus yields an exact sequence
$$0 = \Hom(B,C) \rightarrow \Hom(E',C) \rightarrow \Hom(C^\ell,C) \xrightarrow{\Hom(\eta',C)} \Ext^1(B,C) \rightarrow$$$$\rightarrow \Ext^1(E',C) \rightarrow \Ext^1(C^\ell,C) = 0.$$
The morphism $\Hom(\eta',C)$ must be surjective since $\eta'$ is a minimal left $\add(C[1])$-approximation, and so it is bijective since its domain and codomain have the same dimension. Thus $\Ext^1(E',C) = 0$ and $\Hom(E',C) = 0$. We also have that $\Ext^1(C,E') = 0$ because $\Ext^1(C,C^\ell \amalg B) =~0$.

It remains to show that $E'$ is indecomposable and $(\tau$-)rigid. 
Let $q: E' \twoheadrightarrow B$ be the surjection in the short exact sequence $\eta'$. Applying $\Hom(E',-)$ to $\eta'$ yields an exact sequence
$$0 = \Hom(E',C^\ell) \rightarrow \Hom(E',E') \xrightarrow{\Hom(E',q)} \Hom(E',B) \rightarrow \Ext^1(E',C^\ell) = 0 \rightarrow$$$$\rightarrow \Ext^1(E',E') \rightarrow \Ext^1(E',B) \rightarrow 0.$$
Applying $\Hom(-,B)$ to $\eta'$ yields an exact sequence
$$0 \rightarrow \Hom(B,B) \xrightarrow{\Hom(q,B)} \Hom(E',B) \rightarrow \Hom(C^\ell,B) = 0 \rightarrow \Ext^1(B,B) = 0 \rightarrow$$$$\rightarrow \Ext^1(E',B) \rightarrow \Ext^1(C^\ell,B) = 0.$$
 We see directly from combining these long exact sequences that $\Ext^1(E',E') \simeq \Ext^1(E',B) = 0$, so $E'$ is rigid. To see that $E'$ is indecomposable, we again combine these long exact sequences to obtain vector space isomorphisms
$$\Hom(E',E') \xrightarrow{\Hom(E',q)} \Hom(E',B) \xleftarrow{\Hom(q,B)} \Hom(B,B),$$
where $q:E'\rightarrow B$ is the map in $\eta'$. 
The composition $\Hom(q,B)^{-1}\circ \Hom(E',q)$
yields a vector space isomorphism $\Phi: \Hom(E',E') \rightarrow \Hom(B,B)$ characterised by the property that, for $f \in \Hom(E',E')$, the morphism $\Phi(f)$ is the unique element of $\Hom(B,B)$ which satisfies $q \circ f = \Phi(f) \circ q$. For $f, f' \in \Hom(E',E')$, we then compute
$$\Phi(f \circ f')\circ q = q \circ f \circ f' = \Phi(f) \circ q \circ f' = \Phi(f) \circ \Phi(f') \circ q.$$
The characterising property of $\Phi$ then implies that $\Phi(f) \circ \Phi(f') = \Phi(f\circ f')$, so $\Phi: \End(E') \rightarrow \End(B)$ is a $K$-algebra isomorphism. Since $B$ is indecomposable by assumption, we conclude that $\End(E') \simeq \End(B)$ is a local algebra, and thus that $E'$ is indecomposable. This concludes the proof of (a).

    (b) We showed that $E'$ is indecomposable, that $\Hom(E',C) = 0$, and that $C\amalg E'$ is ($\tau$-)rigid (and basic) in the proof of (a). Thus to prove gen-minimality of $L$, we need only show that $E' \notin \Gen C$. 
    For $C^p \in \add C$, applying $\Hom(C^p,-)$ to $\eta'$ yields an exact sequence
    $$0 \rightarrow \Hom(C^p,C^\ell) \rightarrow \Hom(C^p,E') \rightarrow \Hom(C^p,B) = 0.$$
    Thus any morphism $C^p \rightarrow E'$ factors through the injection $C^\ell \hookrightarrow E'$ comprising $\eta'$; that is, this injection is a right $\add(C)$-approximation. In particular, this means $E' \notin \Gen C$, completing the proof of (b).
    Hence, $L = C \amalg E'$ is
gen-minimal.

(c) Since $B\in \Gen E'$, we have $B\in \Gen L$. Since $L$ is rigid by (b), the subcategory $\Gen L$ is closed under extensions. Hence, since $B,C\in \Gen L$, also $E\in \Gen 
    L$, giving $\widetilde{L}=B\amalg E\in \Gen L$.

    Let $M\in \J(L)$. Then $\Hom(L,M)=0$ and $\Ext^1(L,M)=0$, so $\Hom(C,M)=0$, $\Hom(E',M)=0$, $\Ext^1(C,M)=0$ and $\Ext^1(E',M)=0$. Applying $\Hom(-,M)$ to $\eta'$ yields the exact sequence:
    $$\Hom(C^{\ell},M)\rightarrow \Ext^1(B,M)\rightarrow \Ext^1(E',M)$$
    in which the first and last terms are zero, so $\Ext^1(B,M)=0$. Applying $\Hom(-,M)$ to $\eta$ yields the exact sequence:
    $$\Ext^1(B^k,M)\rightarrow \Ext^1(E,M)\rightarrow \Ext^1(C,M)$$
    in which the first and last terms are zero, so $\Ext^1(E,M)=0$. Hence, using the fact that $\Lambda$ is hereditary, we have
    $\Ext^1(\widetilde{L},M)= \Hom(M,\tau \widetilde{L})=0$ as required.

    (d) We first note that $E$ is indecomposable, that $\widetilde{L}$ is a basic rigid module, that $\Hom(B,E) = 0$, and that $E \notin \Cogen B$ by an argument dual to that used to prove (a) and (b).
    Note that $\Ext^1(B,C)\not=0$ by assumption, so $B\not\in\P(\Lambda)$. 
Furthermore, $C$ is a non-projective submodule of $E$, and so $E$ is not projective by the hereditary property. 
    It follows that $\tau \widetilde{L}$ is a basic rigid module with two indecomposable direct summands. It remains to show that $\tau\widetilde{L}$ is cogen-minimal. To see this, first note that $0 = \Hom(B,E) = \Hom(\tau B, \tau E)$ since $\Lambda$ is hereditary and $B,E \notin \P(\Lambda)$. Now let $f: \tau E \rightarrow \tau B^m$ be a minimal left $\add(\tau B)$-approximation, and note that $\tau E \in \Cogen(\tau B)$ if and only if $f$ is a monomorphism. Now since $E \notin \P(\Lambda)$, the map $f$ corresponds functorially and bijectively to a morphism $f': E \rightarrow B^m$. Moreover, $f'$ is a minimal left $\add(B)$-approximation by functoriality, and so $m = k$ and $\ker f' = C$. Since $C\notin \P(\Lambda)$, the inclusion $C \hookrightarrow E$ corresponds functorially and bijectively to a nonzero morphism $h: \tau C \rightarrow \tau E$ which satisfies $f \circ h = 0$. We conclude that $f$ is not a monomorphism, and thus that
    $\tau E \notin \Cogen \tau B$.

    (e) Let $M \in \J(E')$. Then $\Ext^1(E',M) = 0$ by construction. Since $C \in \Cogen E'$, the fact that $\Hom(E',C)=0$ from (b) implies that $\Ebm_{E'}(C) = C \in \P(\J(E'))$.

    (f) 
    It is clear from the definitions of $E$ and $E'$ that $\J(L) = \J(\widetilde{L}) = \Jinv(\tau\widetilde{L})$. By (d) and Proposition~\ref{prop:gen_minimal} (taking $\F = \Cogen \tau \widetilde{L}$), it follows that \begin{equation}\label{eqn:hereditary_6} \tau \widetilde{L} = \Imods(\Jinv(\tau\widetilde{L})^\perp) = \Imods(\J(L)^\perp).\end{equation}
    Recall from the proof of (a) that $\Hom(E',C) = 0$. Moreover, since $\ell>0$ and $C \notin \P(\Lambda)$, the fact that $\Lambda$ is hereditary implies that $E' \notin \P(\Lambda)$. This, together with \cite[Lemma~6.1]{happel}, implies that $E' \notin \Pmods({}^\perp \tau C)$. Equivalently, $B = \Ebm_C(E') \notin \P(\J(C))$ by Lemma~\ref{lem:fUbijection}. Since $C \notin \P(\Lambda)$, Proposition~\ref{prop:JJCBsincere} and Theorem~\ref{thm:J_E} then say that $\J(B,C) = \J(L)$ is sincere. Applying $\tau^{-1}$ to \eqref{eqn:hereditary_6} and using Proposition~\ref{prop:sincere_summary_2}(a'), we conclude that $\Pns(\FiltGen(\J(L))) = \tau^{-1}\Imods(\J(L)^\perp) = \widetilde{L}$. (Note that $\widetilde{L} = B \amalg E$ has no projective direct summands by the proof of (d) and that $\Imods(\J(L)^{\perp})$ has no injective direct summands by~\eqref{eqn:hereditary_6}.)
\end{proof}

We are now prepared to prove Proposition~\ref{prop:hereditary_move_left}.

\begin{proof}[Proof of Proposition~\ref{prop:hereditary_move_left}]

Suppose first that $(B,C)$ is left regular.
Then $C \not\in \P({}^\perp{\tau \Fbm_C(B)})$ or $C \in \P(\Lambda)$. 
Recall from Definition-Proposition~\ref{defprop:forward} that $B'=B_{C\uparrow}=\Fbm_{C_+}(B)$.
We consider three cases:
\medskip

\noindent{\bf Case I:} Suppose $C \in \Gen \Fbm_C(B)$. Then $B' = \Fbm_C(B) = B$ by Lemma~\ref{lem:brick_gen}.

\medskip

\noindent{\bf Case II:} Suppose $C \in \P(\Lambda)$.
Then $B'=\Fbm_{C[1]}(B)=B$ (see Theorem~\ref{thm:bmcorr}).

\medskip

\noindent{\bf Case III:} Suppose that $C \notin \P(\Lambda)$, and that $C \notin \Gen(\Fbm_C(B))$. By assumption, $C \notin \P({}^\perp \tau \Fbm_C(B))$. Thus, by Lemma~\ref{lem:fUbijection}, we have $\Ebm_{\Fbm_C(B)}(C) \notin \P(\J(\Fbm_C(B)))$. 
Suppose for a contradiction that $\Ext^1(B,C) \neq 0$. Then Proposition~\ref{prop:hereditary_case_IV}(a,e) implies $\Ebm_{\Fbm_C(B)}(C) = C \in 
\P(\J(\Fbm_C(B)))$. This is a contradiction, and hence $\Ext^1(B,C) =0$. 
It follows that $B' = \Fbm_C(B) = B$ by Proposition~\ref{prop:hereditary_trivial}.
\medskip 

The claim of the proposition is thus proved for $(B,C)$ being left regular.
Suppose  that $(B,C)$ is left irregular. Then $C \in \P({}^\perp{\tau \Fbm_C(B)})$ and $C \notin \P(\Lambda)$.
Suppose for a contradiction that $\Ext^1(B,C) = 0$. Then $\Fbm_C(B) = B$ by Proposition~\ref{prop:hereditary_trivial}. Hence, we have $C \in \P({}^\perp \tau B)$, so $C$ is a direct summand of the Bongartz complement of $B$.
The assumption that $\Hom(C,B) = 0$ and \cite[Lem.~6.1]{happel} then imply that $C \in \P(\Lambda)$, a contradiction. We conclude that $\Ext^1(B,C) \neq 0$.

Now denote $L = C \amalg \Fbm_C(B)$ and $\widetilde{L} = \Pmodns(\FiltGen(\J(L)))$. By Proposition~\ref{prop:hereditary_case_IV}(a,f), these are precisely the modules $L$ and $\widetilde{L}$ defined in Proposition~\ref{prop:hereditary_case_IV}. In particular, Proposition~\ref{prop:hereditary_case_IV}(c) yields a decomposition $\widetilde{L} = B \amalg E$ with $B \in \Gen E$. Proposition~\ref{prop:JUsinceresetup1}(g) and the fact that $B \in \Gen E$ then imply that $E = \Pmods(\Gen \Pmodns(\FiltGen(\J(L))))$. The definition of $B'$ (Definition-Proposition~\ref{defprop:forward2}) then yields $B' = \widetilde{L}/E = B$.
\end{proof}


\section{Mutation in rank two}\label{sec:rank_2_exchange}

\subsection{Connection with $\tau$-tilting mutation}
We recall background information about the mutation of support $\tau$-tilting modules from \cite[Sect.\ 2.3-2.4]{air}. Let $M,N \in \mods\Lambda$ be support $\tau$-tilting modules. Then $M$ and $N$ are \emph{related by a mutation} if either (i) there exists an indecomposable module $X$ such that $M = N\amalg X$ or $N = M\amalg X$, or (ii) there exist indecomposable modules $X, Y$ and a
$\tau$-rigid module $L$ such that $M = L \amalg X$ and $N = L \amalg Y$. One further says that $M$ is a \emph{left mutation} of $N$ (and thus also that $N$ is a \emph{right mutation} of $M$) if also 
$M \in \Gen N$. The \emph{oriented exchange graph} of $\Lambda$ can then be realised as the quiver $\Hasse(\stt\Lambda)$ with vertices the support $\tau$-tilting modules and an arrow $N \rightarrow M$ whenever $M$ is a left mutation of $N$.

Throughout this section, we assume that $\Lambda$ has rank two. We denote by $\upsilon$ the natural map from the set of TF-ordered $\tau$-tilting modules to the set of $\tau$-tilting modules. We also recall from Theorem~\ref{thm:tfo} the bijective map $\tfo$ from the set of complete $\tau$-exceptional sequences to the set of TF-ordered $\tau$-tilting modules.

We identify indices mod $2$ throughout this section. 
Let $P_i$ be the indecomposable projective, with corresponding simple $S_i$, for $i \in \{1,2\}$. 
Furthermore, let $S_{i+1}^e$ be the indecomposable generator-cogenerator of the Serre subcategory $P_i^\perp = \Filt(S_{i+1})$. Note that $S_{i+1}^e$ is support $\tau$-tilting. We denote by $R_i$ the co-Bongartz complement of $P_i$. Note that $R_i = 0$ if and only if $P_i = S_i^e$. For example, when $\Lambda = K(1\rightarrow 2)$ we have $R_1 = 0$ and $R_2 = S_1^e = S_1$. 

\begin{remark}\label{rem:Hasse_shape}
    The shape of $\Hasse(\stt\Lambda)$ is as follows. Every vertex has valence $2$. The unique source is the vertex labelled $P(\Lambda)$, and its two outgoing arrows are $P(\Lambda) \rightarrow P_1 \amalg R_1$ and $P(\Lambda) \rightarrow P_2\amalg R_2$. The unique sink is the vertex labelled 0, and its two incoming arrows are $S_1^e \rightarrow 0$ and $S_2^e \rightarrow 0$. (The modules $0, S_1^e$, and $S_2^e$ are the only support $\tau$-tilting modules which are not $\tau$-tilting.) Every other vertex is the source of one arrow and the target of one arrow. Thus if $M$ is a support $\tau$-tilting module different from $0$ and $P(\Lambda)$, then $M$ has a single left mutation and a single right mutation.
\end{remark}

Let $\Hasse(\tex \Lambda)$ denote the quiver whose vertices are the complete $\tau$-exceptional pairs with arrows given by:
\begin{equation}\label{eq:ctauex}
(B,C)\rightarrow (B',C') \iff (B,C)=\varphi(B',C').\end{equation}
In this section, we explain how $\Hasse(\stt\Lambda)$ and $\Hasse(\tex\Lambda)$ are related (for algebras of rank two). The main result is that $\Hasse(\tex\Lambda)$ has the same number of connected components as $\Hasse(\stt\Lambda)$. As a consequence, in rank two, mutation of $\tau$-exceptional sequences is transitive if and only if mutation of support $\tau$-tilting modules is transitive (see Corollary~\ref{cor:connected}). In particular, by \cite[Cor.~2.34]{air}, mutation of $\tau$-exceptional sequences is transitive for all $\tau$-tilting finite algebras of rank two. Using \cite[Prop.~3.5]{bmrrt}, Corollary~\ref{cor:connected} also gives a new proof that mutation of ($\tau$-)exceptional sequences is transitive for all hereditary algebras of rank two. (Recall that mutation is transitive for hereditary algebras of arbitrary rank \cite[Thm.]{cb},~\cite[Sect.\ 7 Cor.]{rin}.)

\begin{lemma}\label{lem:rank2_one_order}
    Let $M$ be a $\tau$-tilting module. If $M$ is not projective, then it admits exactly one TF-ordering.
\end{lemma}

\begin{proof}
    It follows from Lemma ~\ref{lem:twogens} that $M$ 
    has at least one TF-ordering. (Note that this was proved more 
    generally in~\cite{mt}.)
    Suppose that $M$ admits more than one TF-ordering. Then $M$ must be gen-minimal. Moreover, since $M$ is $\tau$-tilting, we have $\J(M) = 0$. It follows from Theorem~\ref{thm:gen_minimal} that $\Gen M = {}^\perp{0} = \mods\Lambda$; i.e., that $M$ is projective.
\end{proof}

\begin{lemma}\label{lem:rank2_no_case_iii}
    Let $(B,C)$ be a complete $\tau$-exceptional sequence.
    Then $(B,C)$ is left regular and either
    $C\in \Gen \Fbm_C(B)$ or $C\in \P(\Lambda)$.
\end{lemma}

\begin{proof}
    Suppose $C \notin \P(\Lambda)$. Then $(\Fbm_C(B),C)$ is the only TF-ordering of $\Fbm_C(B) \amalg C$ by Lemma
    ~\ref{lem:rank2_one_order}. This means that $C \in \Gen \Fbm_C(B)$.
\end{proof}

Note that it is a consequence of Lemma~\ref{lem:rank2_no_case_iii} that all complete $\tau$-exceptional sequences are also left mutable, so this gives an alternative proof for part of Theorem~\ref{prop:mutation_complete}. 

By definition (see~\eqref{eq:ctauex}), there is an arrow $(B,C) \rightarrow (B',C')$ in $\Hasse(\tex\Lambda)$ if and only if $(B,C) = \varphi(B',C')$. Recall also, by Theorem~\ref{thm:bmcorr}, that when $C'$ is projective, we have
$\Fbm_{C'[1]}(B') = B'$.
So, combining this with Lemma~\ref{lem:rank2_no_case_iii} and Definition-Proposition~\ref{defprop:forward}, we have that 
\begin{equation}\label{eq:lMutRankTwo} 
\varphi(B',C') = 
\begin{cases}
  (\Ebm_{\Fbm_{C'}(B')}(C')[-1], \Fbm_{C'}(B'))  & \text{ if } C' \in \Gen \Fbm_{C'}(B') \\
  (\Ebm_{B'}(C'[1])[-1], B') & \text{ if } C' \in \P(\Lambda)
\end{cases}
\end{equation}

\begin{proposition}\label{prop:arrows_rank_2_a}
    Let $(B,C)$ and $(B',C')$ be complete $\tau$-exceptional sequences, and suppose that neither $C$ nor $C'$ is projective. Then the following are equivalent.
    \begin{enumerate}
        \item[(a)] There is an arrow $(B,C) \rightarrow (B',C')$ in $\Hasse(\tex\Lambda)$.
        \item[(b)] $C = \Fbm_{C'}(B')$.
        \item[(c)] There is an arrow $\upsilon\circ\tfo(B,C) \rightarrow \upsilon\circ\tfo(B',C')$ in $\Hasse(\stt\Lambda)$.
    \end{enumerate}
\end{proposition}

\begin{proof}

($a\implies b$): This follows from Equation \eqref{eq:lMutRankTwo} (noting~\eqref{eq:ctauex}).

($b\implies a$): This follows also from Equation \eqref{eq:lMutRankTwo}, using in addition that $C$ uniquely determines $B$, by Theorem~\ref{thm:unique}.

    ($b \implies c$): Suppose $C = \Fbm_{C'}(B')$. Then $\upsilon\circ\tfo(B,C)$ and $\upsilon\circ\tfo(B',C')$ are either equal or related by a mutation,
    since they have a common direct summand. 
    By Lemma \ref{lem:rank2_one_order}, they cannot be equal. Then Lemma~\ref{lem:rank2_no_case_iii}  and the fact that $C' \notin \P(\Lambda)$ imply that $C' \in \Gen \Fbm_{C'}(B')$; i.e., that $C'$ is the co-Bongartz complement of $\Fbm_{C'}(B') = C$. It follows from~\cite[Thm. 2.18]{air} and Lemma~\ref{lem:Bongartzsplit}(c) that $\Fbm_C(B)$ must be the Bongartz complement of $C$, and so condition (c) holds.

    ($c \implies b$): Suppose that there is an arrow $\upsilon\circ\tfo(B,C) \rightarrow \upsilon\circ\tfo(B',C')$ in $\Hasse(\stt\Lambda)$. Since $C$ is not projective, Lemma~\ref{lem:rank2_one_order} implies that $C \in \Gen \Fbm_C(B)$, and hence that $C$ is the co-Bongartz complement of $\Fbm_C(B)$ and $\Fbm_C(B)$ is the Bongartz complement of $C$. It follows that $C$ is a direct summand of $\Fbm_{C'}(B') \amalg C'$. Now note that $C \neq C'$, since otherwise $B = B'$ by Theorem~\ref{thm:unique}, contradicting the existence of an arrow $\upsilon\circ\tfo(B,C) \rightarrow \upsilon\circ\tfo(B',C')$ in $\Hasse(\stt\Lambda)$. We conclude that $C = \Fbm_{C'}(B')$, as desired. 
\end{proof}

\begin{proposition}\label{prop:arrows_rank_2_b}
    Let $(B,C)$ and $(B',C')$ be complete $\tau$-exceptional sequences, and suppose that $C \in \P(\Lambda)$ and $C' \notin \P(\Lambda)$. Then the following are equivalent.
    \begin{enumerate}
        \item[(a)] There is an arrow $(B,C) \rightarrow (B',C')$ in $\Hasse(\tex\Lambda)$.
        \item[(b)] $C = \Fbm_{C'}(B')$.
        \item[(c)] $\upsilon\circ\tfo(B,C) = P(\Lambda)$, $\Hom(C',B) = 0$, and there is an arrow $P(\Lambda) \rightarrow \upsilon\circ \tfo(B',C')$ in $\Hasse(\stt\Lambda)$.
    \end{enumerate}
\end{proposition}

\begin{proof}
    ($a\iff b$): This can be seen by arguing as in the proof of $(a\iff b)$ in  Proposition~\ref{prop:arrows_rank_2_a}.
    ($b \implies c$): Suppose $C = \Fbm_{C'}(B')$. Then, arguing as in the proof of $(b\implies c)$ in Proposition~\ref{prop:arrows_rank_2_a}, we have that $\upsilon\circ\tfo(B,C)$ and $\upsilon\circ\tfo(B',C')$ are related by a mutation. As before, Lemma~\ref{lem:rank2_no_case_iii} and the fact that $C' \notin \P(\Lambda)$ imply that $C' \in \Gen \Fbm_{C'}(B')$; i.e., that $C'$ is the co-Bongartz complement of $\Fbm_{C'}(B') = C$. It follows that $\Fbm_C(B)$ must be the Bongartz complement of $C$. As $C$ is projective, the Bongartz completion of $C$ is $P(\Lambda)$, so this implies that there is an arrow $P(\Lambda) \rightarrow \upsilon\circ \tfo(B',C')$ in $\Hasse(\stt\Lambda)$. Now note that $\Hom(C,B) = 0$ since $(B,C)$ is a $\tau$-exceptional sequence. The fact that $C' \in \Gen \Fbm_{C'}(B') = \Gen C$ thus implies that $\Hom(C',B) = 0$ as well.

    ($c \implies b$): Suppose (c) holds. The fact that $\upsilon\circ \tfo(B,C) = P(\Lambda)$ implies that $\Fbm_C(B)$, and thus also its factor module $B =\Ebm_C(\Fbm_C(B)) = f_{C}(\Fbm_C(B))$, must have a simple top $S$. 
    Since $C$ is projective, the module $B$ lies in the Serre subcategory $C^\perp = \Filt(S)$. In particular, this means the socle of $B$ lies in $\add S$, and so $\Hom(C',B) = 0$ implies $\Hom(C',S) = 0$. This implies that the projective cover of $C'$ lies in $\add C$, and hence that $C' \in \Gen C$. 

    Now since $C' \notin \P(\Lambda)$, the existence of an arrow $P(\Lambda) \rightarrow \upsilon\circ \tfo(B',C')$ in $\Hasse(\stt\Lambda)$ implies that $\Fbm_{C'}(B') \in \P(\Lambda)$. Moreover, Lemma~\ref{lem:rank2_no_case_iii} implies that $C' \in \Gen \Fbm_{C'}(B')$. By the uniqueness of projective covers, it follows that $C = \Fbm_{C'}(B')$.
\end{proof}

\begin{proposition}\label{prop:arrows_rank_2_c}
    Let $(B,C)$ and $(B',C')$ be complete $\tau$-exceptional sequences, and suppose that $C' \in \P(\Lambda)$. Then the following are equivalent.
    \begin{enumerate}
        \item[(a)] There is an arrow $(B,C) \rightarrow (B',C')$ in $\Hasse(\tex\Lambda)$.
        \item[(b)] $C = B'$.
        \item[(c)] $\Hom(C',C) = 0$ and there is an arrow $\upsilon \circ \tfo(B,C) \rightarrow C$ in $\Hasse(\stt\Lambda)$.
    \end{enumerate}
\end{proposition}

\begin{proof}

 ($a\iff b$): This can be seen by arguing as in the proof of Proposition~\ref{prop:arrows_rank_2_a} and using the fact that $\Fbm_{C'[1]}(B') = B'$ (see Theorem~\ref{thm:bmcorr}).
 
    ($b\implies c$): Suppose $C = B'$. Since $(B',C')$ is a $\tau$-exceptional sequence, this implies that $\Hom(C',C) = 0$. Equivalently, we have that $C \amalg C'[1]$ is a support $\tau$-tilting pair, and thus the co-Bongartz complement of $C$ is zero. It follows that $\Fbm_C(B)$ is the Bongartz complement of $C = B'$, and so there is an arrow $\upsilon \circ \tfo(B,C) \rightarrow C$ in $\Hasse(\stt\Lambda)$.

    ($c \implies b$): Suppose (c) holds. Since $(B',C')$ is a $\tau$-exceptional sequence, we have that $B'$ is the unique indecomposable module which is $\tau_{\J(C')}$-rigid. On the other hand, we have that $\J(C'[1]) = \J(C')$, and so $C = \Ebm_{C'[1]}(C)$ 
    is $\tau_{\J(C')}$-rigid by Theorem~\ref{thm:bmcorr}. 
    We conclude that $C = B'$.
\end{proof}

\begin{lemma}\label{lem:simples_rank_2}
    Let $(B,C)$ be a $\tau$-exceptional sequence and $(B',C') = \psi(B,C)$. Then the following are equivalent.
    \begin{enumerate}
        \item[(a)] $C = S_i^e$ for some $i \in \{1,2\}$.
        \item[(b)] $C' \in \P(\Lambda)$.
        \item[(c)] $\upsilon\circ \tfo (B',C') = P(\Lambda)$.
    \end{enumerate}
    Moreover, if these equivalent conditions hold, then $B' = C$ and $C' = P_{i+1}$ for $i$ realising $C = S_i^e$.
\end{lemma}

\begin{proof}
    Since $\Lambda$ has rank two, we have that $(S_i^e,P_{i+1})$ is the unique $\tau$-exceptional sequence of the form $(-,P_{i+1})$. From Equation \eqref{eq:lMutRankTwo}, we have $\varphi(S_i^e,P_{i+1}) = (X_i,S_i^e)$ for $X_i = \Ebm_{S_i^e}(P_{i+1}[1])[-1]$. Again since $\Lambda$ has rank two, this is the unique $\tau$-exceptional sequence of the form $(-,S_i^e)$. The fact that $\varphi$ and $\psi$ are inverses then implies the equivalence between (a) and (b). This also shows the ``moreover'' part of the statement. Finally, the equivalence between (b) and (c) follows from the observation that $\Fbm_{P_{i+1}}(S_i^e) = P_i$, since $\Ebm_{P_{i+1}}(P_i) = f_{P_{i+1}}(P_i) = S_i^e$.
\end{proof}

We now prove the main result of this section,
which, for algebras of rank two, gives a precise relationship between
$\Hasse(\stt\Lambda)$ and $\Hasse(\tex\Lambda)$.

\begin{theorem}\label{thm:rank_2_hasse}
    Let $Q$ be the quiver obtained from $\Hasse(\stt\Lambda)$ by doing the following.
    \begin{itemize}
        \item Delete the vertices $P(\Lambda)$ and $0$ and all arrows incident to these vertices.
        \item Add an arrow from the vertex $S_1^e$ to the vertex $P_2 \amalg R_2$.
        \item Add an arrow from the vertex $S_2^e$ to the vertex $P_1 \amalg R_1$.
    \end{itemize}
    Then there is an isomorphism of quivers $\rho: \Hasse(\tex\Lambda) \rightarrow Q$ given by
    $$\rho(B,C) = \begin{cases}
        C & C \in \{S_1^e,S_2^e\} \\
        \upsilon \circ \tfo \circ \psi(B,C) & \textnormal{otherwise.}\\
    \end{cases}$$
\end{theorem}

Note that $\rho(B,C)$ is a $\tau$-tilting module whenever $C \notin \{S_1^e, S_2^e\}$. 

\begin{proof}
    Note that, by Remark~\ref{rem:Hasse_shape}, the quiver $Q$ has the property that every vertex has precisely one outgoing arrow and precisely one incoming arrow.

    We first show that $\rho$ is well-defined as a map on vertices. Let $(B,C)$ be a complete $\tau$-exceptional sequence. If $C \in \{S_1^e,S_2^e\}$, then it is clear that $\rho(B,C)$ is a vertex of $Q$. Thus suppose $C \notin \{S_1^e,S_2^e\}$. Then $\upsilon\circ\tfo \circ \psi(B,C) \neq P(\Lambda) = P_1 \amalg P_2$ by Lemma~\ref{lem:simples_rank_2}. We conclude that $\rho(B,C)$ is a vertex of $Q$.

    We next show that $\rho$ is injective on vertices. Suppose $\rho(B,C) = \rho(B',C')$. Suppose first that $C = S_i^e$ for $i \in \{1,2\}$. Since $\upsilon\circ \tfo\circ \psi(B',C')$ is $\tau$-tilting, it follows that $C' = S_i^e$, and thus that $(B,C) = (B',C')$ by Theorem~\ref{thm:unique}. Without loss of generality, we can thus assume that $\{C,C'\} \cap \{S_1^e,S_2^e\} = \emptyset$. Then $\rho(B,C) \neq P(\Lambda)$ by Lemma \ref{lem:simples_rank_2}, and hence $\rho(B,C)$ admits exactly one TF-ordering by Lemma~\ref{lem:rank2_one_order}. 
    The same argument holds for $(B',C')$, and thus $(B,C) = (B',C')$ since $\tfo$ and $\psi$ are both injective. We conclude that $\rho$ is injective on vertices.

    We next show that $\rho$ is surjective on vertices. Let $M$ be a vertex of $Q$. If $M$ is not $\tau$-tilting, then either $M = S_1^e$ or $M = S_2^e$, both of which are in the image of $\rho$. Thus suppose that $M$ is $\tau$-tilting. Then there exists a TF-ordering $(M_1,M_2)$ of $M$ with corresponding $\tau$-exceptional sequence $(M_1',M_2) = (\Ebm_{M_2}(M_1), M_2)$. Let $(B,C) = \varphi(M_1',M_2)$. Lemma~\ref{lem:simples_rank_2} and the assumption that $M \neq P(\Lambda)$ then imply that $C \notin \{S_1^e,S_2^e\}$. The fact that $\varphi$ and $\psi$ are inverses thus implies that $M = \rho(B,C)$. We conclude that $\rho$ is surjective on vertices.

    Now let $(B,C)$ and $(B',C')$ be distinct $\tau$-exceptional sequences. In particular, this means $C \neq C'$ and $B \neq B'$ by Theorem~\ref{thm:unique}. It remains to show that there is an arrow $(B,C) \rightarrow (B',C')$ in $\Hasse(\tex\Lambda)$ if and only if there is an arrow $\rho(B,C) \rightarrow \rho(B',C')$ in $Q$. For readability, denote $(B_0,C_0) = \psi(B,C)$ and $(B'_0,C'_0) = \psi(B',C')$. Since $\varphi$ and $\psi$ are inverses, there is an arrow $(B_0,C_0) \rightarrow (B'_0,C'_0)$ in $\Hasse(\tex\Lambda)$ if and only if there is an arrow $(B,C) \rightarrow (B',C')$ in $\Hasse(\tex\Lambda)$. Thus it suffices to show that there is an arrow $(B_0,C_0) \rightarrow (B'_0,C'_0)$ in $\Hasse(\tex\Lambda)$ if and only if there is an arrow $\rho(B,C) \rightarrow \rho(B',C')$ in $Q$. We prove this by breaking the result down into several cases.

    \smallskip
    \noindent {\bf Case I:} Suppose that neither $C_0$ nor $C'_0$ is projective. By Lemma~\ref{lem:simples_rank_2}, this is equivalent to assuming $\{C,C'\} \cap \{S_1^e,S_2^e\} = \emptyset$. Thus $\rho(B,C) = \upsilon\circ \tfo(B_0,C_0)$ and $ \rho(B',C') = \upsilon\circ \tfo(B'_0,C'_0)$. 
    Proposition~\ref{prop:arrows_rank_2_a} then implies that there is an arrow $(B_0,C_0) \rightarrow (B'_0,C'_0)$ in $\Hasse(\tex\Lambda)$ if and only if there is an arrow $\rho(B,C) \rightarrow \rho(B',C')$ in $Q$, as desired.

   \smallskip

\noindent{\bf Case II:}
Suppose that $C_0 \in \P(\Lambda)$ and $C'_0 \notin \P(\Lambda)$. 
Using that $C_0 \in \P(\Lambda)$, it follows from 
Lemma~\ref{lem:simples_rank_2} that 
$C= S_i^e = B_0$ for some $i\in \{1,2\}$, that $\upsilon\circ \tfo(B_0,C_0) = P(\Lambda)$,
and that
$C_0 = P_{i+1}$.
Furthermore, using that $C'_0 \notin \P(\Lambda)$, it also follows from 
Lemma~\ref{lem:simples_rank_2} that $C' \notin \{S_1^e, S_2^e\}$. 

Assume first there is an arrow 
$\rho(B,C) \to \rho(B',C')$ in $Q$.
Since $C= S_i^e$ we have $\rho(B,C) =S_i^e$, and since $C' \notin \{S_1^e, S_2^e\}$, it follows from the definition of $Q$ that
$\rho(B',C') = \upsilon\circ \tfo(B_0',C_0') = P_{i+1} \amalg R_{i+1}$. So there is an arrow $P(\Lambda) \to \upsilon\circ \tfo(B_0',C_0')$ in $\Hasse(\stt\Lambda)$ by Remark~\ref{rem:Hasse_shape}.

Note that $C_0'$ is a direct summand of  $\upsilon\circ \tfo(B_0',C_0') = P_{i+1} \amalg R_{i+1}$. Since $C_0'\notin \P(\Lambda)$, this means $C_0' = R_{i+1}.$
Since $R_{i+1}$ is the co-Bongartz complement of $P_{i+1}$, we have that $R_{i+1} \in \Gen P_{i+1}$.
Thus $\Hom(R_{i+1}, B_0) = 0$, since $\Hom(P_{i+1}, B_0) =0$.
Summarising, we have (i) $\upsilon\circ \tfo(B_0,C_0) = P(\Lambda)$,
(ii) $\Hom(C_0',B_0)= \Hom(R_{i+1}, B_0) = 0$ and (iii) 
there is an arrow $P(\Lambda) \to \upsilon\circ \tfo(B_0',C_0')$ in $\Hasse(\stt\Lambda)$. It then follows from Proposition~\ref{prop:arrows_rank_2_b} that there is an arrow
$(B_0,C_0) \to (B_0',C_0')$ in $\Hasse(\tex\Lambda)$.

For the converse, assume there is an arrow 
$(B_0,C_0) \to (B_0',C_0')$ in $\Hasse(\tex\Lambda)$.
Proposition~\ref{prop:arrows_rank_2_b} then implies that 
(i) $\upsilon\circ \tfo(B_0,C_0) = P(\Lambda)$, that (ii)
$\Hom(C_0',B_0)=0$ and that (iii) there is an arrow 
$P(\Lambda) \to \upsilon\circ \tfo(B'_0,C'_0)$ in $\Hasse(\stt\Lambda)$.
Then (iii) implies that $\upsilon\circ \tfo(B'_0,C'_0) = P_j \amalg R_j$ for some $j \in \{1,2\}$. 
We must have $C_0' = R_j$, since $C_0'\notin \P(\Lambda)$. Since $\rho$ is bijective on vertices and $C' \notin \{S_1^e,S_2^e\}$, we also have $$P_j \amalg R_j = \upsilon\circ \tfo(B'_0,C'_0) = \upsilon\circ \tfo\circ \psi(B',C') = \rho(B',C')\notin \{S_1^e,S_2^e\},$$ and so $R_j \neq 0$ (see the comment preceding Remark~\ref{rem:Hasse_shape}). Thus $\Hom(R_j,S_j^e) \neq 0$ since $S_j$ is the simple top of $P_j$ and $0 \neq R_j \in \Gen P_j$. Hence, recalling that $B_0 = S_i^e$, we have that $j = i+1$ by (ii).
By construction, there is an arrow 
$S_i^e \to P_{i+1} \amalg R_{i+1}$ in $Q$. 
This is precisely an arrow 
$S_i^e = \rho(B,C) \to \rho(B',C') = \upsilon\circ \tfo(B'_0,C'_0) =
P_{i+1} \amalg R_{i+1}$. This concludes the proof for case II.
    
   \smallskip
   \noindent {\bf Case III:} Suppose that $C_0' \in \P(\Lambda)$. By Lemma~\ref{lem:simples_rank_2}, this is equivalent to assuming that $B'_0 = C' = S_i^e$ and $C'_0 = P_{i+1}$ for some $i \in \{1,2\}$. Thus $\rho(B',C') = S_i^e$. We now have two subcases to consider.

    \smallskip
    
   \noindent {\bf Case III(a):} Suppose first that $S_i^e(=C')$ is projective; i.e., that $P_i = S_i^e$. Then $B' = S_{i+1}^e$ and $\rho(B',C') = S_i^e =
   P_i \amalg R_i$. Thus, by the definition of $Q$, there is an arrow $\rho(B,C) \rightarrow \rho(B',C')$ in $Q$ 
   if and only if $\rho(B,C) = S^e_{i+1}$. On the other hand, there is by definition an arrow $(B,C) \rightarrow (B',C')$ in $\Hasse(\tex\Lambda)$ if and only if $\varphi(B', C') = (B,C)$.
   By \eqref{eq:lMutRankTwo} we have that $\varphi(B', C') = (\ast, B') = (\ast , S^e_{i+1})$, since $C'$ is projective. 
   
\smallskip

   \noindent {\bf Case III(b):} 
   Assume now $S_i^e(=C')$ is not projective.
    We first discuss the case where $C = S_{i+1}^e$. (Note that $C \neq S_i^e$ since $C\neq C'$ by assumption.) Then $C_0 = P_i$ by Lemma~\ref{lem:simples_rank_2}, and so $(B_0,C_0) \neq (B',C')$. Thus there is no arrow $(B_0,C_0) \rightarrow (B'_0,C'_0)$ in $\Hasse(\tex\Lambda)$. Furthermore, we also have $\rho(B,C) = S_{i+1}^e$ and $\rho(B',C') = S_i^e \neq P_{i+1}\amalg R_{i+1}$, and so there is no arrow $\rho(B,C) \rightarrow \rho(B',C')$ in~$Q$.

   We next discuss the case with $C \notin \{S_1^e,S_2^e\}$. Then $\rho(B,C) = \upsilon\circ \tfo(B_0,C_0)$, by the definition of the map $\rho$.
   Now let $U_i$ denote the Bongartz complement of $S_i^e$, and note that there is an arrow $\rho(B,C) = \upsilon \circ \tfo(B_0,C_0) \rightarrow \rho(B',C') = S_i^e$ in $Q$ if and only if $\rho(B,C) = U_i \amalg S_i^e$. Now since $S_i^e$ is not projective, we have that $(U_i,S_i^e)$ is the unique TF-ordering of $U_i\amalg S_i^e$ by Lemma~\ref{lem:rank2_one_order}. Thus $\upsilon\circ\tfo(B_0,C_0) = U_i \amalg S_i^e$ if and only if $C_0 = S_i^e (= C')$, or equivalently (by Theorem~\ref{thm:unique}) if and only if $(B_0,C_0) = (B',C')$. Now note that, by construction, there is an arrow $(B_0,C_0) \rightarrow (B'_0,C'_0)$ in $\Hasse(\tex\Lambda)$ if and only if $(B_0,C_0) = \varphi(B'_0,C'_0)$; i.e., if and only if $(B_0,C_0) = (B',C')$. The result follows also in this case, and this concludes the proof of the theorem.
\end{proof}

We have the following direct consequence.
\begin{corollary}[Theorem~\ref{cor:connected_intro}]\label{cor:connected}
    $\Hasse(\stt\Lambda)$ and $\Hasse(\tex\Lambda)$ have the same number of connected components.
\end{corollary}

\subsection{Connection with brick labelling}\label{sec:brick_label}

We momentarily return to the case where $\Lambda$ has arbitrary rank. Let $M \in \mods\Lambda$ be an indecomposable module, and let $\mathrm{rad}(M,M)$ denote the set of non-invertible endomorphisms of $M$. We then denote
$$\beta(M) = M/\sum_{f \in \mathrm{rad}(M,M)} \mathrm{im}(f) = M/\mathrm{rad}_{\mathrm{End}(M)}M.$$
Now recall that a module $X$ is called a \emph{brick} if every nonzero endomorphism of $X$ is invertible, and is called an \emph{f-brick} if in addition the torsion class $\FiltGen(X)$ is functorially finite. The \emph{brick-$\tau$-rigid correspondence} \cite[Thm. 1.5]{dij} then says that the association $M \mapsto \beta(M)$ induces a bijection from the set of indecomposable $\tau$-rigid modules to the set of f-bricks. We further note that, for $M$ an arbitrary module and $\W$ an arbitrary wide subcategory, if $M \in \W$, then $\beta(M) \in \W$ as well. Thus for $\W$ a $\tau$-perpendicular subcategory, the same map $\beta$ induces a bijection from the set of indecomposable objects which are $\tau_\W$-rigid to the set of bricks which are f-bricks in $\W$. (Note that an f-brick in $\W$ is also a brick in $\mods\Lambda$, but it is not necessarily an f-brick in $\mods\Lambda$.) Note, however, that while the definition of $\beta$ does not depend on the wide subcategory in which we are working, the inverse of the brick-$\tau$-rigid correspondence generally does. See~\cite[Rk.~3.19]{bh2} for an example.

Let $(M_s,\ldots,M_n)$ be a $\tau$-exceptional sequence. Following~\cite{ht}, we say that $(\beta(M_s),\ldots,\beta(M_n))$ is a \emph{brick-$\tau$-exceptional sequence}. It is clear that $\tau$-exceptional sequences (of length $n-s+1$) are in bijection with brick-$\tau$-exceptional sequences (of length $n-s+1$).

\begin{definition}\cite[Def.~2.14 (specialised to rank two)]{asai}
    Let $M$ and $N$ be support $\tau$-tilting modules. Suppose that $N$ is a left mutation of $M$ (so in particular $\Gen(N) \subsetneq \Gen(M)$). Then this mutation corresponds to an arrow in $\Hasse(\stt\Lambda)$ labelled by $\beta(f_N M)$.
\end{definition}

Let us now restrict again to the case where $\Lambda$ has rank two. As mentioned previously, every support $\tau$-tilting module $M$ with $M \notin \{0, P(\Lambda)\}$ has exactly one left mutation $\mu^+(M)$ and exactly one right mutation $\mu^-(M)$. We denote the bricks labelling these mutations by $\beta^+(M)$ and $\beta^-(M)$, respectively. Note that if $\Gen M = \Gen P_i$, then $\beta^+(M) = S_{i+1}$. We now recall the following.

\begin{proposition}\cite[Prop. 6.7]{ap} \cite[Prop. 8.4]{bh2}\label{prop:brick_tau_seq}
    Let $M \in \mods\Lambda$ be a support $\tau$-tilting module with $M \notin \{0,P(\Lambda)\}$. Then $\beta^+(M)$ is the unique brick in $\J(P_{s}(\Gen M))$.
\end{proposition}

Combining this with Theorem~\ref{thm:rank_2_hasse}, we establish the following relationship between the bijection $\rho$ and the brick labelling.

\begin{corollary}\label{cor:brick_label}
    Let $M$ be a vertex of $Q$ (so $M$ is a support $\tau$-tilting module and $M \notin \{0,P(\Lambda)\}$). Then
    $\beta\circ\rho^{-1}(M) = (\beta^+(M),\beta^-(M))$.
\end{corollary}

\begin{proof}
    Suppose first that $M = S_i^e$. Then, by Theorem~\ref{thm:rank_2_hasse}, we have that $\rho^{-1}(M) = (B,S_i^e)$ for $B$ the unique module which is $\tau_{\J(S_i^e)}$-rigid. In particular, $\beta(B)$ is the unique brick in $\J(S_i^e)$, and so $\beta(B) = \beta^+(S_i^e)$ by Proposition~\ref{prop:brick_tau_seq}. Finally, note that $\beta(S_i^e) = S_i$ by the construction of $S_i^e$. This proves the result in this case.

    Now suppose that $M \notin \{S_1^e,S_2^e\}$. It follows that $M$ has two indecomposable direct summands. Moreover, we have that $M \neq P(\Lambda)$ since $M$ is a vertex of $Q$. Thus, by Lemma~\ref{lem:rank2_one_order}, the $\tau$-tilting module $M$ admits a unique TF-ordering $(M_1,M_2)$. This ordering thus satisfies $M_2 = \Pmodns(\Gen M)$ and $M_1 = \Pmods(\Gen M)$. In particular, $M_2 \in \Gen(M_1)$. Thus, by Theorem~\ref{thm:rank_2_hasse} and Definition-Proposition~\ref{defprop:forward}, we have that $\rho^{-1}(M) = \varphi(\Ebm_{M_2}(M_1),M_2) = (\Ebm_{M_1}(M_2)[-1],M_1)$. (Note that $M_2 \notin \P(\Lambda)$ since $M_2 \in \Gen(M_1)$.) It remains to show that $\beta(M_1) = \beta^-(M)$ and that $\beta(\Ebm_{M_1}(M_2)[-1]) = \beta^+(M)$.
    
    We first show that $\beta(M_1) = \beta^-(M)$. Let $N = \mu^-(M)$. Since $M_2 = \Pmodns(\Gen M)$, it follows that $M_2$ is a direct summand of $N$, and in fact that $M_2 = \Pmods(\Gen N)$. Since $M_2 \in \Gen(M_1)$, this implies that $\beta(M_1) = \beta(f_{N} M_1)$; i.e., that $\beta(M_1) = \beta^-(M)$.

    It remains to show that $\beta(\Ebm_{M_1}(M_2)[-1]) = \beta^+(M)$. We do this using a uniqueness argument. Indeed, since $M_1 = \Pmods(\Gen M)$, Proposition~\ref{prop:brick_tau_seq} says that $\beta^+(M)$ is the unique brick in $\J(M_1)$. At the same time, we have that $\beta(\Ebm_{M_1}(M_2)[-1])$ is also a brick in $\J(M_1)$ by the fact that the map $\Ebm_{M_1}$ is a bijection.
\end{proof}


\section{Examples}\label{sec:examples}

In this section, we provide four detailed examples. The first is an algebra of rank two where mutation of $\tau$-exceptional sequences is not transitive, the second is an algebra which contains a left immutable $\tau$-exceptional sequence, the third is a non-hereditary representation-finite algebra for which mutation of $\tau$-exceptional sequences is transitive but does not satisfy the braid relations, and the fourth gives an example of irregular mutation. Note that the algebras appearing in Examples~\ref{ex:immutable}, \ref{ex:aslak}, and \ref{ex:preproj} are all examples of gentle algebras~\cite{assk}.

\begin{example}\label{ex:kronecker}
    Let $\Lambda = KQ/I_\mathrm{cyc}$ for
   $$Q=\begin{tikzcd}
        1 \arrow[r,bend left]\arrow[r,bend left,yshift = 0.5em]&2\arrow[l,bend left]\arrow[l,bend left,yshift=-0.5em]
    \end{tikzcd}$$
    and $I_\mathrm{cyc}$ the ideal generated by all $2$-cycles. We also let $Q_L$ (resp. $Q_R$) be the Kronecker quiver obtained by deleting the two right-pointing (resp. left-pointing) arrows of $Q$. It is proved in \cite[Thm.~5.1]{terland} that the quiver $\Hasse(\stt \Lambda)$ has two connected components. Thus, by Corollary~\ref{cor:connected}, mutation of $\tau$-exceptional sequences is not a transitive operation for this algebra. We can also verify this more directly. 
    
    The proof of \cite[Thm.~5.1]{terland} establishes that the $\tau$-tilting theory of $\Lambda$ can be described explicitly as a symmetric version of the cluster tilting theory of the Kronecker quiver. More specifically, the only $\tau$-rigid module supported on all 4 arrows is the free module $\Lambda$. Every other $\tau$-rigid module (including the two indecomposable projectives) is a representation of either $Q_L$ or $Q_R$. This means that the $f$-bricks and indecomposable $\tau$-rigid modules coincide. Furthermore, except at the free module $\Lambda$, the right mutation of support $\tau$-tilting modules is inherited from the cluster theories of $Q_L$ and $Q_R$. Explicitly, for $i \in \mathbb{N}$, we denote by $M_i^L$ (resp. $M_i^R$) the preprojective representation of $Q_L$ (resp. preinjective representation of $Q_R$) with dimension vector $(i+1,i)$. Note that $M_1^L = P_2$ and $M_1^R = I_2$ (as modules over $\Lambda$). We then obtain a mutation orbit of $\tau$-exceptional sequences over $\Lambda$ as follows:
    $$ \cdots \xrightarrow{\varphi} (M_2^L,M_3^L) \xrightarrow{\varphi} (M_1^L,M_2^L) \xrightarrow{\varphi} (S_1,M_1^L) \xrightarrow{\varphi} (M_1^R,S_1) \xrightarrow{\varphi} (M_2^R,M_1^R)\xrightarrow{\varphi} (M_3^R,M_2^R) \xrightarrow{\varphi}\cdots.$$
    The other orbit is given by the symmetric construction.
\end{example}

\begin{example}\label{ex:immutable}
    Let $\Lambda = KQ/(ac)$ for
    $$\begin{tikzcd}Q= & 1\arrow[r,yshift=0.25em,"a"above]\arrow[r,yshift=-0.25em,"b"below] & 2\arrow[r,"c" above] & 3.\end{tikzcd}$$
     In \cite[Ex.~3.13]{asai}, Asai showed that the wide subcategory $\add I_3$ is right finite (and is thus a $\tau$-perpendicular subcategory of corank two by Theorem~\ref{thm:WLtorsionclass}). We claim that $(I_3,S_1,S_2)$ is a complete $\tau$-exceptional sequence which is not left $2$-mutable. Equivalently, $(S_1,S_2)$ is a left immutable $\tau$-exceptional pair with $\J(S_1,S_2) = \add I_3$. 

    The fact that $(I_3,S_1,S_2)$ is a $\tau$-exceptional sequence follows from Proposition~\ref{prop:ASExt} and direct computation of the relevant Hom- and Ext-spaces. From this, it follows that $(I_3)$ is a complete $\tau_{\J(S_1,S_2)}$-exceptional sequence. But $(I_3)$ is also a complete $\tau_{\add I_3}$-exceptional sequence, and so $\J(S_1,S_2) = \add I_3$ by \cite[Cor.~10]{ht}.

    We now show that $(S_1,S_2)$ is not left mutable. We already know that $\J(S_1,S_2) = \add I_3$ is not left finite, so we need only show that $(S_1,S_2)$ is left irregular. The torsion class $\FiltGen(S_1\amalg S_2) = \Filt(S_1\amalg S_2)$ is the Serre subcategory of those modules supported only at the vertices $1$ and $2$, and we compute $\Ps(\Filt(S_2\amalg S_1)) = \add(S_2,P_1/P_3)$. By Theorem~\ref{thm:F_map_1}, this yields $\Fbm_{S_2}(S_1) = P_1/P_3 \notin \P(\Lambda)$. Now clearly $S_2 \notin \P(\Lambda) \cup \Gen P_1/P_3$. So, we need to show that $S_2 \in \P({}^\perp \tau(P_1/P_3))$. By Lemma~\ref{lem:fUbijection} and the observation that $f_{(P_1/P_3)} S_2 = S_2$, it suffices to show that $S_2 \in \P(\J(P_1/P_3))$.

    Noting that $S_2\amalg (P_1/P_3)$ is $\tau$-rigid, we know that $S_2 \in \J(P_1/P_3)$, and is thus a simple object in $\J(P_1/P_3)$. Similarly, we have that $\J(S_1,S_2) \subseteq \J(P_1/P_3)$ by Theorem~\ref{thm:J_E}, and so $I_3 \in \J(P_1/P_3)$. The module $I_3$ has two proper quotients ($I_3/P_3$ and $S_1$), and neither of these lies in 
    $(P_1/P_3)^\perp \supseteq \J(P_1/P_3)$. 
    This means $I_3$ is also a simple object in $\J(P_1/P_3)$. Now recall that $\J(P_1/P_3)$ has rank two by Definition-Proposition~\ref{prop:J_def}, and so $S_2$ and $I_3$ are its only simple objects. The fact that $S_2 \in \P(\J(P_1/P_3))$ then follows from the fact that $\Ext^1(S_2,I_3) = 0$.

    We can also say more about the orbit (under $\psi = \psi_2$) containing the $\tau$-exceptional pair $(S_1,S_2)$. By Definition-Propositions~\ref{defprop:backward} and~\ref{defprop:backward2}, we have that $\add I_3 = \J(\psi^k(S_1,S_2))$ for all $k \in \mathbb{N}$. (Note that there is no right immutable $\tau$-exceptional pair in this orbit since $\add I_3$ is right finite.) Thus the orbit lives entirely in the Serre subcategory ${}^\perp I_3$, which can be identified with the 
    module category of the path algebra of the Kronecker quiver $Q'$ obtained by deleting the vertex 3 from $Q$. As with many of the $\tau$-exceptional sequences in Example~\ref{ex:kronecker}, the mutation of $\tau$-exceptional pairs in the orbit of $(S_1,S_2)$ is largely controlled by that of exceptional sequences over the Kronecker. Explicitly, for $i \in \mathbb{N}$, we denote by $M_i$ the preprojective representation of $Q'$ with dimension vector $(i,i+1)$ (viewed as a module over $\Lambda$). The orbit of $(S_1,S_2)$ can then be shown to be
    $$(S_1,S_2) \xrightarrow{\psi} (S_2,M_1) \xrightarrow{\psi} (M_1,M_2) \xrightarrow{\psi} \cdots$$
    Since $\Lambda$ is a gentle algebra, we know that every indecomposable $\tau$-rigid module, and also every f-brick, is a string module. We then verified, using \cite{geuenich}, that the only $\tau$-exceptional pairs $(X,Y)$ which satisfy $\J(X,Y) = \add I_3$ are those in the orbit above.
\end{example}

\begin{example}\label{ex:aslak}

Let $\Lambda = KQ/(ab)$ for 
$$\begin{tikzcd}
Q = 1 \ar[r, "a" above] & 2 \ar[r, "b" above] & 3.
\end{tikzcd}$$
Then $\mods \Lambda$ has five indecomposable modules, all of which are $\tau$-rigid bricks, with the AR-quiver:
$$
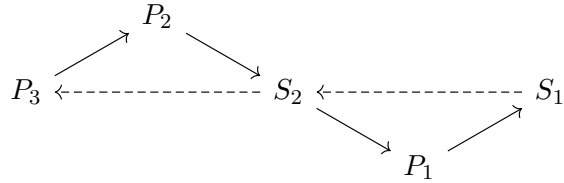
\begin{tikzcd}[column sep=1cm, row sep=0.4cm]
& P_2 \ar[dr] & & & \\
P_3 \ar[ur]& & S_2 \ar[ll, dashed] \ar[dr] & & S_1
\ar[ll, dashed] \\
& & & P_1 \ar[ur]& 
\end{tikzcd}$$

The complete (left-)mutation graph is given below, where the action of $\varphi_2$ is indicated by solid arrows, and
the action of $\varphi_1$ is indicated by dashed arrows.
For this algebra, it turns out that all complete $\tau$-exceptional sequences are left and right $i$-regular.

$$
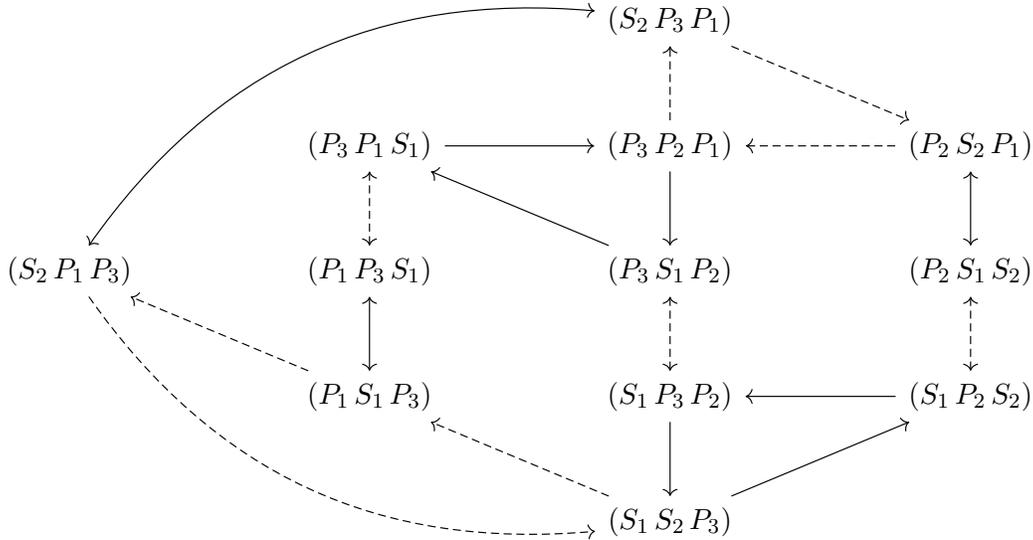
\begin{tikzcd}[column sep=2cm, row sep=1cm]
 & & (S_2\, P_3\, P_1) \arrow[rd, dashed] \arrow[lldd,bend right, leftrightarrow]& \\
 & (P_3\, P_1\, S_1) \arrow[d, dashed, leftrightarrow] \arrow[r] & (P_3\, P_2\, P_1) \arrow[u, dashed] \arrow[d]
  &(P_2\, S_2\, P_1) \arrow[l, dashed] \arrow[d, leftrightarrow]\\
(S_2\, P_1\, P_3) \arrow[rrdd, bend right,dashed] & (P_1\, P_3\, S_1) \arrow[d, leftrightarrow]& (P_3\, S_1\, P_2) \arrow[d, leftrightarrow, dashed] \arrow[ul] & (P_2\, S_1\, S_2) \ar[d, leftrightarrow, dashed] \\
 & (P_1\, S_1\, P_3) \arrow[ul,dashed]& (S_1\, P_3\, P_2) \ar[d] & (S_1\, P_2\, S_2) \arrow[l]\\
 & & (S_1\, S_2\, P_3) \arrow[ul,dashed] \arrow[ur]& 
\end{tikzcd}$$
Note in particular that we have
$$\varphi_1\circ \varphi_2\circ \varphi_1(S_1\, S_2\, P_3) = (P_3\, P_1\, S_1) \qquad \text{and} \qquad \varphi_2\circ \varphi_1\circ \varphi_2(S_1\, S_2\, P_3) = (P_2\, S_2\, P_1),$$
showing that the braid relations do not hold.
\end{example}

\begin{example}\label{ex:preproj}
    Let $\Lambda = KQ/I_\mathrm{cyc}$ for
    $$ \begin{tikzcd} Q = & 1\arrow[r,bend right] & 2\arrow[r,bend right]\arrow[l,bend right] & 3\arrow[l,bend right]\end{tikzcd}$$
    and $I_\mathrm{cyc}$ the ideal generated by all $2$-cycles. Then $\mods\Lambda$ has 11 indecomposable modules, all of which are $\tau$-rigid bricks, with the AR quiver: 

    $$
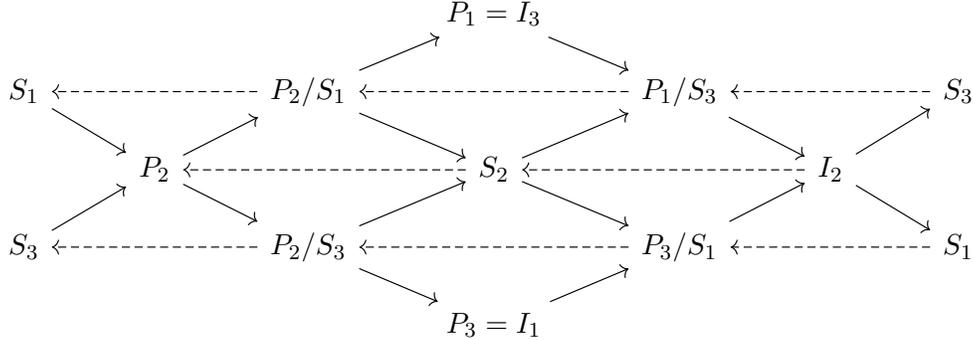
\begin{tikzcd}[column sep=1cm, row sep=0.4cm]
&&&P_1 = I_3\arrow[dr]\\
S_1\arrow[dr]&&P_2/S_1\arrow[ll,dashed]\arrow[ur]\arrow[dr]&&P_1/S_3\arrow[ll,dashed]\arrow[dr]&&S_3\arrow[ll,dashed]\\
&P_2\arrow[ur]\arrow[dr]&&S_2\arrow[ll,dashed]\arrow[ur]\arrow[dr]&&I_2\arrow[ll,dashed]\arrow[ur]\arrow[dr]\\
S_3\arrow[ur]&&P_2/S_3\arrow[ll,dashed]\arrow[ur]\arrow[dr]&&P_3/S_1\arrow[ll,dashed]\arrow[ur]&&S_1\arrow[ll,dashed]\\
&&&P_3 = I_1\arrow[ur]
\end{tikzcd}$$
There are three $\tau$-exceptional sequences of the form $(P_3,-,-)$ which fit into the following $\psi_2$-orbit:
$$\begin{tikzcd}[column sep=1cm, row sep=0.4cm]
    (P_3,S_2,P_3/S_1)\arrow[dr] &&(P_3,S_3,P_2/S_1)\arrow[ll]\\
    &(P_3,P_3/S_1,S_3)\arrow[ur]
\end{tikzcd}$$
Indeed, since $P_3 = I_1$, the only $\tau$-exceptional sequences of this form must have their second and third terms supported only on the vertices $2$ and $3$. Combined with the fact that $\Hom(P_3,\tau S_2) \neq 0$, this gives three choices for the rightmost term of the sequence: $P_3/S_1$, $S_3$, and $P_2/S_1$. Theorem~\ref{thm:unique} then implies that knowing the first and third terms of a complete $\tau$-exceptional sequence is enough to uniquely determine the middle term.

The $\tau$-exceptional sequence $(P_3,P_3/S_1,S_3)$ is right $2$-irregular. Indeed, $\J(S_3)$ contains both $P_3$ and $P_3/S_1$, which means $P_3/S_1 \notin \P(\J(S_3))$. Equivalently (by Lemma~\ref{lem:fUbijection}), we have $\Fbm_{S_3}(P_3/S_1) \notin \P({}^\perp \tau S_3)$. We also have that $\Hom(S_3,\tau(P_3/S_1)) = 0$, which tells us that $\Fbm_{S_3}(P_3/S_1) = P_3/S_1$. Since $S_3 \in \Gen(P_3/S_1)$, we conclude that $(P_3,P_3/S_1,S_3)$ is right $2$-irregular.

To demonstrate how to mutate right irregular pairs, we explicitly compute $\psi_2(P_3,P_3/S_1,S_3)$ using Definition-Proposition~\ref{defprop:backward2}. In the notation of Definition-Proposition~\ref{defprop:backward2}, we have already computed $\widetilde{L} = (P_3/S_1) \amalg S_3$. We then compute $\J(\widetilde{L}) = \add P_3$. Since $P_3 = I_1$, this yields $$\Pmods({}^\perp \J(\widetilde{L})) = \Pmods({}^\perp I_1) = (P_2/S_1)\amalg (P_3/S_1).$$
From the AR quiver, we obtain
$$\Cogen(\tau \Pmods({}^\perp \J(\widetilde{L}))) = \add(P_2/S_3, S_1).$$
Then $P_3/S_3$ is the split injective in this torsion-free class, and $\tau^{-1}(P_2/S_3) = P_3/S_1$. We conclude that
$$\psi_2(P_3,P_3/S_1,S_3) = (P_3,\Ebm_{P_2/S_1}(P_3/S_1),P_2/S_1) = (P_3,S_3,P_2/S_1).$$
It is possible to verify directly that this $\tau$-exceptional sequence is left $2$-irregular, as is implied by Definition-Proposition~\ref{defprop:backward2}.
\end{example}

\noindent \textbf{Acknowledgements}:
All authors contributed to the conceptualisation, methodology, project administration, development and writing of this paper.

\noindent \textbf{Data access statement}:
No data are associated with this paper.

\noindent \textbf{Rights retention statement}:
For the purpose of open access, the authors has applied a Creative Commons Attribution (CC BY) licence to any Author Accepted Manuscript version arising from this submission.


\begin{thebibliography}{99}
		
		\bibitem{air}
		T.~Adachi, O.~Iyama and I.~Reiten,
		\emph{$\tau$-tilting theory},
		Compos. Math. 150 (2014), no. 3, 415--452.
		
		\bibitem{asai}
		S.~Asai, \emph{Semibricks}, Int. Math. Res. Not. IMRN 2020 (2018), no. 16, 4993--5054.

 \bibitem{ap}
	S.~Asai and C.~Pfeifer, \emph{Wide subcategories and lattices of torsion classes},
	Algebr. Represent. Theory 25 (2022), 1611--1629.

\bibitem{assk}
I.~Assem and A.~Skowro\'{n}ski,
\emph{Iterated tilted algebras of type $\tilde{A}_n$},
Math. Z. 195, no. 2, 269--290 (1987).
		
		\bibitem{ASS06}
		I.~Assem, D.~Simson and A.~Skowroński, 
		Elements of the representation theory of associative algebras. Vol. 1.
		Techniques of representation theory,
		London Math. Soc. Stud. Texts 65,
		Cambridge University Press, Cambridge, 2006.
		
		\bibitem{as}
		M.~Auslander and S.~O.~Smal\o,
		{\it Almost split sequences in subcategories},
		J. Algebra 69 (1981), no. 2, 426--454.
		
		\bibitem{bh2}
		E.~Barnard and E.~J.~Hanson, \emph{Exceptional sequences in semidistributive lattices and the poset topology of wide subcategories}, J. Algebra Appl. 25 (2026), no. 4, 2550360, 51 pp.


  \bibitem{bon}
A.~Bondal, \emph{Representations of associative algebras and coherent sheaves}, Math. USSR, Izv. 34 (1990), 23--42.

\bibitem{borve}
    E.~D.~B{\o}rve, \emph{Two-term silting and $\tau$-cluster morphism categories}, arXiv:2110.03472.

\bibitem{bst}
T.~Br\"{u}stle, D.~Smith and H.~Treffinger,
\emph{Wall and chamber structure for finite-dimensional algebras},
Adv. Math. 354 (2019), 106746, 31 pp.
  
		\bibitem{bh}
		A.~B.~Buan and E.~J.~Hanson,
		\emph{$\tau$-perpendicular wide subcategories}, Nagoya Math. J. 252 (2023), 959--984.

  \bibitem{bm2}
	A.~B.~Buan and B.~R.~Marsh, {\it A category of wide subcategories}, Int. Math. Res. Not. IMRN 2021 (2019), no. 13, 10278--10338.
		
		\bibitem{bm}
		A.~B.~Buan and B.~R.~Marsh,
		{\it $\tau$-exceptional sequences},
		J. Algebra 585 (2021), 36--68.
		
		\bibitem{bmpreprint}
		A.~B.~Buan and B.~R.~Marsh,
		\emph{Mutating signed $\tau$-exceptional sequences}, Glasg. Math. J. 65 (2023), no. 3, 716--729.

\bibitem{bmrrt}
A.~B.~Buan, B.~R.~Marsh, I.~Reiten, M.~Reineke and G.~Todorov, \emph{Tilting theory and cluster combinatorics}, Adv. Math. 204 (2006), no. 2, 572--618.
  
		\bibitem{brt}
		A.~B.~Buan, I.~Reiten, and H.~Thomas, \emph{Three kinds of mutation}, J. Algebra 339 (2011), no. 1, 97--113.

\bibitem{cfy}

F.~Campanini, F.~Fedele and E.~Y{\i}ld{\i}r{\i}m,
\emph{Lattices of pretorsion classes},
arXiv:2511.19223.

\bibitem{cb}
W.~Crawley-Boevey, \emph{Exceptional sequences of representations of quivers}, Representations of algebras (Ottawa, ON, 1992), CMS Conf. Proc., vol. 14, Amer. Math. Soc., Providence, RI, 1993, pp. 117--124.

\bibitem{dij} L.~Demonet, O.~Iyama, and G.~Jasso, 
\emph{$\tau$-tilting finite algebras, bricks, and 
$g$-vectors}, 
Int. Math. Res. Not. IMRN 2019 (2017), no. 3, 852--892.

\bibitem{dirrt} 
L.~Demonet, O.~Iyama, N.~Reading, I.~Reiten and H.~Thomas, \emph{Lattice theory of torsion classes: Beyond $\tau$-tilting theory}, 
Trans. Amer. Math. Soc. Ser. B 10 (2023), 542--612. 

\bibitem{fz}
S.~Fomin and A.~Zelevinsky, \emph{Cluster algebras I. Foundations}, J. Amer. Math. Soc. 15 (2002), no. 2, 497--529.

\bibitem{gl}
W.~Geigle and H.~Lenzing,
\emph{Perpendicular categories with applications to representations and sheaves},
J. Algebra 144 (1991), no. 2, 273--343.

\bibitem{geuenich}
J.~Geuenich,
\emph{String-applet},
\url{https://www.math.uni-bielefeld.de/~jgeuenich/string-applet}, accessed 2023.

\bibitem{goru}
 A.~L.~Gorodentsev and A.~N.~Rudakov, \emph{Exceptional vector bundles on projective spaces},
 Duke Math. J. 54 (1987), no. 1, 115--130.

\bibitem{happel}
 D.~Happel, \emph{Triangulated categories in the representation of finite dimensional algebras.} London Mathematical Society Lecture Notes Series 119, Cambridge University Press, Cambridge, 1988.

\bibitem{h}
		E.~J.~Hanson, {\it $\tau$-exceptional sequences and the shard intersection order in type $A$}, Electron. J. Combin. 33 (2026), no. 2, P2.59, 54 pp. 
  
		\bibitem{hi_2}
		E.~J.~Hanson and K.~Igusa, \emph{Pairwise compatibility for $2$-simple minded collections}, J. Pure Appl. Algebra 225 (2021),
        no. 6, 106598.

    \bibitem{hi}
  E.~J.~Hanson and K.~Igusa, \emph{$\tau$-cluster morphism categories and picture groups}, Comm. Algebra 49 (2021), no. 10, 4376--4415.

		\bibitem{ht}
		E.~J.~Hanson and H.~Thomas, \emph{A uniqueness property of $\tau$-exceptional sequences}, 
        Algebr. Represent. Theory 27 (2024), 461--468.

  \bibitem{igto} 
K.~Igusa and G.~Todorov, \emph{Signed exceptional sequences and the cluster morphism category}, arXiv:1706.02041.

  \bibitem{itw}
  K.~Igusa, G.~Todorov and J.~Weyman, \emph{Picture groups of finite type and cohomology in type $A_n$}, arXiv:1609.02636.

  \bibitem{IngTho} C.~Ingalls and H.~Thomas, \emph{Noncrossing partitions and representations of quivers}, Compos. Math. 145 (2009),
no. 6, 1533–1562.
		
		\bibitem{jasso}
		G.~Jasso,
		\emph{Reduction of $\tau$-tilting modules and torsion pairs},
		Int. Math. Res. Not. IMRN 2015 (2014), no. 16, 7190--7237.

  \bibitem{kaipel}
  M.~Kaipel, \emph{The category of a partitioned fan}, J. Lond. Math. Soc. 111 (2025), no. 2, e70071, 41 pp.

\bibitem{kra}
H.~Krause, \emph{Highest weight categories and recollements}, Ann. Inst. Fourier (Grenoble) vol. 67, no. 6 (2017), 2679--2701.
  
		\bibitem{ms}
		F.~Marks and J.~\v{S}\v{t}ov\'{\i}\v{c}ek,
		\emph{Torsion classes, wide subcategories and localisations},
		Bull. London Math. Soc. 49 (2017), no. 3, 405--416.
		
		\bibitem{mt}
		H.~O.~Mendoza and H.~Treffinger,
		\emph{Stratifying systems through $\tau$-tilting theory},
		Doc. Math. 25 (2020), 701--720.

\bibitem{rin}
C.~M.~Ringel, \emph{The braid group action on the set of exceptional sequences of a hereditary Artin algebra}, Abelian Group Theory and Related Topics (Oberwolfach, 1993), Contemp. Math., vol. 171, Amer. Math. Soc., Providence, RI (1994) pp. 339--352. 

\bibitem{rud}
A.~N.~Rudakov, \emph{Exceptional collections, mutations and helices},
London Math. Soc. Lecture Note Ser., 148,
Cambridge University Press, Cambridge, 1990, 1--6.

\bibitem{sch}
A.~Schofield, \emph{Semi-invariants of quivers}, 
Journal of the London Mathematical Society, s2-43 (1991), 385--395.

\bibitem{sttw}
S.~Schroll, A.~Tattar, H.~Treffinger, and N.~J.~Williams, \emph{A geometric perspective on the $\tau$-cluster morphism category}, 
 Math. Z. 312 (2026), 77, 20 pp.
  
\bibitem{smalo84}
S.~O.~Smal\o, \emph{Torsion theories and tilting modules}, Bull. London Math. Soc. 16 (1984), no. 5, 518--522.

\bibitem{terland}
    H. U. Terland, \emph{Computing $\tau$-rigid modules}. Master's Thesis, Norwegian University of Science and Technology (NTNU), June 2021. Available at \url{https://ntnuopen.ntnu.no/ntnu-xmlui/handle/11250/2778393}.


\end{thebibliography}
\end{document}